\documentclass[11pt]{amsart}
\usepackage{amsfonts}

\usepackage{amsmath,amssymb,amsthm}\allowdisplaybreaks[4]
\usepackage{url}
\usepackage{xcolor} % color package
\usepackage{graphicx}   % Optional: allows for the easy inclusion of pictures.
\usepackage{verbatim}       %Package used for the comment environment%
\usepackage{layout}
\usepackage{mathrsfs}
\usepackage{tikz}

\usepackage[margin=1in]{geometry}
\usepackage{setspace}
\usepackage{enumerate}
\xdefinecolor{dred}{rgb}{0.6,0,0}

\newtheorem{theorem}{Theorem}[section]
\newtheorem{lemma}[theorem]{Lemma}
\newtheorem{proposition}[theorem]{Proposition}
\newtheorem{remark}[theorem]{Remark}

\newtheorem{corollary}[theorem]{Corollary}

\newcommand{\ind}{1{\hskip -2.5 pt}\hbox{I}}

\newfam\msbmfam\font\tenmsbm=msbm10\textfont
\msbmfam=\tenmsbm\font\sevenmsbm=msbm7
\scriptfont\msbmfam=\sevenmsbm

\def\EE{\mathbb{E}}
\def\NN{\bb N}\def\PP{\bb P}
\def\RR{\mathbb{R}}

\def\bN{\mathbb{N}}

\def\bZ{\mathbb{Z}}

\def\cL{\mathcal{L}}

\def\<{\left<}\def\>{\right>}
\def\({\left(}\def\){\right)}

\def\blemma{\begin{lemma}}\def\elemma{\end{lemma}}
 \def\bproposition{\begin{prop}}\def\eproposition{\end{prop}}
 \def\btheorem{\begin{theorem}}\def\etheorem{\end{theorem}}
 \def\bcorollary{\begin{corollary}}\def\ecorollary{\end{corollary}}

\def\beqlb{\begin{eqnarray}}\def\eeqlb{\end{eqnarray}}
 \def\beqnn{\begin{eqnarray*}}\def\eeqnn{\end{eqnarray*}}

\graphicspath{%
    {converted_graphics/}% inserted by PCTeX
    {/}% inserted by PCTeX
}
\def\EE{\mathbb{E}}
\def\NN{\mathbb N}\def\PP{\mathbb P}
\def\RR{\mathbb{R}}

\def\bN{\mathbb{N}}

\begin{document}
\title{Coalescent results for diploid exchangeable population models}
\date{\today}
\author{Matthias Birkner, Huili Liu and Anja Sturm}
\thanks{We would like to thank two anonymous referees for their careful reading and useful comments
which helped to improve the presentation. \\
This project received financial support
from the German Research Foundation (DFG), within
the Priority Programme 1590 ``Probabilistic Structures in Evolution''
through Grants BI\ 1058/2-1,  BI\ 1058/2-2 and STU\ 527/1-1,  STU\ 527/1-2, Liu also acknowledges research support by National Natural Science Foundation of China NSFC 11501164}
\address{Matthias Birkner: Institut f\"ur Mathematik, Johannes-Gutenberg-Universit\"at Mainz, 55099 Mainz, Germany}
\email{birkner@mathematik.uni-mainz.de}
\address{Huili Liu: College of Mathematics and Information Science, Hebei Normal University, Shijiazhuang, Hebei Province, 050024, China}
\email{l\_huili@live.concordia.ca}
\address{Anja Sturm: Institut f\"ur Mathematische Stochastik, Georg-August-Universit\"at G\"ottingen, 37077 G\"ottingen, Germany}
\email{asturm@math.uni-goettingen.de}

\date{\today}

\maketitle
\pagestyle{myheadings} \markboth{\textsc{Coalescent results for diploid exchangeable population models}}{\textsc{M.~Birkner, H.~Liu and A.~Sturm}}
\begin{abstract}
We consider diploid bi-parental analogues of Cannings models: in a
population of fixed size $N$ the next generation is composed of
$V_{i,j}$ offspring from parents $i$ and $j$, where $V=(V_{i,j})_{1\le i
  \neq j \le N}$ is a (jointly) exchangeable (symmetric) array. Every
individual carries two chromosome copies, each of which is inherited
from one of its parents. We obtain general conditions, formulated in
terms of the vector of the total number of offspring to each
individual, for the convergence of the properly scaled ancestral
process for an $n$-sample of genes towards a ($\Xi$-)coalescent.
This complements M\"ohle and Sagitov's (2001) %\cite{Mohle2001}
result for the haploid case and sharpens the profile of
M\"ohle and Sagitov's (2003) %\cite{MohleSagitov2003}
study of the diploid case, which focused on fixed couples, where each
row of $V$ has at most one non-zero entry.

We apply the convergence result to several examples, in particular to two
diploid variations of Schweinsberg's (2003) %\cite{Schweinsberg2003}
model, leading to
Beta-coalescents with two-fold and with four-fold mergers, respectively.
\end{abstract}

\section{Introduction and main results}
\label{sec:IntroResult}
For haploid population models, in which every individual (gene) has one parent (gene), coalescent processes have been used widely in order to describe the ancestral structure of a sample of $n$ genes when the total population size $N$ is sufficiently large. The purpose of this work is to extend the coalescent theory to general  diploid population models, in which individuals carry two copies of each gene which they inherit from two distinct parental individuals. In this context, we derive the diploid analogue of M\"ohle and Sagitov's classification of the ancestral
processes in exchangeable haploid population models \cite{Mohle2001}. This gives a unified picture of studying genealogies in an exchangeable diploid setting, which has up to now only been available in special cases.

We consider a general diploid exchangeable population model with fixed
constant population size $N\in\mathbb{N}:=\left\{1,2,\ldots\right\}$
and non-overlapping generations
$m\in\mathbb{N}_{0}:=\left\{0,1,2,\ldots\right\}$ without explicit
sexes (see however our remarks in Section \ref{sect:extensions}). The generations are labeled backwards in time. That is, $m=0$ is the current generation; $m=1$ is one generation backwards in time and so on. Each individual possesses two chromosome copies, each inherited
from one of its two parents. Which parental chromosome is
inherited is a uniform random pick, independently for each child.
For $m\in\mathbb{N}$, let $V^{\(m\)}_{i,j}$ be the number of
children by individuals $i$ and $j$ (for $i < j$) in the $m$-th generation.
We call these quantities {\em pairwise offspring numbers} and  implicitly define $V^{(m)}_{k,j} = V^{(m)}_{j,k}$ for $k>j$
when notationally necessary.  We exclude the possibility of
self-fertilisation, i.e., $V^{(m)}_{i,i} = 0$.  We assume that the
reproduction law is independent and identically distributed from
generation to generation, i.e., the matrices
$\big(V^{(m)}_{i,j}\big)_{1 \le i < j \le N}$, $m \in \mathbb{N}$ are
i.i.d. We will often write $V_{i,j}=V_{i,j}^{\(1\)}$ for simplicity.
We have $\sum_{1 \le i < j \le N} V_{i,j} = N$ because the population size is fixed. Note that despite this dependence on $N$ of the law of $(V_{i,j})_{1 \le i < j \le N}$
we will suppress the $N$-dependence in the notation.
Our fundamental assumption is the following exchangeability condition:
\begin{equation}
  \label{eq:exchangeable}
  \(V_{i,j}\)_{1 \le i < j \le N} \mathop{=}^d \(V_{\sigma\(i\),\sigma\(j\)}\)_{1 \le i < j \le N}
  \quad \text{for any permutation $\sigma$ of $\{1,\dots,N\}$},
\end{equation}
%% We refer to Austin \cite{Austin} for the description of exchangeable arrays.
i.e., $\(V_{i,j}\)_{1 \le i < j \le N}$ is a finite jointly exchangeable array
(see also Remark~\ref{Remark4} in Section \ref{sec:Discussion}).
\medskip

\begin{center}
  \def\pgone{{2,2,1,4,6,6,7}} % parents for gene 1 in each ind.
  \def\pggone{{1,0,0,0,1,0,1}} % index of parental gene for gene 1 in each ind.
  \def\pgtwo{{1,1,3,3,7,4,6}} % parents for gene 2 in each ind.
  \def\pggtwo{{1,1,0,0,1,1,0}} % index of parental gene for gene 2 in each ind.
  \begin{tikzpicture}
    \foreach \y in {0,1}
    \foreach \x in {1,2,3,4,5,6,7} {
      \filldraw[black] (\x -0.1,\y) circle[radius=0.05]
                       (\x +0.1,\y) circle[radius=0.05];
      \draw (\x,\y) ellipse[x radius=0.35, y radius=0.25];
    };

    \foreach \i in {0,1,2,3,4,5,6} {
      \pgfmathsetmacro{\j}{\pgone[\i]+0.1*(2*\pggone[\i]-1)}
      \draw (\i+1-0.1,0) -- (\j,1) ;
      \pgfmathsetmacro{\j}{\pgtwo[\i]+0.1*(2*\pggtwo[\i]-1)}
      \draw (\i+1+0.1,0) -- (\j,1) ;
    }

    \draw (0.2,0) node[left] {children} ;
    \draw (0.2,1) node[left] {parents} ;
  \end{tikzpicture}
  \medskip

  \begin{minipage}{0.75\textwidth}
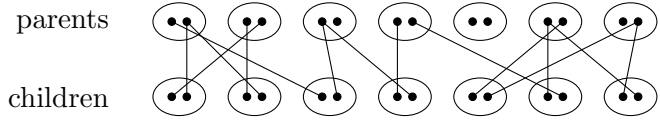
    {\bf Figure 1\ } An example for the assignment of parental genes in a population of size $N=7$, each
    individual has two gene copies (the filled circles).
  \end{minipage}
\end{center}
\medskip

Generally, we are interested in tracking the genealogy of a sample of
$n \in \mathbb{V}_N:=\{1,2,\ldots, N\}$ genes from the present population of
size $N$. We will follow the customary approach of describing
ancestral relations among $n$ sampled genes by partitions
characterising which genes are descended from the same parental gene. Unless specified otherwise, asymptotic relations refer to letting $N\rightarrow\infty$ throughout the paper.

Let $\mathcal{E}_n$ be the collection of partitions of
$\mathbb{V}_n$ and $\mathcal{E}_{\infty}$ be the collection of partitions of $\mathbb{N}$. Any element in $\mathcal{E}_n$ can be expressed by
$\xi=\{C_1,C_2,\ldots,C_b\}$ where $C_i\cap C_j=\emptyset$ for $i\neq
j$ and $\cup_{i=1}^bC_i=\mathbb{V}_n$ with $b=\left|\xi\right|$
the number of  partition elements
in $\xi$.  When it is necessary to make the representation unique we
order the $C_i, i=1, \dots, b$ by their smallest element in ascending
order. In the following we will also refer to the partition elements as blocks.
For any $\xi$, $\eta\in\mathcal{E}_n$, write $\xi\subseteq\eta$ if and
only if every block
of $\eta$ is a union of (one or more) blocks
of $\xi$.

In order to also specify which ancestral genes belong to the same
ancestral individuals we use notation introduced by M\"ohle and Sagitov in \cite{MohleSagitov2003}
and consider the state space
 \begin{equation*}
% \label{S_n}
 \mathcal{S}_n= \left\{ \left\{\left\{C_1,C_2\right\},\ldots,\left\{C_{2x-1},C_{2x}\right\},C_{2x+1},\ldots,C_b\right\} :
 b \in \mathbb{V}_n, x \in \mathbb{V}_{\left\lfloor\frac{b}{2}\right\rfloor}, \{C_1, \ldots, C_b\} \in \mathcal{E}_n \right\},
  \end{equation*}
where $\lfloor x\rfloor$ is the largest integer less than or equal to $x$.
We equip the space $ \mathcal{E}_n$ as well as  $\mathcal{S}_n$ with the discrete topology.
  For later use we define a map $\mathsf{cd} :
  \mathcal{S}_n\rightarrow\mathcal{E}_n$ such that for
  any $$\xi=\left\{\left\{C_1,C_2\right\},\ldots,\left\{C_{2x-1},C_{2x}\right\},C_{2x+1},\ldots,C_b\right\}\in\mathcal{S}_n,$$
$$%\tilde{\xi}:=
\mathsf{cd}(\xi):=\left\{C_1,C_2,\ldots,C_{2x-1},C_{2x},C_{2x+1},\ldots,C_b\right\} \in \mathcal{E}_n.$$
Following \cite{BBE13}, we call $\mathsf{cd}(\xi)$ the {\it complete dispersion} of $\xi$.

Now sample $n$ genes randomly from the current generation. We can think of
sampling $n/2$ individuals and looking at both of their genes or
sampling $n$ individuals and inspecting only one randomly chosen gene in each
or something in-between, this will not matter in the limit we are interested
in.
For $m\in\mathbb{N}_0$, let $\xi^{n,N}\(m\)$
be the configuration of the genealogical structure for the sampled
genes when looking $m$ generations backwards in time:
$i$ and $j$ are in the same block of $\xi^{n,N}\(m\)$ if and only if
the $i$-th and the $j$-th sampled genes have the same ancestral gene
$m$ generations ago. We also keep track of the grouping of these
ancestral genes into ancestral (diploid) individuals
(this is necessary so that the dynamics of $\xi^{n,N}$ is Markovian).
For example, in Figure 1 if we sample all genes of the children then the leftmost %rightmost
parent carries two ancestral genes while the third parent from the left %right
only carries one.
We are interested in the convergence of the (suitably time-scaled)
{\it ancestral process} $\(\xi^{n,N}\(m\)\)_{m\in\mathbb{N}_0}$, which is a
Markov chain with state space $\mathcal{S}_n.$

For the description of the possible limit processes as $N \rightarrow \infty$ the {\em total offspring numbers}
\begin{align}
  \label{def:Vi}
  V_i := \sum_{1 \le j \le N} V_{i,j},
\end{align}
giving the total number of offspring of individual $i$ for $1\le i \le N$ will play a crucial role.
Note that these $V_i$ children may be full or half siblings.
We have $\sum_{i=1}^N V_i = 2N$ and the vector $(V_i)_{1 \le i \le N}$ inherits exchangeability
from the array $(V_{i,j})$. Indeed, for any permutation $\sigma$ on $\mathbb{V}_N$,
we have
\begin{equation}
\label{eq:Vexch}
\begin{split}
\(V_{\sigma\(1\)},\ldots,V_{\sigma\(N\)}\) =\left(\sum_{j=1}^NV_{\sigma\(i\),j}\right)_{i\in \mathbb{V}_N}
& =\left(\sum_{j=1}^NV_{\sigma\(i\),\sigma{\(j\)}}\right)_{i\in \mathbb{V}_N} \\
& {\buildrel d \over =}\left(\sum_{j=1}^NV_{i,j}\right)_{i\in \mathbb{V}_N}
=\(V_{1},\ldots,V_{N}\).
\end{split}
\end{equation}
Thus, $(V_i)_{1 \le i \le N}$ can essentially be viewed as an offspring distribution
for a Cannings model with population size $2N$ (in which always only $N$ individuals are parents to offspring in the following generation).
\smallskip

In order to consider a suitable scaling for the large population limit
a key quantity is the probability that two genes (picked at
random) from two distinct individuals, which are chosen randomly
without replacement from the same generation, have a common ancestor
(gene) in the previous generation.
In our model this quantity is given by
\begin{equation}
  \label{eq:cN1}
  c_N = \frac18 \EE\big[V_{1,2}^2-\tfrac{2}{N-1}\big] +
  \frac{N-2}{8}\EE\big[V_{1,2}V_{1,3}\big]= \frac{1}{8}\frac{1}{N-1} \EE[ (V_1)_2]
\end{equation}
where $(v)_k:=v(v-1)\cdots(v-k+1)$ denotes the $k$-th falling factorial (see Lemma~\ref{lem:cN} below, where also alternative expressions for $c_N$ are given).
If $c_N\to0$ as $N\to\infty$, the correct time scaling is
$1/c_N$ and any limiting genealogical process will
be a continuous-time Markov chain. We will assume that
$c_N\rightarrow0$ as $N\rightarrow\infty$ throughout this paper.

\smallskip

We write $V_{(1)} \ge V_{(2)} \ge \cdots \ge V_{(N)}$
for the ranked version of $(V_1,\dots,V_N)$ and
\begin{align*}
%\label{rankedsiblingfrequency}
\Phi_N := \mathscr{L} \Big( {\textstyle \frac{V_{(1)}}{2N}, \frac{V_{(2)}}{2N}, \dots,
 \frac{V_{(N)}}{2N}, 0, 0, \dots} \Big)
\end{align*}
for the law of their \emph{ranked (total) offspring frequencies},
viewed as a probability measure on the
infinite dimensional simplex
$\Delta:=\left\{\(x_1,x_2,\ldots\): x_1\geq x_2\geq\cdots\geq 0, \sum_{i=1}^{\infty}x_i\leq 1\right\}$. For all $x=\(x_1,x_2,\ldots\)\in\Delta$, denote by $|x|:=\sum_{i=1}^{\infty}x_i$ and $\(x,x\):=\sum_{i=1}^{\infty}x_i^2$, put $\mathbf{0} :=
(0,0,\dots) \in \Delta$. We equip $\Delta$ with the
topology of coordinate-wise convergence, metrised e.g.\ via
$d_\Delta(x,y) = \sum_{i=1}^\infty 2^{-i} |x_i-y_i|\,\text{for any}\,x,y\in\Delta$.

We assume that
\begin{align}
\label{eq:PhiNconv}
\frac1{2 c_N} \Phi_N(dx) \mathop{\longrightarrow}_{N\to\infty}
\frac{1}{\(x,x\)} \Xi'(dx)
\quad \text{vaguely on}\; \Delta \setminus \{\mathbf{0}\}
\end{align}
where $\Xi'$ is a probability measure on $\Delta$.
It is possible that the vague limit of the l.h.s.\ in \eqref{eq:PhiNconv}
is a strict sub-probability measure, we then add the remaining
mass to $\Xi'$ as the weight of an atom at $\mathbf{0}$,
i.e., we put $\Xi'(\{\mathbf{0}\}) =
1 - \Xi'\(\Delta \setminus \{\mathbf{0}\}\)$.
It is known that \eqref{eq:PhiNconv} is equivalent to the following
condition (see Lemma \ref{lem:equiv} and Remark~\ref{Remark3} in Section \ref{sec:Discussion}):
\begin{equation}
\label{eq:Vfmcond}
\phi_j(k_1,\ldots,k_j) := \lim_{N\rightarrow\infty}
\frac1{c_N} \frac{\EE\big[\(V_1\)_{k_1}\cdots\(V_j\)_{k_j}\big]}{N^{k_1+\cdots+k_j-j}2^{k_1+\cdots+k_j}}
\quad \text{exists for all $j\in\mathbb{N}$ and }
k_1, \dots, k_j\geq 2.
\end{equation}

Other characterisations of the equivalent conditions \eqref{eq:PhiNconv}
and \eqref{eq:Vfmcond} are recalled in the Appendix (see Conditions I, II
and III).
If either of the conditions \eqref{eq:PhiNconv} and \eqref{eq:Vfmcond}
holds, the two limiting objects are connected via (see Theorem \ref{th:sibling})
\begin{equation}
\label{eq:Vfmrel}
\phi_j(k_1,\ldots,k_j) =
\ind_{\{j=1,k_1=2\}} \cdot 2\Xi'(\{\mathbf{0}\}) +
\int_{\Delta\setminus\{\mathbf{0}\}}
\sum_{\scriptstyle i_1,\dots,i_j=1 \atop
\scriptstyle \text{distinct}}^\infty
x_{i_1}^{k_1} x_{i_2}^{k_2} \cdots x_{i_j}^{k_j} \, \frac{2\Xi'(dx)}{(x,x)}.
\end{equation}
%(with $(x,x) = \sum_{i=1}^\infty x_i^2$).
Furthermore, in this case
also the limits
\begin{equation}
\label{eq:Vfmcond2}
\psi_{j,s}(k_1,\ldots,k_j) =
\lim_{N\rightarrow\infty}
\frac1{c_N} \frac{\EE\big[\(V_1\)_{k_1}\cdots\(V_j\)_{k_j}
V_{j+1} \cdots V_{j+s}\big]}{N^{k_1+\cdots+k_j-j}2^{k_1+\cdots+k_j+s}}
%%\quad \text{exist}
\end{equation}
exist, see \cite[Lemma~3.5]{Mohle2001}.
\smallskip

The objects and conditions appearing in \eqref{eq:PhiNconv},
\eqref{eq:Vfmcond} and \eqref{eq:Vfmrel} are familiar from the
theory of coalescents with simultaneous multiple mergers, so called
$\Xi$-coalescents (see \cite{Mohle2001, Schweinsberg2000}).
Let $\Xi$ be a finite measure on $\Delta$. An $n$-$\Xi$-coalescent
is a continuous time (jump-hold) Markov chain $(\xi^n(t))_{t\ge0}$ on
$\mathcal{E}_n$ where in each move, possibly several groups of
 blocks are merged.
If $\eta \in \mathcal{E}_n$ has $b$ blocks and $\eta'$ with $a < b$
blocks arises from $\eta$ by merging $j$ groups of sizes $k_1,\dots,k_j\ge2$
(in particular, there are $s=b-k_1-\cdots-k_j$ ``singleton'' blocks in $\eta$
which do not participate in any merger),
the transition from $\eta$ to $\eta'$ occurs at rate
\begin{align}
\label{eq:Xitransrate}
r_{\eta,\eta'} = \lambda_{b;k_1,\dots,k_j;s}
= \, &
{\ind}_{\{j=1,k_1=2\}}  \Xi(\{\mathbf{0}\}) \notag \\
& {}  +
\int_{\Delta\setminus\{\mathbf{0}\}} \sum^s_{\ell=0} \sum_{\scriptstyle i_1,\dots,i_{j+{\ell}}=1 \atop \scriptstyle \text{distinct}}^\infty
{s \choose \ell} x_{i_1}^{k_1}\cdots x_{i_j}^{k_j} x_{i_{j+1}} \cdots x_{i_{j+\ell}} \(
1- |x|\)^{s-\ell} \frac{\Xi(dx)}{(x,x)}.
\end{align}
With this notation we can now state our main result.
\medskip
\begin{theorem}
\label{thm:res}
Assume that $c_N\to0$ and that the laws of $(V_1,\dots, V_N)$,
derived from $(V_{i,j})$ via \eqref{def:Vi}, satisfy one of
the conditions (\ref{eq:PhiNconv}) and (\ref{eq:Vfmcond}).
Assume also that $\xi^{n,N}(0) = \xi_0 \in \mathcal{S}_n$
for all $N$. Then
\begin{equation*}
  \( \xi^{n,N}\( \lfloor t/c_N \rfloor \) \)_{t\ge0} \longrightarrow
  \( \xi^{n}(t) \)_{t\ge 0}
\end{equation*}
as $N\to \infty$ in the sense of finite-dimensional distributions.
The limit process $\xi^{n}$ is an $n$-$\Xi$-coalescent starting
from $\xi^{n}(0)=\mathsf{cd}(\xi_0)$ with
$\Xi = \Xi' \circ \varphi^{-1}$ where $\Xi'$ is the probability measure on
$\Delta$ appearing on the r.h.s.\ of \eqref{eq:PhiNconv}
and \eqref{eq:Vfmrel}
and $\varphi : \Delta \to \Delta$ is given by
$\varphi(x_1,x_2,x_3,\dots) = (x_1/2,x_1/2,x_2/2,x_2/2,\dots)$.
\end{theorem}

Note that (\ref{eq:Xitransrate}) shows that $\Xi(\{\mathbf{0}\})$ corresponds to the rate for
binary mergers of two {blocks}, which is the dynamics of Kingman's coalescent.
We remark that since $\varphi(\mathbf{0})=\mathbf{0}$ the measures $\Xi'$ and $\Xi$ give the same
mass to $\mathbf{0}$ and so have the same Kingman coalescent component.

We also point out that in Theorem~\ref{thm:res} we only state f.d.d. convergence
since one cannot expect weak convergence on $D([0,\infty), \mathcal{S}_n),$ the set of
 $\mathcal{S}_n$-valued c\`adl\`ag paths equipped with Skorohod's
  $J_1$-topology (see, e.g., \cite[Ch.~3.4]{EK86}). This is because whenever two ancestral
  genes descend from the same parental individual the probability that they descend from different
  ancestral genes (carried by the parental individual) is $1/2,$
  as is the probability that they descend from the same ancestral gene (resulting in a coalescent event).
  We have chosen our scaling such that  the latter event happens at a finite rate in the limit. Thus, also the
  former event, which creates some partition $\xi \in \mathcal{S}_n \setminus \mathcal{E}_n$ happens at a positive rate in the limit.
  But the reason we have f.d.d. convergence in  $\mathcal{E}_n$ in Theorem~\ref{thm:res} is that
  in the limit any $\xi \in \mathcal{S}_n \setminus \mathcal{E}_n$ transitions instantaneously to $\mathsf{cd}(\xi)\in \mathcal{E}_n$.
  Thus, due to the discrete topology on $\mathcal{S}_n$ we always have a non-vanishing probability of an accumulation of jumps of
  finite size which precludes weak convergence in Skorohod's
  $J_1$-topology. However, if we instead consider the process which tracks the succession of complete dispersion states
  then weak convergence on $D([0,\infty), \mathcal{E}_n)$ holds:

\begin{corollary}
  \label{cor:res}
  Let $\widetilde{\xi}^{n,N}(m) := \mathsf{cd}\big( \xi^{n,N}(m)
  \big) \in \mathcal{E}_n$ be the ancestral partition of the $n$~sampled genes
  $m$~generations in the past, irrespective of the grouping into diploid
  individuals.  Under the assumptions of Theorem~\ref{thm:res}, we
  have
  \begin{equation*}
    \( \widetilde{\xi}^{n,N}\( \lfloor t/c_N \rfloor \) \)_{t\ge0} \mathop{\longrightarrow}_{N\to\infty}
    \( \xi^{n}(t) \)_{t\ge 0} \quad
    \text{ weakly on } D([0,\infty), \mathcal{E}_n)
  \end{equation*}
%  where $D([0,\infty), \mathcal{E}_n)$ denotes the set of  $\mathcal{E}_n$-valued c\`adl\`ag paths equipped with Skorohod's $J_1$-topology (see, e.g., \cite[Ch.~3.4]{EK86})
  and the limit  process is the $n$-$\Xi$-coalescent from Theorem~\ref{thm:res}.
\end{corollary}

Before continuing we briefly outline the structure of the remaining paper.
After discussing our main result and its relation to the literature in Section~\ref{sec:Discussion}
we consider various examples and discuss their biological motivation in Section~\ref{sect:Examples}.
In particular, we study two diploid variations of a model by Schweinsberg\ \cite{Schweinsberg2003}
in Section~\ref{Section:Exrff} and \ref{section:supercriticalGaltonWatson}.
We also discuss the relation to previous results on coalescents for diploid population models
and possible extensions in more detail in Section \ref{sect:ExamplesDiscussion}.

The final Section~\ref{Section:Proofs} is dedicated to the proofs of our main results. In Section~\ref{s3}
we prove Theorem~\ref{thm:res} and Corollary~\ref{cor:res}. In Sections~\ref{sProofRandomIndFitness}
and \ref{sProofSchweinsberg} we prove the more technical convergence results for the models considered
in Section~\ref{sect:Examples}: Proposition~\ref{Exrff:prop1} of Section~\ref{Section:Exrff} is proven in
Section~\ref{sProofRandomIndFitness} and Proposition~\ref{prop:diploidSchweinsberg} of
Section~\ref{section:supercriticalGaltonWatson} in Section~\ref{sProofSchweinsberg}.

\subsection{Discussion}
\label{sec:Discussion}

In this section we first give a brief overview over existing coalescent theory in the haploid and diploid setting.
Subsequently, we make several remarks regarding our main results.

Classical large population approximation results  in the haploid setting can be found in Kingman \cite{Kingman82a,Kingman82b}, where a convergence theorem to the classical coalescent (nowadays known as Kingman's coalescent) is established for a class of exchangeable populations.
In recent years, there has been a tremendous development in coalescent theory. We refer to Pitman \cite{Pitman99}, Sagitov \cite{Sag99} and Donnelly \& Kurtz \cite{DK99} for coalescents with multiple mergers and to Schweinsberg \cite{Schweinsberg2000} and M\"ohle \& Sagitov \cite{Mohle2001} for coalescents with simultaneous multiple mergers. At the same time, coalescent theory has been applied to more complex population models. Sagitov \cite{Sag99} deduced a necessary condition for  the convergence of the haploid ancestral process to coalescents with multiple mergers. M\"ohle and Sagitov \cite{Mohle2001} then fully classified haploid exchangeable population models, so called Cannings models,  in terms of the
convergence of their ancestral lines to coalescents with simultaneous multiple mergers. They characterised the coalescent generators in terms of the joint moments of offspring sizes as well as in terms of a sequence of measures defined on the infinite dimensional simplex. Subsequently, Sagitov \cite{Sagitov2003} presented a criterion of weak convergence to the coalescent with simultaneous multiple mergers by a scaled vector of the ranked offspring sizes which constitute a given generation.

For diploid population models, the available theory has been more limited.
M\"ohle \cite{Mohle98a} introduced a diploid population model with selfing and studied the ancestral process in the Wright-Fisher case. He proved that in this case the limit is Kingman's coalescent.
We recover M\"ohle's result without selfing as a special case of our general result, see Section~\ref{sect:Examples}.
In M\"ohle \cite{Mohle98b} it was proved that the scaled ancestral process of $n$ sampled genes in the two-sex Wright-Fisher model behaves like Kingman's coalescent. In this context, M\"ohle also derived coalescence estimates for general offspring mechanisms if only two genes are sampled.
Subsequently, M\"ohle and Sagitov \cite{MohleSagitov2003}  completely classified the coalescent patterns in two-sex diploid exchangeable population models and established conditions for the limiting scaled ancestral process  to either be Kingman's coalescent or the coalescent with (simultaneous) multiple mergers.
In contrast to our set-up, individuals are either male or female ($N$ individuals each) and in each generation $N$ couples are formed that have children according to a general exchangeable offspring distribution. Sexes are again assigned randomly conditioned on there being again $N$ males and $N$ females. This is a special case of our result, see Section \ref{sect:randompairs}.
In fact, Theorem~\ref{thm:res} is in a sense an explicitly worked-out version of the remarks in \cite[Section~7]{MohleSagitov2003}.

Birkner et al.\ \cite{BBE13} studied a diploid Moran type population model in which two individuals drawn uniformly at random contribute a (potentially) large number of offspring relative to the total population size. They proved that due to this property and the diploid inheritance the scaled ancestral process admits in the limit simultaneous multiple mergers in up to four groups. In Section \ref{sect:Exampleonelargefamily} we give more details on the relationship to our main results. In particular, the single-locus analogues of Theorems 1.2 and 1.3 in \cite{BBE13} can be recovered as a special case of Theorem~\ref{thm:res}.

\bigskip
\noindent
We would like to emphasize a number of points regarding our main results:
  \begin{enumerate}[1.]
  \item \label{Remark1}
    Broadly speaking, Theorem~\ref{thm:res} says that we can (for
    $N$ large) use ``equivalent'' sampling on the gene level and ignore
    the grouping of genes into diploid individuals. This phenomenon has
    been observed many times before (e.g.\ in \cite{MohleSagitov2003} and \cite{BBE13}), it is explained
    by an asymptotic separation of time-scales: the ``breaking up''
    of grouping into diploids is much faster than non-trivial coalescence
    on the gene level (see the proof of Theorem~\ref{thm:res}).

    For finite $N$, the process $\widetilde{\xi}^{n,N}$ is in general not a Markov chain,
    this is one of the reasons why we consider $\xi^{n,N}$ in Theorem~\ref{thm:res}. However,
    the limit process is Markovian.
    \medskip

  \item  \label{Remark2}
    $\big( \xi^{n}(t) \big)_{t\ge 0}$ can in a natural way be
    interpreted as a tree describing the genealogy of $n$ sampled
    genes. In population genetics applications, functionals of this
    tree, in particular the total length
    \begin{align*}
      L_{\mathrm{tot}}(\xi^n) & := \int_0^\tau \# \xi^{n}(t) \, dt \qquad \text{ with } \tau := \inf\{ t \ge 0 : \# \xi^{n}(t) = 1 \} \\
      \intertext{and the length of all branches subtending $i$ leaves
        for $i \in \{1,2,\dots,n-1\}$} L_i(\xi^n) & := \int_0^\tau \#
      \big\{ C \in \xi^{n}(t) : \# C = i \big\}\, dt
    \end{align*}
    are of interest. By Corollary~\ref{cor:res}, the distribution of such
    functionals of $\big( \widetilde{\xi}^{n,N}\big( \lfloor t/c_N \rfloor \big) \big)_{t\ge0}$
    converges as well.
    \medskip

  \item \label{Remark3}
    Note the normalisation with $2 c_N$ in \eqref{eq:PhiNconv}.
    % and \eqref{eq:Vfmcond}. $\Xi'$ as it appears in \eqref{eq:PhiNconv},
    The expression in \eqref{eq:Vfmrel} is the limit object related to sampling according to
    the (``haploid'') offspring vectors $(V_1,\dots,V_N)$ and the
    asymptotically correct scaling of $\Phi_N$ (so that the corresponding limit object
    $\Xi'$ is a probability measure on $\Delta$) is given by $1/c_N'$ with
    \begin{equation*}
      c_N' = \EE\left[ \frac1{2N(2N-1)} \sum_{i=1}^N V_i (V_i-1) \right]
      = \frac{\EE[V_1(V_1-1)]}{2(2N-1)},
    \end{equation*}
    see e.g.\ \cite[Eq.~(1.5)]{Sagitov2003}.
    We can interpret $c_N'$ as referring to sampling directly on the level of
    chromosomes where $c_N$ as defined in \eqref{eq:cN1} refers to sampling
    on the level of diploid individuals. We have $c_N' \sim 2 c_N$ for
    $N\to\infty$ (see also Lemma~\ref{lem:cN})
    and our normalisation in \eqref{eq:Vfmrel} entails $\phi_1(2)=2$.
    \medskip

 % \item Arguably, Theorem~\ref{thm:res} is in a sense an explicitly worked-out version
   % of the remarks in \cite[Section~7]{MohleSagitov2003}.
   % See Section~\ref{sect:ExamplesDiscussion} for a more complete discussion of the relation to previous work.
    %Thm.~1.3 in \cite{BBE13} can be recovered as a special case of Theorem~\ref{thm:res}.
    %\medskip
    % I have taken this out as it was repetitive given the above literature review

  \item \label{Remark4}
    \eqref{eq:exchangeable} says that $(V_{i,j})$ is a finite
    jointly exchangeable array. A related notion is that of
    ``separately exchangeable arrays'' where rows and columns may use different permutations.
    See also the discussion in Section~\ref{sect:ExamplesDiscussion} and
    %%For (at lot more!) background on (mostly infinite) exchangeable arrays, etc. see :
    see e.g.\ \cite{Kallenberg05}, \cite{Austin08} for general background on exchangeable arrays.

    % Aldous / Hoover showed: Any infinite jointly exchangeable array can
    % be represented as
    % \[
    % V_{ij} = f(U, U_i, U_j, U_{ij})
    % \]
    % with $U$, $U_i$, $U_{ij}$, $i, j \in \bN$, $j<i$ i.i.d., say $\sim \mathrm{unif}([0,1])$ (and
    % the implicit definition $U_{ji}=U_{ij}$).
    % \smallskip

    % (The other notion is a ``separately exchangeable array''
    % where rows and columns may use different permutations, then $U_{ij}$ and $U_{ji}$ are independent draws.)
  \end{enumerate}

\section{Examples}\label{sect:Examples}

We will now apply Theorem \ref{thm:res} to various examples. Some of these
have been considered in the literature before, and we recover the
known limiting results in an efficient way, some of the examples are
new or analysed here in a more general setting.

\smallskip

Arguably, the simplest diploid population model is
the diploid Wright-Fisher model %% without selfing
where each individual in
the children's generation is independently assigned two distinct
parents by drawing twice without replacement from the parent's
generation; the joint distribution of $\(V_{i,j}\)_{1\leq
  i<j\leq N}$ is then a ${N\choose 2}$-dimensional multinomial distribution
with uniform weights.
%% $\mathrm{Multinomial}\big(N;1/{N\choose 2}, \dots,1/{N\choose 2}\big)$.
This model was considered e.g.\ by M\"ohle \cite{Mohle98a} and
to set the stage we briefly discuss how we recover his
result for the case with no selfing ($s=0$ in M\"ohle's
\cite{Mohle98a} notation) %easily
from Theorem~\ref{thm:res}.
\begin{proposition}
  \label{ex:diploidWF}
  In the diploid Wright-Fisher model without selfing, using time-scaling with $c_N=1/(2N)$, the limiting
  coalescent (in the f.d.d. sense) is Kingman's coalescent (which corresponds to the case
  $\Xi'=\Xi=\delta_{\mathbf{0}}$ in Thm.~~\ref{thm:res}).
\end{proposition}
\begin{proof}
  By choosing $W_i \equiv 1$ in Section~\ref{Section:Exrff}
    this model is a special case of the class considered there and the
    result follows from Proposition~\ref{Exrff:prop1},
    case~1. Alternatively, one can easily check that the (sharp)
    criterion on the third factorial moment of $V_1$ from M\"ohle
    \cite[Eq.~(14) in Sect.~4]{Mohle00} for convergence to Kingman's
    coalescent is satisfied.

\end{proof}
\medskip

It is well known that the class of possible coalescent processes
arising as limiting genealogies in population models is much richer
than just Kingman's coalescent.  An important family of examples for
the haploid case is given by the $\mathrm{Beta}(2-\alpha,
\alpha)$-coalescents, where $0 < \alpha < 2$. For the sub-case $1 \le
\alpha < 2$ these are well motivated by a class of models that were
proposed and analysed by Schweinsberg\ \cite{Schweinsberg2003},
which work as follows: Let each generation consist of $N$
haploid (adult) individuals. Individual $i$ produces $X_i$ juveniles,
where $X_1,\dots,X_N$ are independent copies of $X$ with $\EE[X]>1$ and
$X$ has a (strictly) regularly varying tail,
\begin{equation}
  \label{eq:SchweinsbergXtail}
\PP(X > x) \sim c x^{-\alpha} \quad \text{as } x\to\infty
\end{equation}
with $c \in (0,\infty)$. Then, $N$ of the $S_N = X_1+\cdots+X_N$ ($\,
> N$ typically) juveniles are drawn at random without replacement to
form the next (adult) generation.  It turns out (see
\cite[Thm.~4]{Schweinsberg2003}) that in the limit $N \to \infty$, the
suitably scaled genealogies of samples from such a population model
converge to a $\mathrm{Beta}(2-\alpha, \alpha)$-coalescent.  This is a
particular so-called $\Lambda$-coalescent. In the notation of
\eqref{eq:Xitransrate} it is given by
\[
\Xi \(dx\)= \int_{[0,1]} \delta_{\(x,0,0,\dots\)} \, \mathrm{Beta}\(2-\alpha, \alpha\)\(dx\)
\]
where $\mathrm{Beta}(2-\alpha, \alpha)$ is the probability law on $[0,1]$ with density
\begin{equation}
  \label{eq:Betadens}
  %%\frac{1}{\Gamma(2-\alpha)\Gamma(\alpha)}
  \frac1{B\(2-\alpha,\alpha\)} x^{1-\alpha} \(1-x\)^{\alpha-1}, \; \; 0 < x < 1.
\end{equation}
Here, for $a, b >0$, $B(a,b) = \Gamma(a)\Gamma(b)/\Gamma(a+b)$ denotes the Beta-function.
There is also a rich
mathematical structure linking these particular $\Lambda$-coalescents to
stable branching processes, see e.g. \cite{BBS07, BBCEMSW05}.
\smallskip

This model captures situations where occasionally some individuals can, for example
due to environmental fluctuations, possibly produce many more
offspring than others (note that \eqref{eq:SchweinsbergXtail} with $\alpha<2$ implies
$\mathrm{Var}[X_i] = \infty$). It is thus a possible mathematical
formalisation of the concept of ``sweepstakes reproduction'' that
appears in the biological literature, see e.g.\ Eldon and Wakeley
\cite{EW06} and the discussion and references therein.
\medskip

There are various possibilities how one can extend this model -- literally
or in spirit -- to a diploid scenario. We explore two such possibilities
below in more detail: In Section~\ref{Section:Exrff} we assign each individual $i$
in a given generation independently a random ``fitness value'' $W_i \ge 0$ and
decree that each child has a chance $\propto W_i W_j$ to descend from couple $(i,j)$, i.e.\
the joint law of offspring numbers is given by \eqref{Exrff:eq:lawVij}.
In Section~\ref{section:supercriticalGaltonWatson}, each couple $(i,j)$
independently produces $X_{i,j}$ juveniles and then $N$ out of the $\sum_{i<j\le N} X_{i,j}$
juveniles are drawn at random to form the next generation, analogous to \cite{Schweinsberg2003}.
\smallskip

It turns out that two distinct forms of ``diploid $\mathrm{Beta}(2-\alpha, \alpha)$-coalescents''
arise from these two set-ups: The limiting $\Xi$ in Section~\ref{Section:Exrff} arises
from a $\mathrm{Beta}(2-\alpha, \alpha)$-distributed $x$ by replacing it with \emph{two} equal weights
$x/4$ (see Equation~\eqref{Exrff:prop1.e2} in Proposition~\ref{Exrff:prop1}) whereas in
Section~\ref{section:supercriticalGaltonWatson} it is split into \emph{four} equal weights
(see (see Equation~\eqref{eq:diploidSchweinsbergXi4} in Proposition~\ref{prop:diploidSchweinsberg}).

Intuitively, this can be understood as follows: In
Section~\ref{Section:Exrff}, the dominant contribution comes from
situations when one $W_i$ is exceptionally large ($\approx O(N)$)
whereas all others are much smaller; then there is a large family of
half-siblings and the two weights correspond to the two chromosome
copies of the exceptional individual $i$; all the other parents will typically have a total
  number of offspring which is negligible in comparison to $N$ and
  none of their genes will be involved in a multiple merging event.
On the other hand, in
Section~\ref{section:supercriticalGaltonWatson} the dominant
contribution comes from cases when $X_{i,j} \approx O(N)$ for exactly
one couple $(i,j)$ and all other $X_{k,\ell}$ ($\{k,\ell\} \neq
\{i,j\}$) are much smaller; then there is a large family of full
siblings and the four weights correspond to the four chromosome copies
of the two individuals in the successful couple.

In particular, we see that the answer to the question which diploid coalescent is
appropriate for a given biological population with potentially highly
skewed individual reproductive success can depend on the typical mating
behaviour.
\medskip

We describe and analyse these diploid variations of Schweinsberg's
\cite{Schweinsberg2003} model in more detail in
Sections~\ref{Section:Exrff} and
\ref{section:supercriticalGaltonWatson} below.  In
Section~\ref{sect:ExamplesDiscussion}, we briefly discuss how results
of previous studies of diploid population models, especially from
 \cite{BBE13} and \cite{MohleSagitov2003}, fit into our framework.
We also mention possible extensions and additional examples there.
\medskip

The perspicacious reader will observe that in Propositions\
\ref{Exrff:prop1} and \ref{prop:diploidSchweinsberg} below (compare
Assumptions\ \eqref{Exrff:eq:tailassumptW1} and \eqref{ass:Xtail},
respectively), we have excluded the boundary cases $\alpha=2$ and
$\alpha=1$.  In view of Schweinsberg's \cite{Schweinsberg2003} results
for the haploid case, we expect the following: For $\alpha=2$,
$c_N \sim c (\log N)/N$ in Lemmas\ \ref{Exrff:lem:cN} and
\ref{lemma:SchweinsbergcN} and convergence to Kingman's coalescent;
for $\alpha=1$, $c_N \sim c/\log N$ and in both
Proposition~\ref{Exrff:prop1},~2.\ and
Proposition~\ref{prop:diploidSchweinsberg},~2., the uniform
distribution on $[0,1]$ will appear, i.e.\ the limiting coalescent
will then be a variation on the Bolthausen-Sznitman-coalescent where
jumps are broken into two groups and into four groups, respectively.

We leave the details to future work.

\subsection{Diploid population model with random individual fitness}
\label{Section:Exrff}
%% \note{[``Random individual fitness'' / a variation on ``sweepstakes reproduction'': ]}

Let $W_1, W_2, \ldots \ge 0$ be independent copies of nonnegative random variables $W$ with $\mu_W := \EE[W] > 0$, put
\begin{align}
  \label{Exrff:defZN}
  Z_N := \sum_{1 \le i < j \le N} W_i W_j
  = \frac12 \( \sum_{i=1}^N W_i \)^2 - \frac12 \sum_{i=1}^N W_i^2.
\end{align}
Given the $W_i$'s let
\begin{align}
  \label{Exrff:eq:lawVij}
  \(V_{i,j}\)_{1\le i < j \le N} \mathop{=}^d
  \mathrm{Multinomial}\({N}, \frac{W_1 W_2}{Z_N}, \frac{W_1 W_3}{Z_N}, \dots, \frac{W_{N-1} W_N}{Z_N}\).
\end{align}
(we will see in Section \ref{sProofRandomIndFitness} that the event $Z_N=0$ has negligible probability in the limit $N\to\infty$).
Note that when the $W_i$'s are identical, %and constant,
\eqref{Exrff:eq:lawVij} coincides literally with our version of the diploid Wright-Fisher model, see
Proposition~\ref{ex:diploidWF}.
\smallskip

The ``fitness'' in this section's title is not based on an explicitly
modelled genetic type and is not passed on to offspring as the values
are drawn afresh in each generation. The offspring distribution in  \eqref{Exrff:eq:lawVij}
may be appropriate for a population with high individual reproductive
potential (thinking e.g.\ of plants or marine species that can in
principle produce large numbers of seeds or eggs) in an environment
that fluctuates rapidly both in space and time. In reality, there may
be a very complex and highly variable interplay between ecological and
genetic factors that determine the reproductive success of a given
individual at a given time. All this would be subsumed in this model
into a random ``effective fitness parameter'' $W$.

For the fitness parameter we will consider separately the finite variance case
as well as the case of (strictly) regularly varying tails such that
\begin{align}
\label{Exrff:eq:tailassumptW1}
\PP(W \ge x) \sim c_W x^{-\alpha} \quad \text{as } x \to \infty\,\,\,\,\text{for some $c_W \in (0,\infty)$ and $1 < \alpha < 2.$}
\end{align}

Before stating the convergence result we specify the asymptotic behavior
of the scaling parameter $c_N$ as specified in  (\ref{eq:cN1}).
A key quantity is
\begin{equation}
  \label{def:QN}
  Q_N:=\sum_{j=2}^{N}\frac{W_1W_j}{Z_N},
\end{equation}
the probability that a randomly chosen child is an offspring of parent $1$.

% and further described by Lemma \ref{lem:cN} .

\begin{lemma}
  \label{Exrff:lem:cN}
  The pair coalescence probability over one generation
  for $V_{i,j}$'s as in \eqref{Exrff:eq:lawVij} is given by
  \begin{align}
    \label{Exrff:eq:cNformula}
    %c_N & = \frac{N}{8} \EE\left[ \frac{W_1^2 \big( \sum_{j=2}^N W_j\big)^2}{Z_N^2}\right].
    c_N & = \frac{N}{8} \EE\left[ Q_N^2\right].
  \end{align}
  1.\ If $ \mu_W^{(2)}: = \EE\big[ W^2 \big] < \infty$ we have
  \begin{align}
    \label{Exrff:eq:cNKingman}
    c_N \sim C_\mathrm{pair}^{(\mathrm{Kingm})} N^{-1} \quad \text{as } N\to\infty \quad \text{with}\;\;
    C_\mathrm{pair}^{(\mathrm{Kingm})} = \frac{\mu_W^{(2)}}{2 \mu_W^2}.
  \end{align}
  2.\ If \eqref{Exrff:eq:tailassumptW1} holds we have
  \begin{align}
    \label{Exrff:eq:cNBeta}
    c_N \sim C_\mathrm{pair}^{(Beta)} N^{1-\alpha} \quad \text{as } N\to\infty
    \quad \text{with}\;\; C_\mathrm{pair}^{(\mathrm{Beta})} = c_W (2/\mu_W)^\alpha \alpha \Gamma(2-\alpha)\Gamma(\alpha)/8.
  \end{align}
\end{lemma}
By scaling with the appropriate $c_N$ we obtain the following convergence results.
\begin{proposition}
\label{Exrff:prop1}
1.\ If $\mu_W^{(2)} < \infty$ then $(\xi^{n,N}(c N
t))_{t \ge 0}$  with $c=1/C_\mathrm{pair}^{(\mathrm{Kingm})} $ (cf. \eqref{Exrff:eq:cNKingman}) converges  in the
f.d.d. sense to Kingman's coalescent.
\medskip

\noindent 2.\ If the tails of $W$ vary (strictly) regularly as specified in (\ref{Exrff:eq:tailassumptW1})
 then $(\xi^{n,N}(c N^{\alpha-1} t))_{t \ge 0}$ with $c=1/C_\mathrm{pair}^{(\mathrm{Beta})}$
 %$= 8 (2/\mu_W)^{-\alpha}/ (c_W  \alpha \Gamma(2-\alpha)\Gamma(\alpha))$,
 (cf.\ \eqref{Exrff:eq:cNBeta})  converges in the
f.d.d. sense  to a $\Xi$-coalescent with $\Xi=\mathrm{Beta}(2-\alpha,\alpha) \circ \varphi^{-1}$ where
$\varphi : [0,1] \to \Delta$ is given by $\varphi(x)=(x/4,x/4,0,0,\dots),$ that is
\begin{align}
\label{Exrff:prop1.e2}
\Xi (dx)= \int_{[0,1]} \delta_{(\frac{x}{4},\frac{x}{4},0,0,\dots)} \, \mathrm{Beta}(2-\alpha, \alpha)(dx)
\end{align}
with the density of the $\mathrm{Beta}(2-\alpha,\alpha)$ distribution given in   \eqref{eq:Betadens}.
\end{proposition}
The proof of Lemma \ref{Exrff:lem:cN} and Proposition \ref{Exrff:prop1} can be found
in Section \ref{sProofRandomIndFitness}.

\subsection{Diploid population model related to supercritical Galton-Watson processes}
\label{section:supercriticalGaltonWatson}
In this section, we consider another diploid version of the model introduced
and studied by Schweinsberg in \cite{Schweinsberg2003} in which an abundance of offspring is produced in each generation
of which only a limited number survives. The model is similar to that of Section
\ref{Section:Exrff} in that large families may be produced. However, in contrast to the
random individual fitness model of the last section, in which \emph{individual parents} may have many offspring
due to an unusual fitness, we here have \emph{parent couples} that may produce a large family.

More concretely, let $X_{i,j}=X_{i,j}^{(N)},1\le i<j\le N,$ be the
``potential offspring'' of parent $i$ and $j$ in any given generation
with a distribution that %in principle
may depend on $N.$
For notational convenience, we set $X_{j,i}^{(N)}=X_{i,j}^{(N)}$ if $j>i$ and $X_{i,i}^{(N)}=0.$
We also denote the number of potential offspring to parent $i$  by
\begin{equation*}
X_i=X_i^{(N)}=\sum_{j=1, \, j\neq i}^{N} X_{i,j}^{(N)}
\end{equation*}
 for any $1\leq i\leq N.$
 Let
\begin{equation}
\label{S_Ndef}
  S_N= \sum_{ 1\le i<j\le N} X_{i,j}^{(N)}
\end{equation}
be the total number of potential offspring.
%where the $X_{i,j}$ are i.i.d.
%From generation to generation, the vectors $\(\(X_{i,j}^{\(m\)}\)_{1\leq i<j\leq N}\)_{m\in\mathbb{N}_0}$ are i.i.d.
The actual total population size is always fixed at $N.$ If $S_N\geq
N$, we obtain the next generation by sampling $N$ of these offspring
at random without replacement.  We use $V_{i,j}=V_{i,j}^{(N)}$ to
denote the number of offspring sampled from $X_{i,j}$ for any $1\leq
i<j\leq N.$ Our scaling will be such that there are enough offspring
for resampling with sufficiently high probability so that the details
of (any exchangeable) offspring assignment will not be relevant otherwise.
%\smallskip
%
We assume in the following that
\begin{equation}
\label{eq:Xiid}
\(X_{i,j}^{(N)}\)_{1\le i<j\le N} \text{ are i.i.d.}
\end{equation}
so that the potential offspring of parent pairs are generated as in a % (supercritical)
Galton-Watson process.  In addition, we assume that
\begin{equation}
\label{eq:Xijstructure}
\cL\(X_{i,j}^{(N)} \mid X_{i,j}^{(N)}>0\) = \cL(X) \text{ and } p_N := \mathbb{P}\(X_{i,j}^{(N)}>0\) \sim c_{X,1}/N\,\,\text{with $c_{X,1} \in (0,\infty)$}
\end{equation}
where the law of $X$ does not depend on $N$ and satisfies
\begin{equation}
\label{eq:Xmean}
  \mu_X := \EE[X] \in (2/c_{X,1}, \infty).
\end{equation}
Note that
\eqref{eq:Xmean} implies that
\begin{equation}
\label{ES_N}
  \EE[S_N] = \binom{N}{2} p_N \mu_X \mathop{\sim}_{N\to\infty} \mu N \quad \text{with\ \ } \mu := \frac{c_{X,1}}{2} \mu_X > 1 .
\end{equation}
(We will see in Lemma~\ref{lemma:SNbounds} below that this implies that
the event $\{S_N < N\}$ has asymptotically negligible probability in
the scaling regimes we consider.)
Finally, we require one of the following assumptions:
\begin{align}
  \label{ass:Xvariance}
  &  %\mathrm{Var}[X] \in (0,\infty)
  \EE\big[X^2\big] < \infty \\
  \intertext{or}
  \label{ass:Xtail}
  & \PP(X>k) \mathop{\sim}_{k\to\infty} c_{X,2} k^{-\alpha} \quad
  \text{for some $\alpha \in (1,2)$ and $c_{X,2} \in (0,\infty)$}.
\end{align}
These assumptions might appear at first sight somewhat artificial, see however the discussion
in Remark~\ref{rem:diploidSchweinbergdiscuss} below.
\smallskip

Before stating the convergence result we again specify the asymptotic behavior
of the scaling parameter $c_N$ given in  (\ref{eq:cN1}).
% and further described by Lemma \ref{lem:cN} .
%We now let
%\[
%c_N = \frac1{8(N-1)} \EE\big[ (V_1)_2\big]
%\]
%be the pair coalescence probability (for two sampled genes picked from
%two distinct individuals, cf. \eqref{eq:cN1} and Lemma~\ref{lem:cN}).
\begin{lemma}
  \label{lemma:SchweinsbergcN}
  1.\ If \eqref{ass:Xvariance} holds then
  \begin{equation}
    \label{eq:SchweinsbergcN.var}
    c_N \sim \tilde{C}_\mathrm{pair}^{(\mathrm{Kingm})}  N^{-1} \quad \quad \text{as } N \to \infty \quad \text{ with } \;\; \tilde{C}_\mathrm{pair}^{(\mathrm{Kingm})} =\frac12\( \frac{\EE[X(X-1)]}{c_{X,1} \mu_X^2} + 1\).
  \end{equation}
  \smallskip

  \noindent
  2.\ If \eqref{ass:Xtail} holds then
  \begin{equation}
    \label{eq:SchweinsbergcN.tail}
    c_N \sim \tilde{C}_{\mathrm{pair}}^{(\mathrm{Beta})}  N^{1-\alpha}
    \quad \quad \text{as } N \to \infty \quad \text{ with } \;\; \tilde{C}_{\mathrm{pair}}^{(\mathrm{Beta})}
    =\frac{1}{8}\frac{c_{X,1} c_{X,2} \alpha}{\mu^{\alpha}} B\(2-\alpha,\alpha\).
  \end{equation}
\end{lemma}
\medskip
By scaling with $c_N$ we obtain the following convergence result.
\begin{proposition}
\label{prop:diploidSchweinsberg}
\noindent
1.\ If \eqref{ass:Xvariance} holds then $\big(\xi^{n,N}(cNt)\big)_{t\ge 0}$ with $c=1/\tilde{C}_\mathrm{pair}^{(\mathrm{Kingm})}$ (cf. \eqref{eq:SchweinsbergcN.var})  converges in the
f.d.d. sense  to Kingman's coalescent.
\medskip

\noindent
2.\ If \eqref{ass:Xtail} holds then $\big(\xi^{n,N}(cN^{\alpha-1}t)\big)_{t\ge 0}$ with $c=1/\tilde{C}_\mathrm{pair}^{(\mathrm{Beta})}$ (cf. \eqref{eq:SchweinsbergcN.tail}) converges in the
f.d.d. sense  to a Beta-coalescent with simultaneous mergers of four groups. More precisely, the limiting coalescent is a $\Xi$-coalescent with
\begin{align}
\label{eq:diploidSchweinsbergXi4}
\Xi\(dx\)=\int_{(0,1]}\delta_{\(\frac{x}{4},\frac{x}{4},\frac{x}{4},\frac{x}{4},0,0,\ldots\)} \mathrm{Beta}(2-\alpha,\alpha)\(dx\),
\end{align}
where the density of the $\mathrm{Beta}(2-\alpha,\alpha)$ distribution is given in \eqref{eq:Betadens}.
%$$\lambda \(dx\)=\frac{x^{1-\alpha}(1-x)^{\alpha-1}dx}{\Gamma(2-\alpha)\Gamma(\alpha)} \; \text{ for } \; x\in(0,1).$$
\end{proposition}
The proof of Lemma \ref{lemma:SchweinsbergcN} and Proposition \ref{prop:diploidSchweinsberg} can be found
in Section \ref{sProofSchweinsberg}.

\begin{remark}
  \label{rem:diploidSchweinbergdiscuss} \rm
  %{\tt [Discussion: Interpretation / ``naturalness'' / possible generalisations] }
  \begin{enumerate}[1.]

  \item The model considered in this section is appropriate for a
    large, unstructured population of $N$ diploid individuals with
    promiscuous reproductive behaviour. It is intended to capture
    situations where there is potentially great variability between
    the number of juveniles produced by different mating couples
    and this is achieved in analogy to Schweinberg's model
    \cite{Schweinsberg2003} in the mathematically simplest way
    by assuming that the $X^{(N)}_{i,j}$ are independent with
    the same distribution.

    At first sight, it may seem then that the structural assumptions
    \eqref{eq:Xijstructure} and \eqref{ass:Xvariance} or
    \eqref{ass:Xtail} are artificial choices just to ``make the
    mathematical theory work''. However, if we stipulate, as seems biologically reasonable, that
    the law $\mathcal{L}(X^{(N)}_1)$ of the number of potential
    offspring of a typical individual (think e.g.\ of the number of
    gametes produced) should be roughly independent of the population
    size $N$ -- in particular, it should not diverge as $N$ grows
    large -- then together with the i.i.d.-assumption for the $X^{(N)}_{i,j}$
    we see that there is essentially no other choice
    than to assume the first half of \eqref{eq:Xijstructure}:
    $\PP(X^{(N)}_{i,j}>0)$ must be of order $1/N$. \eqref{eq:Xijstructure} also means
    that the average number of partners of a typical individual
    stabilises in distribution as $N \to \infty$, in fact, it is
    approximately Poisson distributed with mean $c_{X,1}$.

    From the point of view of a biological model, we
    suggest to read \eqref{eq:Xijstructure} as follows:
    Given a large population of size $N \gg 1$ let
    \[
    c_{X,1} = N \PP(\text{two randomly drawn individuals produce potential offspring together})
    \]
    and let $\mathcal{L}(X)$ be the law of the number of potential offspring produced
    by two randomly drawn individuals, given that they do produce some.
    Then, if $X$ satisfies \eqref{ass:Xvariance} or
    \eqref{ass:Xtail}, the genealogy of an $n$-sample is over time-scales $\propto 1/c_N$
    approximately described by Proposition~\ref{prop:diploidSchweinsberg}.
    See Section~\ref{sect:extensions} for possible extensions.
    \smallskip

    \item
    We do not strive here to answer in full generality the mathematical question
    ``if one only assumes that $X^{(N)}_{i,j}, 1 \le i < j \le N$ are i.i.d.\
    with law $\nu_N$ such that $\sup_N \EE[X^{(N)}_1] = \sup_N (N-1) \EE[X^{(N)}_{1,2}] < \infty$, what are sharp conditions on the family $\nu_N$ of probability measures
    on $\bZ_+$ so that the genealogical processes of finite samples in such population
    models converge?''

    Obviously, then necessarily $\sup_N N \PP(X^{(N)}_{i,j}>0) <
    \infty$ and we see from the proofs of
    Lemma~\ref{lemma:SchweinsbergcN} and
    Proposition~\ref{prop:diploidSchweinsberg} that in the case of infinite
    variance $\EE\big[(X^{(N)}_{i,j})^2\big]=\infty$ suitable control
    of the tail behaviour $\PP(X^{(N)}_{i,j}>x \,|\, X^{(N)}_{i,j}>0)$
    uniformly in $N$ is required for the limit in
    \eqref{proofprop:beta.target1} to exist.  In fact, one can cook up
    examples where $N \PP(X^{(N)}_{i,j}>0)$ and $N^{\alpha-1}
    c_N$ oscillates as a function of $N$ or where even though
    $\lim_{x\to\infty} x^\alpha \PP(X^{(N)}_{i,j}>x \,|\,
    X^{(N)}_{i,j}>0) =: c_{X,2}$ exists for all $N$ one has
    convergence to different coalescents along different subsequences.

  \smallskip

  \item Our parametrisation in \eqref{eq:Xijstructure} enforces
    $\PP(X=0)=0$. If one prefers to allow $0 < \PP(X=0) < 1$, one
    can replace $p_N$ by $p_N \PP(X>0)$ and $X$ by $X'$ where $\PP(X'
    \in \, \cdot) = \PP(X \in \, \cdot \, | \, X>0)$.
    \smallskip

    \smallskip

  % % \item Our proofs of the results in this section are to some extent parallel  to
  % %   those in \cite{Schweinsberg2003}, especially those of
  % %   Lemmas~\ref{lemma:SchweinsbergcN} and \ref{lemma:SNbounds}.
  % %   We note however that the arguments are somewhat more
  % % involved, in particular due to the fact that $X_i^{(N)}, 1 \leq i \leq N$ are not independent.
  % %   Our proof of Proposition~\ref{prop:diploidSchweinsberg},~2.\ follows a
  % %   slightly different route than that of its analogue,
  % %   \cite[Thm.~4~(c)]{Schweinsberg2003} in that we verify
  % %   condition~\eqref{eq:PhiNconv} on the law of the ranked offspring frequencies
  % %    directly without recourse to the moment criterion
  % %   \eqref{eq:Vfmcond}. One can alternatively prove
  % %   Proposition~\ref{prop:diploidSchweinsberg} by verifying
  % %   \eqref{eq:Vfmcond} but this route appears a little more cumbersome
  % %   here because the $X_i$ are not independent in our set-up.
  \end{enumerate}
\end{remark}

\subsection{Relation to previous work and possible further extensions}\label{sect:ExamplesDiscussion}

\subsubsection{Diploid population model with randomly chosen pairs as couples}
\label{sect:randompairs}
We here recover the convergence result  of \cite[Theorem~4.2 and Corollary~4.3]{MohleSagitov2003} concerning a diploid two-sex population model. In order to make comparisons to \cite{MohleSagitov2003} easier we assume here that the population size is given by $2N.$
 Now, let $\{J_1,\ldots,J_{N}\}$ be a random and randomly ordered partition of $\{1,\ldots,2N\}$ into $N$ subsets of size $2$. This partition describes the grouping of individuals into $N$ distinct couples which give birth to the individuals of the next generation. Let $\widetilde{V}_1,\ldots,\widetilde{V}_{N}$ be a sequence of exchangeable non-negative random variables representing the number of children from each couple respectively with $\widetilde{V}_1+\cdots+\widetilde{V}_{N}=2N$.
  Write $\widetilde{V}_{\(1\)}\geq\widetilde{V}_{\(2\)}\geq\cdots\geq\widetilde{V}_{\(N\)}$ for its ranked version. The offspring distribution $(V_{i,j})_{1 \leq i<j \leq 2N}$ is then given by
\begin{equation*}
V_{i,j}=
\begin{cases}
\widetilde{V}_{\ell},\text{~~if~}\{i,j\}=J_{\ell}\text{~for some }{\ell}\in\{1,2,\ldots,N\},\\
0, \text{~~else.}
\end{cases}
\end{equation*}
It is clear that the $\(V_{i,j}\)_{1 \leq i<j \leq 2N}$ are
exchangeable as in (\ref{eq:exchangeable}).  Note that the
corresponding total offspring size vector $(V_1, \dots, V_{2N})$ is a random
permutation of $(\widetilde{V}_1, \widetilde{V}_1, \dots,
\widetilde{V}_{N}, \widetilde{V}_{N}).$ In particular, $V_1$ and
$\widetilde{V}_1$ have the same distribution. Thus, we obtain from
\eqref{eq:cN1} that (remember that we use population size $2N$ here)
$$c_{2N}=\frac{\mathbb{E} [(V_1)_2 ]}{8\(2N-1\)}=\frac{\mathbb{E} [(\widetilde{V}_1)_2 ]}{8\(2N-1\)}.$$
Let us remark that this model can also be interpreted as a two-sex
model, which is the formulation used in \cite{MohleSagitov2003}. Here,
in each generation, we randomly assign sexes to the offspring such
that there are $N$ male and $N$ female offspring. Subsequently, the
couples are formed at random between the males and the females.
%
%By (\ref{def:Vi}), we have the ranked sibling process
%$$\(V_{(1)}, V_{(2)},\ldots, V_{(N)}\)=\(\widetilde{V}_{(1)},\widetilde{V}_{(1)},\widetilde{V}_{(2)},\widetilde{V}_{(2)},\ldots,\widetilde{V}_{(N/2)},\widetilde{V}_{(N/2)}\).$$
%Applying \eqref{timescaling}, we only care the term of full siblings, then
%$$c_N=\mathbb{E}\left[\sum_{i=1}^{N/2}\frac{\widetilde{V}_i\(\widetilde{V}_i-1\)}{N\(N-1\)}\frac14\right]
%=\frac{\mathbb{E}\left[\widetilde{V}_1\(\widetilde{V}_1-1\)\right]}{8\(N-1\)}.$$
%
We have the following proposition:
\begin{proposition}[M\"ohle \& Sagitov\ \cite{MohleSagitov2003}]\label{propn:tildeV}
Assume that $c_{2N}\rightarrow0$ as $N\rightarrow\infty$ and that
\begin{equation}\label{eq:Vitildecon}
\frac{1}{4 c_{2N}}\mathscr{L}\(\frac{\widetilde{V}_{\(1\)}}{2N},\frac{\widetilde{V}_{\(2\)}}{2N},\ldots,\frac{\widetilde{V}_{\(N\)}}{2N}\)
\mathop{\longrightarrow}_{N\to\infty}
\frac{1}{{\(x,x\)}}\Xi''\(dx\)\text{~vaguely on }\Delta\setminus\{\mathbf{0}\}%\{\(0,0,\ldots\)\},
\end{equation}
where $\Xi''$ is a probability measure on $\Delta$.
%Given $\xi^{n,N}(0) = \xi_0 \in \mathcal{S}_n$ for all $N$, then
%\begin{equation}
%  \Big( \xi^{n,N}\big( \lfloor t/c_N \rfloor \big) \Big)_{t\ge0} \longrightarrow \big( \xi^{n}(t) \big)_{t\ge 0}
%\end{equation}
%as $N\to \infty$ in the sense of finite-dimensional distributions.
%The limit process $\xi^{n}$ is an $n$-$\Xi$-coalescent starting from $\xi^{n}(0)=\mathsf{cd}(\xi_0)$ with
%$\Xi = \Xi'' \circ \varphi^{-1} \circ \varphi^{-1}$ where $\Xi''$ is the probability measure on
%$\Delta$ appearing on the r.h.s.\ of \eqref{eq:Vitildecon} and $\varphi : \Delta \to \Delta$ is given by
%$\varphi(x_1,x_2,x_3,\dots) = (x_1/2,x_1/2,x_2/2,x_2/2,\dots)$.
Then there is convergence in the  f.d.d. sense to a $\Xi$-coalescent with $\Xi = \Xi'' \circ \varphi^{-1} \circ \varphi^{-1} = \Xi'' \circ \widetilde{\varphi}^{-1}$ (recall the function $\varphi$ from Thm.\ \ref{thm:res})
with
\begin{equation*}
\widetilde{\varphi}: \Delta \rightarrow \Delta \text{ given by }\widetilde{\varphi}(x_1, x_2, \dots) =(x_1/4,x_1/4,x_1/4,x_1/4,x_2/4,x_2/4,x_2/4,x_2/4,\dots).
\end{equation*}
\end{proposition}
\begin{remark}
\rm
We note that \eqref{eq:Vitildecon} can be equivalently formulated in terms of moment conditions as in \eqref{eq:Vfmcond}
  or as in Appendix~\ref{app:weakconvcrit}.
  \end{remark}
\begin{proof}
We use Theorem\ \ref{thm:res} with condition (\ref{eq:PhiNconv}), which is in the present context of the form
\begin{align}
\label{eq:PhiNconv-twosex}
\frac{1}{2 c_{2N}} \mathscr{L} \( {\textstyle \frac{ \widetilde{V}_{(1)}}{4N}, \frac{ \widetilde{V}_{(1)}}{4N}, \frac{ \widetilde{V}_{(2)}}{4N},
\frac{ \widetilde{V}_{(2)}}{4N}, \dots, \frac{ \widetilde{V}_{(N)}}{4N} , \frac{ \widetilde{V}_{(N)}}{4N} ,0, 0, \dots} \)
\mathop{\longrightarrow}_{N\to\infty}
\frac{1}{(x,x)} \Xi'(dx)
\end{align}
vaguely on $\Delta \setminus \{ (0,0,\dots) \}$ for a certain probability measure $\Xi'$ on $\Delta$.
Note that the measure on the left-hand side of \eqref{eq:PhiNconv-twosex}
equals
\[
\frac12 \frac{1}{c_{2N}} \mathscr{L} \Big( \varphi\big( \tfrac{\widetilde{V}_{(1)}}{2N}, \tfrac{\widetilde{V}_{(2)}}{2N}, \dots,
\tfrac{\widetilde{V}_{(N)}}{2N} \big)\Big)
\]
and that \eqref{eq:Vitildecon} means that for every continuous function $f$ with compact support in $\Delta \setminus\{\mathbf{0}\}$ %\{ (0,0,\dots) \}$
\[
\frac1{c_{2N}} \EE\left[ f\(\varphi\( \tfrac{\widetilde{V}_{(1)}}{2N}, \tfrac{\widetilde{V}_{(2)}}{2N}, \dots,
\tfrac{\widetilde{V}_{(N)}}{2N} \)\) \right] \mathop{\longrightarrow}_{N\to\infty}
4 \int_\Delta f\(\varphi(y)\) \frac1{(y,y)} \, \Xi''(dy)
= 2 \int_\Delta f(x) \frac1{\(x,x\)} \, \( \Xi'' \circ \varphi^{-1}\) (dx)
\]
(observe $\(\varphi(y),\varphi(y)\) = \frac12 \(y,y\)$).
Thus \eqref{eq:PhiNconv-twosex} is equivalent to \eqref{eq:Vitildecon} with
$\Xi' = \Xi''\circ \varphi^{-1}$ and the result follows.
%%\note{ [May want to double check that the factors of 2, etc.\ are all correct.]}
\end{proof}
Note that we see clearly from the form of $\Xi = \Xi'' \circ \widetilde{\varphi}^{-1}$ that the mass of each large family is split into four equal parts, representing the four chromosomes from a particular couple.

% \begin{remark}
% This should recover the case $c_N \rightarrow 0$ in Theorems 4.2 and Corollary 4.3 in  \cite{MohleSagitov2003}. However, in those results the structure of the limiting process is not made quite as explicit, in particular the fact that the mass of the family sizes is split into four equal parts, which is intuitively quite sensible.
% \end{remark}

% \begin{remark}{\tt This is from the old version which provided no other proof for the correct scaling. This remark seemed not quite so clear.}
% Note that the normalization with $4c_N$ in \eqref{eq:Vitildecon}. Here the correct asymptotical scaling $c_{N}^{''}$ is based on the level of couple. Consequently we have
% $$c_N^{''}=\mathbb{E}\left[\frac{1}{N\(N-1\)}\sum_{i=1}^{N/2}\widetilde{V}_i\(\widetilde{V}_i-1\)\right]
% =\frac{\mathbb{E}\left[\widetilde{V}_1\(\widetilde{V}_1-1\)\right]}{2\(N-1\)}=4c_N.$$
% Furthermore, if the sampling is on the level of diploid individuals, similarly as \eqref{eq:PhiNconv}, we have
% \begin{equation}
% \frac{1}{2c_N}\mathscr{L}\(\frac{\widetilde{V}_{\(1\)}}{2N},\frac{\widetilde{V}_{\(1\)}}{2N},\frac{\widetilde{V}_{\(2\)}}{2N},\frac{\widetilde{V}_{\(2\)}}{2N},\ldots,
% \frac{\widetilde{V}_{\(N/2\)}}{2N},\frac{\widetilde{V}_{\(N/2\)}}{2N}\)\overrightarrow{N\rightarrow\infty}
% \frac{1}{\sum_{i=1}^{\infty}x_i^2}\Xi^{'}\(dx\)\text{~vaguely on }\Delta\setminus\{\(0,0,\ldots\)\},
% \end{equation}
% where $\Xi^{'}=\Xi^{''}\circ \varphi^{-1}$. Applying Theorem \ref{thm:res}, we can get the result in Corollary \ref{co:tildeV}.
% \end{remark}

\subsubsection{A diploid population model with occasional large families}

\label{sect:Exampleonelargefamily}
Here, we briefly discuss how the class of continuous-time diploid population models from \cite{BBE13},
which involve suitably rare but ``large'' reproduction events with a single large family, can be
formulated in our present discrete-time context and thus the ``single-locus'' analogues of Theorems~1.2 and 1.3
there can be obtained from our main result.

% First, we consider a diploid population model with only one large family.
The children's generation arises as follows:
Randomly choose two distinct individuals $\{I_1,I_2\}$ from $\mathbb{V}_{N}=\{1,2,\ldots,N\}$. Individuals $I_1$ and $I_2$ form
a couple (as in \cite{MohleSagitov2003}) and have a random number $\Psi_N$ of children together but with no-one else, i.e. $V_{I_1,I_2}=\Psi_N$,
$V_{I_i,j}=0$ for $j \neq I_{3-i}$, $i=1,2$. The other $N-2$ individuals in $\mathbb{V}_{N}\setminus\{I_1,I_2\}$ give birth to $N-\Psi_N$ children according to the diploid Wright-Fisher model.
We have (see %\eqref{timescaling} in the
proof of Lemma~\ref{lem:cN})
\begin{equation}
\label{eq:exampleMSBBE.cN}
\begin{split}
c_N=&\mathbb{E}\bigg[\frac{\Psi_N\(\Psi_N-1\)}{N\(N-1\)}\frac14+\sum_{\substack{i<j\text{ in } \mathbb{V}_N\setminus\{I_1,I_2\} }}\frac{V_{i,j}\(V_{i,j}-1\)}{N\(N-1\)}\frac14+\sum_{\substack{i,j,k\text{ in } \mathbb{V}_N\setminus\{I_1,I_2\} \\ \text{pairwise distinct} }}\frac{V_{i,j}V_{i,k}}{N\(N-1\)}\frac18\bigg]\\
=&\frac{\mathbb{E}\big[\Psi_N\(\Psi_N-1\)\big]}{4N\(N-1\)}+O\({N^{-1}}\).
\end{split}
\end{equation}
The limiting behaviour depends on the sequence of laws $\mathcal {L}\(\Psi_N\)$, $N\in\mathbb{N}$:
\smallskip
\begin{enumerate}[1.]
\item A simple choice, inspired by \cite{EW06}, is to assume that the ``large'' family constitutes always a fixed
fraction of the total population:
Given $\psi\in\(0,1\)$, we assume
$$\mathbb{P}\(\Psi_N=\lfloor\psi N\rfloor\)=1-\mathbb{P}\(\Psi_N=1\)=N^{-\gamma},$$
where $\gamma > 0$. Then
\begin{enumerate}[I.]
\item If $\gamma\in(0,1)$, we have $c_N\sim\frac{\psi^2}{4}N^{-\gamma}$ and (\ref{eq:PhiNconv}) holds true with
$\Xi^{'}=\delta_{\(\frac{\psi}{2},\frac{\psi}{2},0,0,\ldots\)}$. Thus Theorem \ref{thm:res} yields that
the scaled ancestral process converges to a $\Xi$-coalescent process with $\Xi=\delta_{\(\frac{\psi}{4},\frac{\psi}{4},\frac{\psi}{4},\frac{\psi}{4},0,0,\ldots\)}$.\\
\item If $\gamma=1$, we have $c_N\sim\frac{\psi^2}{4N}+\frac{1}{2N}$ and (\ref{eq:PhiNconv}) holds true with
$\Xi^{'}=\frac{\psi^2}{\psi^2+2}\delta_{\(\frac{\psi}{2},\frac{\psi}{2},0,0,\ldots\)}+\frac{2}{\psi^2+2}\delta_{\mathbf{0}}$.
Thus, Theorem \ref{thm:res} shows that
the scaled ancestral process converges to a $\Xi$-coalescent process with $\Xi=\frac{\psi^2}{\psi^2+2}\delta_{\(\frac{\psi}{4},\frac{\psi}{4},\frac{\psi}{4},\frac{\psi}{4},0,0,\ldots\)}+\frac{2}{\psi^2+2}\delta_{\mathbf{0}}$. \\
\item If $\gamma>1$ the limit process is Kingman's coalescent.
\end{enumerate}

We note that alternatively, we could assume that with probability
$1-N^{-\gamma}$ there is just a Wright-Fisher reproduction step and
with probability $N^{-\gamma}$ a reproduction step with one
exceptionally fertile couple as above occurs. This yields the same
limit process.
\medskip

\item
More generally, we can assume that the sequence of laws $\mathcal {L}\(\Psi_N\)$, $N\in\mathbb{N}$ satisfies
$c_N\rightarrow0\text{ as }N\rightarrow\infty$ with $c_N$ from \eqref{eq:exampleMSBBE.cN}
and there exists a %finite
probability measure $F$ on $[0,1]$ such that
$$\frac{1}{2 c_N}\mathbb{P}\(\Psi_N>Nx\)\longrightarrow\int_x^1\frac{1}{y^2}F\(dy\)\text{ as }N\rightarrow\infty\text{~for any $x \in (0,1)$ where $F$ is continuous}.$$
In this case, (\ref{eq:PhiNconv}) holds true with
$\Xi^{'}=\int_{[0,1]}\delta_{\(\frac{y}{2},\frac{y}{2},0,0,\ldots\)} F\(dy\)$
%{\tt Since$ \frac{1}{2}F\((0,1]\)$ should be $\leq 1$ we should have/assume
%$F\((0,1]\) \leq 2.$ Also, the remaining mass should result in a Kingman component.}
and Theorem \ref{thm:res} yields that the scaled ancestral process converges to a $\Xi$-coalescent process with
%$\Xi=\int_{(0,1]}\delta_{\(\frac{y}{4},\frac{y}{4},\frac{y}{4},\frac{y}{4},0,0,\ldots\)}\frac{1}{2}F\(dy\)$.
$\Xi=\int_{[0,1]}\delta_{\(\frac{y}{4},\frac{y}{4},\frac{y}{4},\frac{y}{4},0,0,\ldots\)} F\(dy\)$.
%\note{[Should check constants carefully again here.]}
\medskip

\item One can easily generalise this model to any number $k$ of
``large'' families. Initially, we randomly choose $k$ couples from
$\mathbb{V}_N$. The number of children of those $k$ couples are
{$\Psi_{N,1}$, $\Psi_{N,2}$,$\ldots$,$\Psi_{N,k}$ with
$\Psi_{N,i}=\lfloor\psi_i N\rfloor$} and suitable $\psi_i\in\(0,1\)$
for any $i=1,2,\ldots,k$. Assume $c_N$ converges to $0$, then the
scaled ancestral process converges to a $\Xi$-coalescent process with\\
$\Xi=\delta_{\(\frac{\psi_1}{4},\frac{\psi_1}{4},\frac{\psi_1}{4},\frac{\psi_1}{4},\ldots,\frac{\psi_k}{4},\frac{\psi_k}{4},\frac{\psi_k}{4},\frac{\psi_k}{4},0,0,\ldots\)}$.
\end{enumerate}

\subsubsection{Further remarks and possible extensions}
\label{sect:extensions}

%\note{[Discuss here ``random graph'' viewpoint?]}
The matrix of offspring numbers $(V_{i,j})_{1
  \le i < j \le N}$ can equivalently be viewed as a (n exchangeable) random
multigraph on $N$ nodes, by drawing $V_{i,j}$ undirected edges (one
for each child) between nodes $i$ and $j$. For example, we can
interpret the model from
Section~\ref{section:supercriticalGaltonWatson} as a variation on
Erd\H{o}s-R\'enyi graphs (remembering \eqref{eq:Xijstructure}): First
draw an Erd\H{o}s-R\'enyi graph on $N$ nodes with edge probability
$c_{X,1}/N$, then replace the $i$-th of the resulting $M \approx
c_{X,1}N/2$ edges by $\nu_{i,M}$ edges, where
$(\nu_{1,M},\dots,\nu_{M,M})$ are drawn as in Schweinsberg\
\cite[Section~1.3]{Schweinsberg2003} (except that we enforce
$\nu_{1,M}+\dots+\nu_{M,M} = N$, not $=M$).

From the point of view of biological modelling, one might feel that
the set-up in Section~\ref{section:supercriticalGaltonWatson}, which
in particular enforces that the number of reproductive partners of a
typical individual is essentially Poisson distributed (and hence the
number of potential offspring has a compound Poisson distribution), is
somewhat restrictive. A natural generalisation of
\eqref{eq:Xijstructure} would be the following: Let $D_i$ be independent
copies of an $\NN_0$-valued random variable $D$ with $\EE[e^{\lambda
  D}]<\infty$ for $\lambda$ in a neighbourhood of $0$; we think of $D_i$ as the
number of reproductive partners of individual $i$. Use the
configuration model to assign reproductive partners, i.e.\ attach
$D_i$ ``half-edges'' to node $i$ and then randomly match all
half-edges conditional on producing no self-loops (if $D_1+\cdots+D_N$
is odd, throw away the last half-edge, say).  Finally replace each
resulting edge $e$ by a random number $X_e$ of edges where $X_e$
are independent\ copies of $X$. If $\EE[X] \EE[D] > 2$ and $X$ satisfies
\eqref{ass:Xvariance} or \eqref{ass:Xtail}, then a suitable analogue
of Proposition~\ref{prop:diploidSchweinsberg} will hold.  We do not go
into detail here but note that by Theorem~\ref{thm:res},
asymptotically for our study of genealogies, only the joint law of
$(V_i)_{1\le i \le N}$, which in the language of random graphs
corresponds to the empirical degree distribution, is important. In the
extension of the model from
Section~\ref{section:supercriticalGaltonWatson} just sketched, this
will again on the relevant time-scales be dominated by one
exceptionally large value of $X_e$ if \eqref{ass:Xtail} holds and
negligible compared to $N$ if \eqref{ass:Xvariance} holds.
See e.g.\ \cite{vdhofstad} for background on random graphs, which is currently a
very active research topic.
% \note{Connect more to random graphs literature?  Bollob\'as; Wormald;
%   Molloy and Reed for configuration model?  note: Paul W. Holland,
%   Samuel Leinhardt, An exponential family of probability distributions
%   for directed graphs.  J. Amer. Statist. Assoc. 76 (1981), no. 373, 33--65 ?}
\medskip

Obviously, these models allow various generalisations where the
``degree of promiscuity'' can be chosen as a parameter:
One could for example assign $a\%$ of the children to fixed couples as
in the model from Section~\ref{sect:randompairs} and the
remaining $(100-a)\%$ of the children by using a ``configuration model''
as just discussed.
\medskip

In most of the models discussed so far we did not include individuals
of different sexes. However, as described in Section \ref{sect:randompairs}
 two-sex models, possibly with unequal sex ratio $r : 1-r$,
can be in principle easily embedded into our set-up: take a random
bi-partite ``exchangeable'' multigraph on $\lfloor r N \rfloor$ and $N
- \lfloor r N \rfloor$ nodes with $N$ edges (equivalently: a
separately exchangeable $\lfloor r N \rfloor \times (N - \lfloor r N
\rfloor)$-matrix, with values in $\bN_0$, summing to $N$), assign
individuals $i=1,\dots,N$ randomly to the two ``sex groups''. One can
combine this with small variations of all the models discussed in
Section~\ref{sect:Examples} and one can for example also incorporate
differences in the variance of reproductive success between the two
sexes in this class of models.
%\note{[Spell this out more explicitly?  Even formulate theorem?]}
\medskip

One can also allow the possibility of selfing, i.e.\ $\PP(V_{i,i}>0)>0$.
Then complete dispersion will not happen (asymptotically) immediately but
only after a certain random number of sampled genes have merged due to selfing,
analogous to \cite{Mohle98a}.
We leave the details to future work.

\section{Proofs}
\label{Section:Proofs}

\subsection{Proof of Theorem~\ref{thm:res} and Corollary~\ref{cor:res}}\label{s3}

In this section we prove our main convergence result Theorem ~\ref{thm:res} and Corollary~\ref{cor:res}.
\subsubsection{The pair coalescence probability}
We start by analyzing the pair coalescence probability $c_N.$  Recall that this is the
probability that two genes (picked at random) from
two distinct individuals, which are chosen randomly without
replacement from the same generation (in the population of size $N$),
have a common ancestor (gene) in the previous generation.

\begin{lemma} We have
\label{lem:cN}
\begin{equation}
  \label{lem:cN.eq}
  c_N = \frac18 \EE\left[V_{1,2}^2-\tfrac{2}{N-1}\right] +
  \frac{N-2}{8}\EE\left[V_{1,2}V_{1,3}\right]  =  \frac{1}{8} \EE\left[  V_{1,2} \(V_1-1\)\right]
  = \frac{1}{8}\frac{1}{N-1} \EE\left[ \(V_1\)_2\right].
\end{equation}
\end{lemma}
\begin{proof}
  Pick two distinct individuals at random from the current population
  and pick from each of them independently one of the two gene copies
  at random by a fair coin flip.  These two genes may be descended
  from the same ancestral gene in the previous generation if the two
  individuals have both parents or just one parent in common (full
  siblings, half siblings).  The probabilities that the two genes are
  descended from the same ancestral gene of one of the parent
  individuals is then $\frac{1}{4}$ and $\frac{1}{8}$ respectively.
  Thus
\begin{align*}
c_N & = \EE{\left[ \sum_{1 \le i < j \le N}  \left[ \frac{V_{i,j} (V_{i,j}-1)}{N(N-1)} \frac14 +
\frac{V_{i,j} (V_i-V_{i,j})}{N(N-1)} \frac18 +
\frac{V_{i,j} (V_j-V_{i,j})}{N(N-1)} \frac18 \right] \right] }\\
& = \frac18 \frac{1}{N(N-1)} \sum_{1 \le i < j \le N}  \EE\left[ V_{i,j} \(V_i + V_j - 2\) \right]
= \frac18 \frac{\binom{N}{2}}{N(N-1)} \EE\left[ V_{1,2} \(V_1 + V_2 - 2\) \right] \\
& = \frac18 \EE\left[ V_{1,2} (V_1-1) \right]
= \frac1{8(N-1)} \EE\left[ \sum_{j=2}^N V_{1,j} \(V_1-1\) \right] = \frac{1}{8}\frac{1}{N-1} \EE\left[ \(V_1\)_2\right]
\end{align*}
by the exchangeability assumptions\ \eqref{eq:exchangeable}.
Alternatively write
\begin{align*}
  \frac18 \EE\big[ V_{1,2} (V_1-1) \big] &= \frac18 \sum_{j=2}^N \EE\big[ V_{1,2} V_{1,j} \big] - \frac18 \EE\big[ V_{1,2} \big]
  =  \frac18 \EE\big[V_{1,2}^2\big] - \frac18 \frac{2}{N-1} +
  \frac{N-2}{8}\EE\big[V_{1,2}V_{1,3}\big]
\end{align*}
to obtain the first equality in \eqref{lem:cN.eq}.
\end{proof}
Note that one can also express $c_N$ in terms of variance and covariances of $V_{i,j}$, as follows:
\begin{align*}
%\label{eq:cN2}
c_N = \frac18 \mathrm{Var}[V_{1,2}] + \frac18 (N-2)
\mathrm{Cov}[V_{1,2},V_{1,3}] + \frac1{4\(N-1\)}.
\end{align*}

\medskip
\noindent
\subsubsection{Transition probabilities}
Next, we analyze the transition probabilities of the ancestral process.
Let $\Pi_{n,N} =\(\pi_{n,N}({\xi,\eta})\)_{\xi,\eta\in\mathcal{S}_n}$ be the
transition matrix of the Markov chain
$\(\xi^{n,N}\(m\)\)_{m\in\mathbb{N}_0}$. %% defined in \eqref{xiproc}
($\Pi_{n,N}$ can be viewed as an $|\mathcal{S}_n| \times
|\mathcal{S}_n|$-matrix if we fix an order for $\mathcal{S}_n$, which we will do later).
In particular, we see from the argument in Lemma~\ref{lem:cN} that
$c_N = \pi_{2,N}(\{\{1\},\{2\}\}, \{ \{1,2\} \})$.
\smallskip

It turns out that for our purposes it is
sufficient to describe the $\pi_{n,N}(\xi,\eta)$ in case that $\xi \in \mathcal{E}_n.$
 For some $b\leq n$ we thus consider the transition probability from states $\xi \in \mathcal{E}_n$ to
$\eta \in \mathcal{S}_n$ of the form
 \begin{eqnarray}
 \label{xi}
 \xi&=&\left\{C_1,\ldots, C_b\right\}\quad \text{and} \\
\label{eta}
\eta&=&\left\{\left\{D_1,D_2\right\},\ldots,\left\{D_{2d-1},D_{2d}\right\}, D_{2d+1}, \ldots, D_a\right\},
\end{eqnarray}
for some $a\leq b$ and $2d \leq a$ such that $\xi \subseteq \mathsf{cd}(\eta).$
Assume that $D_i$ is a union of $k_i \ge 1$ classes from $\xi$ with
$k_1+\dots+k_a=b.$

Denote by $\mathcal{E}_{a,d}$ the collection of elements in $\mathcal{E}_{a}$ with $a-d$ blocks,
$d$ of which have cardinality $2$ while the other $a-2d$ have cardinality $1$. Then we can describe the grouping
into diploid individuals in $\eta$  via $\zeta = \{\zeta_1,\dots,\zeta_{a-d}\}
\in \mathcal{E}_{a,d}:$ Let $D_i$ and $D_j$ belong to
the same diploid ancestral individual in $\eta$ if and only if $\{i,j\} \in
\zeta$.
We put
\begin{equation}
\label{eq:ellk}
\ell_j := \sum_{i \in \zeta_j} k_i, \quad j=1,\dots,a-d,
\end{equation}
which is the number of offspring classes in $\xi$ that belong to
the $j$-th ancestral individual described by $\eta$, and we have therefore that $\sum_{j=1}^{a-d} {\ell}_j=b.$
%We call an ancestral individual of type $i \in\left\{0,1,2\right\}$ if exactly $i$
%of its two genes belong to the set of ancestral genes. \note{[technically, when $i=0$ it is not an ancestral individual]}
%This terminology was only used once in the following.

In order to calculate and describe the transition probabilities we will introduce another useful concept and corresponding notation.
Recall that  $V_{i,j}$ ($=V_{j,i}$) represents the random number of offspring of the parental individuals
$i$ and $j.$ By definition, each offspring inherits one chromosome copy
from each parent. Let us assume that we randomly and uniformly ``mark''
one of these chromosome copies as ``relevant'' (in the sense that it will be
this copy that we possibly later examine in the child),
let $\widehat{V}_{i,j}$ be the number of offspring with parents
from $i$ and $j$ who inherited their relevant chromosome copy from
parent $i$. Mathematically, this means that conditional on $(V_{u,w})_{1 \le u,w \le N}$,
$\widehat{V}_{i,j}$ is $\mathrm{Bin}(V_{i,j},\tfrac12)$-distributed,
$\widehat{V}_{j,i} = V_{i,j} - \widehat{V}_{i,j}$ and
$\widehat{V}_{i,j}$ and $\widehat{V}_{k,\ell}$ are independent when
$\{i,j\} \neq \{k,\ell\}$.
We will  write
\begin{equation*}
%\label{eq:Vhat}
\widehat{V}_i := \sum_{j\neq i}^N \widehat{V}_{i,j}
\end{equation*}
for the total number of ``relevant'' offspring of individual $i$.
Note that we have $\sum_{i=1}^N \widehat{V}_i = N$ by definition.
% and that the vector $(\widehat{V}_1, \widehat{V}_2,\dots, \widehat{V}_N)$ is
% exchangeable (see Lemma \ref{lem:sibsizeexch}).
\medskip

\begin{lemma}
\label{lem:Vhatex}
Both the array $\big( \widehat{V}_{i,j} \big)_{1\le i\neq j \le N}$ of relevant pairwise offspring
numbers and the vector $(\widehat{V}_{i})_{1\leq i\leq N}$ of relevant total offspring numbers are exchangeable.
%\note{[are the names okay?]}
% The relevant chromosome size vector $\Big( \widehat{V}_{i,j} \Big)_{1\le i\neq j \le N}$ satisfies
% $$\Big( \widehat{V}_{\sigma(i),\sigma(j)} \Big)_{1\le i \neq j\le N}
% \mathop{=}^d \Big( \widehat{V}_{i,j} \Big)_{1\le i\neq j \le N}£¬$$
% where $\sigma$ is any permutation on $\mathbb{V}_N$.
\end{lemma}
\begin{proof}
Note that
\begin{equation}\label{eq:Vhat2}
\begin{split}
&\mathbb{P}\(\widehat{V}_{i,j}=\hat{v}_{i,j},1\leq i\neq j\leq N\)\\
=&\mathbb{P}\({V}_{i,j}=\hat{v}_{i,j}+\hat{v}_{j,i},1\leq i< j\leq N\)\prod_{1\leq i<j\leq n}{{\hat{v}_{i,j}+\hat{v}_{j,i}}\choose{\hat{v}_{i,j}}}2^{-\(\hat{v}_{i,j}+\hat{v}_{j,i}\)}.\\
\end{split}
\end{equation}
For any permutation $\sigma$ on $\mathbb{V}_N$, we have
\begin{eqnarray}
\nonumber
\mathbb{P}\(\widehat{V}_{i,j}=\hat{v}_{\sigma\(i\),\sigma\(j\)},1\leq i\neq j\leq N\)
&=&\mathbb{P}\({V}_{i,j}=\hat{v}_{\sigma\(i\),\sigma\(j\)}+\hat{v}_{\sigma\(j\),\sigma\(i\)},1\leq i< j\leq N\)\\
\label{eq:Vhat1}
&&%\phantom{AAA}
\cdot \prod_{1\leq i<j\leq n}{{\hat{v}_{\sigma\(i\),\sigma\(j\)}+\hat{v}_{\sigma\(j\),\sigma\(i\)}}\choose{\hat{v}_{\sigma\(i\),\sigma\(j\)}}}2^{-\(\hat{v}_{\sigma\(i\),\sigma\(j\)}+\hat{v}_{\sigma\(j\),\sigma\(i\)}\)}.
\end{eqnarray}
It follows from (\ref{eq:exchangeable}) that $\(\ref{eq:Vhat2}\)$ equals $\(\ref{eq:Vhat1}\)$,
i.e.\ ${\rm \( \widehat{V}_{\sigma(i),\sigma(j)} \)_{1\le i \neq j\le N}
\mathop{=}^d\( \widehat{V}_{i,j} \)_{1\le i\neq j \le N}}$.
\medskip

Exchangeability of $(\widehat{V}_{i})_{1\leq i\leq N}$ follows from
this as in \eqref{eq:Vexch}.
\end{proof}

% \begin{lemma}
% \label{lem:sibsizeexch}
% The sibling size vector $(V_{i})_{1\leq i\leq N}$ as well as the vector $(\widehat{V}_{i})_{1\leq i\leq N}$ is exchangeable.
% \end{lemma}
% \begin{proof}
% By the exchangeability of $\(V_{i,j}\)_{1\leq i<j\leq N}$, we know that
% $$\(V_{\sigma{\(i\)},\sigma{\(j\)}}\)_{i,j}{\buildrel d \over =}\(V_{i,j}\)_{i,j},$$
% where $\sigma$ is any permutation on $\mathbb{V}_N$.
% %
% For any $1\leq i\leq N$, we have $V_i=\sum_{j=1}^{N}V_{i,j}$. Consequently, given any permutation $\sigma$ on $\mathbb{V}_N$, it follows that
% \begin{equation*}
% \begin{split}
% \(V_{\sigma\(1\)},\ldots,V_{\sigma\(N\)}\)&=\left(\sum_{j=1}^NV_{\sigma\(i\),j}\right)_{i\in \mathbb{V}_N}=\left(\sum_{j=1}^NV_{\sigma\(i\),\sigma{\(j\)}}\right)_{i\in \mathbb{V}_N}
% {\buildrel d \over =}\left(\sum_{j=1}^NV_{i,j}\right)_{i\in \mathbb{V}_N}\\
% &=\(V_{1},\ldots,V_{N}\).
% \end{split}
% \end{equation*}
% The same arguments now apply to $(\widehat{V}_{i})_{1\leq i\leq N}$ due to the exchangeability of Lemma \ref{lem:Vhatex}.
% \end{proof}

\medskip

The following lemma states that
if the limits in (\ref{eq:Vfmcond}) exist then they can also be expressed in terms of the
quantities $(\widehat{V}_i)_{i=1, \dots, N}$ instead of the $(V_i)_{i=1, \dots, N}.$
\begin{lemma}
\label{lem:Vihatprod}
Let $b, c \in \bN$, $\ell_1,\dots,\ell_c \in \bN$ with
$\ell_1+\cdots+\ell_c=b$. Under Condition (\ref{eq:Vfmcond}), we have
\begin{equation*}
%\label{eq:Vihatprodasympt}
\lim_{N\to\infty} \frac{\EE\big[ (\widehat{V}_1)_{\ell_1} \cdots
(\widehat{V}_c)_{\ell_c}\big]}{c_N \, N^{b-c}}
= \lim_{N\to\infty} \frac{\EE\big[ ({V}_1)_{\ell_1} \cdots
({V}_c)_{\ell_c}\big]}{c_N \, N^{b-c}\, 2^b} .
\end{equation*}
\end{lemma}
\begin{proof}
Recall the combinatorial identity for choosing (without replacement) $\ell$ objects out of $\sum_{i=1}^n a_i$ objects
\begin{align}
\label{eq:combid1}
\binom{\sum_{i=1}^n a_i}{\ell}=\frac1{\ell!} \big( {\textstyle \sum_{i=1}^n a_i} \big)_\ell =
\sum_{(k_1,\dots,k_n) \in \bN_0^n \atop k_1+\cdots+k_n=\ell}
\prod_{i=1}^n \frac{(a_i)_{k_i}}{k_i!}, \quad \ell, n \in \bN, \;
a_1,\dots,a_n \in \bN_0,
\end{align}
which results from choosing exactly $k_i$ out of $a_i$ objects and from considering all the possible choices of
$k_1+\cdots+k_n=\ell.$ Now, set
\[
\mathcal{K}_N(\ell_1,\dots,\ell_c)
:= \big\{ (k_{i,j})_{i=1,\dots,c,\,j=1,\dots,N} \in \bN_0^{c \times N} :
k_{i,1}+\cdots+k_{i,N}=\ell_i \; \text{with}\;k_{i,i}=0\; \text{for}\; i=1,\dots,c \big\}.
\]
Then, expanding the definition of $\widehat{V}_i$ and using \eqref{eq:combid1}
yield
\begin{align}
\prod_{i=1}^c \frac{(\widehat{V}_i)_{\ell_i}}{\ell_i!} & =
\prod_{i=1}^c \Big\{ \hspace{-0.8em}
\sum_{(k_{i,1},\dots,k_{i,N}) \in \bN_0^N \atop
k_{i,1}+\cdots+{k_{i,N}}=\ell_i, k_{i,i}=0}
\prod_{j \neq i}^N \frac{(\widehat{V}_{i,j})_{k_{i,j}}}{k_{i,j}!}
\Big\}
= \sum_{ (k_{i,j}) \in \mathcal{K}_N(\ell_1,\dots,\ell_c) }
\prod_{i=1}^c \Big\{
\prod_{j\neq i}^N \frac{(\widehat{V}_{i,j})_{k_{i,j}}}{k_{i,j}!} \Big\}\
\notag \\
\label{eq0:lem:Vihatprod}
& = \sum_{ (k_{i,j}) \in \mathcal{K}_N(\ell_1,\dots,\ell_c) }
\Big\{ \prod_{1 \le i < j \le c} \frac{(\widehat{V}_{i,j})_{k_{i,j}}}{k_{i,j}!}
\frac{(V_{i,j}-\widehat{V}_{i,j})_{k_{j,i}}}{k_{j,i}!} \Big\}
\times
\Big\{ \prod_{i=1}^c \prod_{j=c+1}^N \frac{(\widehat{V}_{i,j})_{k_{i,j}}}{k_{i,j}!}
\Big\}
\end{align}
Thus,
\begin{align}
& \EE\Big[ \prod_{i=1}^c \frac{(\widehat{V}_i)_{\ell_i}}{\ell_i!} \,\Big| \,
(V_{u,w})_{1 \le u,w \le N} \Big] \notag \\
& = \sum_{(k_{i,j}) \in \mathcal{K}_N(\ell_1,\dots,\ell_c)}
\Big\{ \prod_{1 \le i < j \le c}
\frac{1}{2^{k_{i,j}+k_{j,i}}}\frac{(V_{i,j})_{k_{i,j}+k_{j,i}}}{k_{i,j}! \, k_{j,i}!}
\Big\} \times
\Big\{ \prod_{i=1}^c \prod_{j=c+1}^N \frac{1}{2^{k_{i,j}}}
\frac{(V_{i,j})_{k_{i,j}}}{k_{i,j}!}
\Big\} \notag \\
\label{eq1:lem:Vihatprod}
& = \frac1{2^b} \sum_{(k_{i,j}) \in \mathcal{K}_N(\ell_1,\dots,\ell_c)}
\Big\{ \prod_{1 \le i < j \le c} \frac{(V_{i,j})_{k_{i,j}+k_{j,i}}}{k_{i,j}! \, k_{j,i}!}
\Big\} \times
\Big\{ \prod_{i=1}^c \prod_{j=c+1}^N \frac{(V_{i,j})_{k_{i,j}}}{k_{i,j}!} \Big\}
\end{align}
where we have used the fact that
\[
\EE \Big[ (\widehat{V}_{i,j})_{k_{i,j}}
(V_{i,j}-\widehat{V}_{i,j})_{k_{j,i}} \, \Big| \,
(V_{u,w})_{1 \le u,w \le N} \Big]
= \frac1{2^{k_{i,j}+k_{j,i}}} (V_{i,j})_{k_{i,j}+k_{j,i}}
\]
(a special case of an identity for mixed factorial moments of a multinomial
vector, see also Lemma~\ref{lem:Multinomfactorialmoments})
and the conditional independence properties of $\widehat{V}_{i,j}$'s
in the first equation. Note that if
we replace in \eqref{eq1:lem:Vihatprod} the term
\begin{align}
\label{eq3:lem:Vihatprod}
\prod_{1 \le i < j \le c} \frac{(V_{i,j})_{k_{i,j}+k_{j,i}}}{k_{i,j}! \, k_{j,i}!}
\quad \text{by} \quad
\prod_{1 \le i < j \le c}
\frac{(V_{i,j})_{k_{i,j}} (V_{i,j})_{k_{j,i}}}{k_{i,j}! \, k_{j,i}!}
\end{align}
then we obtain (as in (\ref{eq0:lem:Vihatprod})),
\begin{align}
\label{eq4:lem:Vihatprod}
\frac1{2^b} \sum_{(k_{i,j}) \in \mathcal{K}_N(\ell_1,\dots,\ell_c)}
\prod_{i=1}^c \prod_{j=1 \atop j \neq i}^N \frac{(V_{i,j})_{k_{i,j}}}{k_{i,j}!}
= \frac1{2^b} \prod_{i=1}^c \frac{({V}_i)_{\ell_i}}{\ell_i!}.
\end{align}
The difference of the two terms $(V_{i,j})_{k_{i,j}+k_{j,i}}$ and
$(V_{i,j})_{k_{i,j}} (V_{i,j})_{k_{j,i}}$ which get replaced inside the
product in \eqref{eq3:lem:Vihatprod} vanishes whenever $k_{i,j}+k_{j,i} \le 1$ and is
\[
O\big(  (V_{i,j})^{k_{i,j}+k_{j,i}-1} \big)
\]
otherwise (with a combinatorial constant that depends on $c$ and
$\ell_1,\dots,\ell_c$ but not on $N$). Thus,
\begin{align*}
\frac{1}{c_N N^{b-c}} \Big(
\EE\big[ \text{term in \eqref{eq1:lem:Vihatprod}} \big]
- \EE\big[ \text{term on r.h.s.\ of \eqref{eq4:lem:Vihatprod}} \big] \Big)
\longrightarrow 0 \quad \text{as }N\to\infty
\end{align*}
because of Condition (\ref{eq:Vfmcond}), which is the claim.
\end{proof}

With the help of this lemma we can now prove the following:
\begin{lemma}
\label{lem:piNrenorm}
For $\xi$, $\eta$ as in \eqref{xi}, \eqref{eta} with $\ell_j$ from
\eqref{eq:ellk}, we have
\begin{align}
\label{eq:piNrenorm}
\lim_{N\to\infty} \frac{\pi_{n,N}({\xi,\eta})}{c_N}
\, & = \lim_{N\to\infty} \sum_{i_1,\dots,i_{a-d}=1 \atop \text{distinct}}^N
\frac{\EE\big[ (V_{i_1})_{\ell_1} (V_{i_2})_{\ell_2} \cdots
(V_{i_{a-d}})_{\ell_{a-d}} \big] 2^{a-d}}{c_N N^b 2^{2b}}  \notag \\
& = \lim_{N\to\infty} \frac{\EE\big[ (V_{1})_{\ell_1} (V_{2})_{\ell_2} \cdots
(V_{{a-d}})_{\ell_{a-d}} \big]}{c_N \, 2^b \, (2N)^{b-a+d}}.
%= \phi_{a-d}(\ell_{1},\dots, \ell_{a-d}) \cdot \frac{1}{2^{b-a+d}}.
\end{align}
\end{lemma}
%\note {\tt [there may exists some singletons, i.e. $\exists l_j=1$, $ \phi_{a-d}$ is true only for all the $l_j>1$, then the last equality can be ignored]}.
If $\ell_1,\dots,\ell_{a-d} \ge 2$ we see from \eqref{eq:Vfmcond}
that the limit in \eqref{eq:piNrenorm} equals
\begin{align*}
\phi_{a-d}(\ell_{1},\dots, \ell_{a-d}) \cdot \frac{1}{2^{b-a+d}}.
\end{align*}
When $s\ge 1$ of the $\ell_i$ are equal to $1$, say $\ell_1,\dots,\ell_{a-d-s} \ge 2$,
$\ell_{a-d-s+1}=\cdots=\ell_{a-d}=1$ we see from \eqref{eq:Vfmcond2}
that the limit in \eqref{eq:piNrenorm} equals
\begin{align*}
\psi_{a-d-s,s}(\ell_{1},\dots, \ell_{a-d-s}) \cdot \frac{1}{2^{b-a+d}}.
\end{align*}

\begin{proof}[Proof of Lemma~\ref{lem:piNrenorm}]
Since $\xi \in \mathcal{E}_n \subset \mathcal{S}_n$ is a completely
dispersed state, its $b$ classes belong to $b$ distinct individuals
in the offspring generation and we can think of the chromosome
copies which belong to $\xi$ as the relevant ones (in the corresponding
individuals), i.e., the transition
from $\xi$ to $\eta$ corresponds to drawing $b$ times without replacement
from an urn which contains $\widehat{V}_i$ balls of colour $i$ for
$i=1,\dots,N$. Thus
\begin{align}
\label{eq:piNrepr1}
\pi_{n,N}({\xi,\eta}) = \sum_{i_1,\dots,i_{a-d}=1 \atop \text{distinct}}^N
\EE\Big[ \frac{\prod_{r=1}^{a-d} (\widehat{V}_{i_{r}})_{\ell_r}}{(N)_b}
\times \prod_{r=1}^{a-d} 2^{-\ell_r+1} \Big]
= 2^{-b+a-d} \frac{(N)_{a-d}}{(N)_b}
\EE\Big[ \prod_{r=1}^{a-d} (\widehat{V}_{r})_{\ell_r} \Big].
\end{align}
Note that the factor $\prod_{r=1}^{a-d} 2^{-\ell_r+1}$ accounts for the
fact that we still have to assign for each $r=1,\dots,a-d$ which of
the $\ell_r$ classes in $\cup_{i \in \zeta_r} D_i$ descends from which
of the two chromosome copies in the $i_r$-th ancestral individual
(this makes $\ell_r$ assigned picks if we decree who descends from the
``first'' and who from the ``second'' chromosome in individual $i_r$
but we gain a factor of $2$ because the roles of the ``first'' and the
``second'' chromosome are arbitrary and can be swapped). The second
equation is a consequence of exchangeability of the $\widehat{V}_i$. Finally,
\eqref{eq:piNrenorm} follows from \eqref{eq:piNrepr1} and
Lemma~\ref{lem:Vihatprod}.
\end{proof}

For $n\in\bN$, $\xi, \eta \in \mathcal{E}_n$, put
\begin{align}\label{eq:hatpisum}
\widetilde{\pi}_{n,N}(\xi,\eta) :=
\sum_{\left\{\eta' \in \mathcal{S}_n \, : \, \mathsf{cd}(\eta')=\eta\right\}}
\pi_{n,N}(\xi,\eta').
\end{align}
Note that $\widetilde{\pi}_{n,N}$ is a Markov transition matrix on
$\mathcal{E}_n$, a step according to $\widetilde{\pi}_{n,N}$ means first
taking a step according to $\pi_{n,N}$ and then applying the complete
dispersion operator, i.e., ignoring the grouping into diploid individuals.

In particular, there is sampling consistency:
For $\xi\in\mathcal{E}_n$, $\eta = \{ D_1,\dots,D_a \} \in \mathcal{E}_{n}$ we have
\begin{align}
\label{eq:pihatcons}
\widetilde{\pi}_{n,N}(\xi,\eta)
= \,\, & \widetilde{\pi}_{n+1,N}(\xi \cup \{\{n+1\}\},\eta \cup \{\{n+1\}\})
\notag \\
& {} + \sum_{i=1}^a \widetilde{\pi}_{n+1,N}\big(\xi \cup \{\{n+1\}\}, \{D_1,\dots,D_{i-1},
D_i \cup \{n+1\},D_{i+1},\dots,D_a\}\big).
\end{align}
Furthermore, the transition probabilities depend on $n$ only
implicitly through the merger structure that the transition from
$\xi$ to $\eta$ induces.

\begin{lemma}
\label{lem:piNcd}
Let $n\in\bN$, $\xi \in \mathcal{E}_n$, $\eta \in \mathcal{S}_n$ where
$\xi$ has $b$ classes and $\mathsf{cd}(\eta)$ arises from $\xi$
by merging $j \ge 1$ groups of classes with sizes $k_1, k_2, \dots, k_j \ge 2$
from $\xi$ and leaving $s \ge 0$ singleton classes (in particular,
$\eta$ has $a=j+s$ classes and $b=k_1+\cdots+k_j+s$). Then
\begin{align}
\label{eq:piNcdlim}
\lim_{N\to\infty} \frac1{c_N} \widetilde{\pi}_{n,N}\big(\xi, \mathsf{cd}(\eta) \big)
%\sum_{\left\{\eta' \, : \, \mathsf{cd}(\eta')=\mathsf{cd}(\eta)\right\}}
%\pi_N(\xi, \eta')
= \lambda_{b;k_1,\dots,k_j;s}
\end{align}
where $\lambda_{b;k_1,\dots,k_j;s}$ are the transition rates of the
$\Xi$-coalescent defined in Theorem~\ref{thm:res} (recalled in
\eqref{eq:Xitransrate}).
\end{lemma}
\begin{proof}
  Assume $\xi$ and $\eta$ are given by \eqref{xi} and \eqref{eta} with classes denoted by $C_i$ and $D_i,$
  respectively.    Recall that $D_i$ is a union of $k_i \ge 1$ classes from $\xi$ with
$k_1+\dots+k_a=b$ and that $\zeta$ describes the  grouping
into diploid individuals. %see ?

  Now note that $\mathsf{cd}^{-1}\(\mathsf{cd}\(\eta\)\) := \{ \eta' \in
  \mathcal{S}_n \, : \, \mathsf{cd}(\eta')=\mathsf{cd}(\eta)\}$ can be
  parametrised by choosing any $d \in \{0,1,\dots,\lfloor a/2
  \rfloor\}$ and $\zeta \in \mathcal{E}_{a,d}$ ($d$ describes the
  number of ancestral individuals in $\eta$ carrying two ancestral genes and $\zeta$
  describes the grouping of the $a$ ancestral chromosomes in $\eta$
  into diploid individuals).
  For $\zeta \in \mathcal{E}_{a,d}$ and $i=1,2,\ldots,a$,
  we define $\widehat{\zeta}(i) =k \in \{ 1,\dots, a-d\}$ if $i\in\zeta_k$.
  \smallskip

First consider the case $s=0$ and hence $a=j>1$:
From (\ref{eq:hatpisum}), Lemma~\ref{lem:piNrenorm}, as well as (\ref{eq:Vfmcond}) and (\ref{eq:Vfmrel}), %% {eq:piNrenorm},
\begin{align}
%\sum_{\left\{\eta' \, : \, \mathsf{cd}(\eta')=\mathsf{cd}(\eta)\right\}} & \lim_{N\to\infty} \frac1{c_N} \pi_N(\xi, \eta')
%% {\tt [prefactor ..]}
\nonumber
\lim_{N\to\infty} \frac1{c_N} \widetilde{\pi}_{n,N}\big(\xi, \mathsf{cd}(\eta) \big)
&
\nonumber
= \sum_{d=0}^{\lfloor a/2 \rfloor} \sum_{\zeta \in \mathcal{E}_{a,d}} 2^{-b+a-d}
\int_\Delta \sum_{\scriptstyle i_1,\dots,i_{a-d}=1 \atop
\scriptstyle \text{distinct}}^\infty \prod_{{\ell}=1}^{a-d}
x_{i_{\ell}}^{\big(\sum_{r \in \zeta_{\ell}} k_r\big)} \, \frac{2\Xi'(dx)}{(x,x)}
\notag \\
&
\nonumber
= \int_\Delta \sum_{d=0}^{\lfloor a/2 \rfloor} 2^{a-d}
\sum_{\scriptstyle i_1,\dots,i_{a-d}=1 \atop
\scriptstyle \text{distinct}}^\infty \sum_{\zeta \in \mathcal{E}_{a,d}}
\(\tfrac12 x_{i_{\widehat\zeta(1)}}\)^{k_1}
\(\tfrac12 x_{i_{\widehat\zeta(2)}}\)^{k_2} \cdots
\(\tfrac12 x_{i_{\widehat\zeta(a)}}\)^{k_a} \, \frac{2\Xi'(dx)}{(x,x)} \\
&
\nonumber
= \int_\Delta \sum_{\scriptstyle i'_1,\dots,i'_{a}=1 \atop
\scriptstyle \text{distinct}}^\infty
\( \varphi(x)_{i'_1} \)^{k_1} \( \varphi(x)_{i'_2} \)^{k_2}
\cdots \( \varphi(x)_{i'_a} \)^{k_a} \,
\frac{\Xi'(dx)}{(\varphi(x),\varphi(x))} \\
&
\label{eq:lim_pi_tilde}
= \int_\Delta \sum_{\scriptstyle i'_1,\dots,i'_{a}=1 \atop
\scriptstyle \text{distinct}}^\infty x_{i'_1}^{k_1} \cdots x_{i'_a}^{k_a} \,
\frac{\Xi(dx)}{(x,x)} = \lambda_{b;k_1,\dots,k_a;0}
\end{align}
where we used that by definition of $\varphi$, for any
function $F : [0,1]^a \to [0,\infty)$ and $x=(x_1,x_2,\dots) \in \Delta$
\begin{align}
\nonumber
  \sum_{\scriptstyle i'_1,\dots,i'_{a}=1 \atop
    \scriptstyle \text{distinct}}^\infty
  F\( \varphi(x)_{i'_1}, \varphi(x)_{i'_2}, \dots, \varphi(x)_{i'_a} \)
  = \sum_{d=0}^{\lfloor a/2 \rfloor} 2^{a-d}
  \sum_{\scriptstyle i_1,\dots,i_{a-d}=1 \atop
  \scriptstyle \text{distinct}}^\infty \sum_{\zeta \in \mathcal{E}_{a,d}}
F\( \frac{1}{2} x_{i_{\widehat\zeta(1)}},  \frac{1}{2} x_{i_{\widehat\zeta(2)}}, \dots,  \frac{1}{2} x_{i_{\widehat\zeta(a)}} \)
\end{align}
and that $(\varphi(x),\varphi(x)) = \tfrac12 (x,x)$. In the case $j=1,$ we additionally have the Kingman
term in (\ref{eq:lim_pi_tilde}) of the form
\begin{equation*}
{\ind}_{\{j=1,k_1=2\}}\  2 \ \Xi'(\{\mathbf{0}\}) 2^{-b+a-d}= {\ind}_{\{j=1,k_1=2\}}  \Xi(\{\mathbf{0}\})
\end{equation*}
since $-b+a-d=-1$ in this case and $\Xi'(\{\mathbf{0}\})=\Xi(\{\mathbf{0}\}).$ Thus, \eqref{eq:piNcdlim} holds when $s=0$.
\smallskip

For the general case $s>0$, we can employ the consistency relations (\ref{eq:pihatcons}) in order to use induction on $s:$
%Claim \eqref{eq:piNcdlim} is that for any $n\in \bN$, $\lim_{N\to\infty} \widetilde{\pi}_{n,N}\big(\xi, \mathsf{cd}(\eta) \big)/c_N = \lambda_{b;k_1,\dots,k_j;s}$ whenever $\mathsf{cd}(\eta)$ arises from $\xi$ by a merger in $j$ groups of sizes $k_1,\dots,k_j\ge 2$, leaving $s \in \bN_0$ singleton classes which do not participate in any merger. The case $s=0$ is treated above. We check the general case by induction on $s$.
%
Assume that \eqref{eq:piNcdlim} holds whenever the number of
``singleton classes'' involved is at most $s$, and $b=k_1+\cdots+k_j+s$ with $k_1,\dots,k_j\ge
2$. Let
$\mathsf{cd}(\eta) \in \mathcal{E}_{n+1}$ arise from $\xi \in
\mathcal{E}_{n+1}$ by a merger in $j$ groups of sizes $k_1,\dots,k_j\ge
2$, leaving $s+1$ singleton classes. By the symmetries of the model,
we may (without changing the transition probability) assume that one
of the relevant singleton classes in $\eta$ is $\{n+1\}$. Then,
rearranging \eqref{eq:pihatcons} and using the induction hypothesis we
see that
\begin{equation*}
\lim_{N\to\infty}
\frac{\widetilde{\pi}_{n+1,N}\big(\xi, \mathsf{cd}(\eta) \big)}{c_N}
= \lambda_{b;k_1,\dots,k_j;s} - \sum_{i=1}^j \lambda_{b+1;k_1,\dots,k_i+1,\dots,k_j;s}
- s \lambda_{b+1;k_1,\dots,k_j,2;s-1}.
\end{equation*}
The term on the right-hand side equals $\lambda_{b+1;k_1,\dots,k_j;s+1}$
by the consistency relation for transition probabilities of
$\Xi$-coalescents (implicit in \cite[Eq.~(11) and Lemma~3.4]{Mohle2001},
explicitly spelled out for example in \cite[Eq.~(2.5)]{Sagitov2003}).
\end{proof}
\smallskip

A central ingredient in the proof of Theorem~\ref{thm:res} is the
following result from \cite{Mohle98a}.
It makes the separation of time-scales behind Theorem~\ref{thm:res}
explicit: On the fast time scale $O(1)$, any diploid sample configuration
is transformed into its complete dispersion whereas on the much slower
time scale $O(1/c_N)$, non-trivial merging occurs.

\begin{lemma}[M\"ohle \cite{Mohle98a}]\label{th:converge}
Let $X_N=\(X_N\(m\)\)_{m\in\mathbb{N}_0}$ be a sequence of time homogeneous Markov chains on a probability space $\(\Omega,\mathcal{F},\mathbb{P}\)$ with the same finite state space $\mathcal{S}$ and let $\mathbf{\Pi}_N$ denote the transition matrix of $X_N$. Assume that the following conditions are satisfied:
\begin{enumerate}[1.]
\item $A:=\lim_{N\rightarrow\infty}\mathbf{\Pi}_N$ exists and $\mathbf{\Pi}_N\neq A$ for all sufficiently large $N$.
\item $P:=\lim_{m\rightarrow\infty}A^m$ exists.
\item $G:=\lim_{N\rightarrow\infty}PB_NP$ exists, where $B_N:=\(\mathbf{\Pi}_N-A\)/c_N$ and $c_N:=\|\mathbf{\Pi}_N-A\|$ for all $N\in\mathbb{N}$.
\end{enumerate}
If the sequence of initial probability measures $\mathbb{P}_{X_N\(0\)}$ converge weakly to some probability measure $\mu$, then the finite dimensional distributions of the process $\(X_N\(\left\lfloor t/c_N\right\rfloor\)\)_{t\geq 0}$ converge to those of a time continuous Markov process $\(X_t\)_{t\geq 0}$ with initial distribution
$$X_0\mathop{=}^d\mu,$$
transition matrix $\mathbf{\Pi}\(t\):=P-I+e^{tG}=Pe^{tG}$, $t>0$, and infinitesimal generator $G$.
\end{lemma}

\begin{proof}[Proof of Theorem \ref{thm:res}:]
Applying Lemma~\ref{th:converge}, the strategy to prove our result is based on the decomposition of the transition matrix
$\Pi_{n,N}=(\pi_{n,N}(\xi^{'},\eta^{'}))_{\xi^{'},\eta^{'}\in\mathcal{S}_n}$. In order to have these transitions be well defined as matrices we choose a specific order of $\mathcal{S}_n:$ Namely, consider the
standard order of $\mathcal{E}_n \subset \mathcal{S}_n.$ We will then insert all remaining elements of $\mathsf{cd}^{-1}(\xi) \subset \mathcal{S}_n$ directly following any $\xi \in \mathcal{E}_n$ (the order here is fixed in an arbitrary way).
%for example in an ascending order for the number of ancestral individuals with two genes in the sample.
%And then by ordering all the classes (pairs of classes) by the minimal element contained in them.
This way, the matrix ${\Pi}_{n,N}$ decomposes into sub-matrices
$(\tilde{\Pi}_{n,N}(\xi,\eta))_{{\xi},{\eta} \in\mathcal{E}_n}$ with $\tilde{\Pi}_{n,N}({\xi},{\eta})$ a $\left|  \mathsf{cd}^{-1}(\xi) \right|\times\left|{\mathsf{cd}^{-1}(\eta)}\right|$ matrix.

What we need to do is to find the decomposition such that $${\Pi}_{n,N}=A+c_NB_{n,N},$$
where $A=\lim_{N\rightarrow\infty}{\Pi}_{n,N}$ does not depend on $N$, $c_N\to0$ and $B_{n,N}$ is bounded.
In the sense of sub-matrix structure, it is necessary to find the decomposition such that
%${\Pi}_N$ can be decomposed into
%$\left|\mathcal{E}_n\right|\times\left|\mathcal{E}_n\right|$ small sub-matrix %such that
$$\tilde{\Pi}_{n,N}\({\xi},{\eta}\)=\tilde{A}\({\xi},{\eta}\)+c_N \tilde{B}_{n,N}\({\xi},{\eta}\),$$
for ${\xi},{\eta}\in\mathcal{E}_n,$ where both $\tilde{A}\({\xi},{\eta}\)$ and $\tilde{B}_{n,N}\({\xi},{\eta}\)$ are
$\left|  \mathsf{cd}^{-1}(\xi) \right|\times\left|{\mathsf{cd}^{-1}(\eta)}\right|$ sub-matrices of $A$ and $B_{n,N}.$
Set $\tilde{A}\({\xi},{\eta}\)=0$ for ${\xi}\neq{\eta}$ and
\begin{equation*}
\begin{split}
\tilde{A}\({\xi},{\xi}\)=\left(
                                                      \begin{array}{cccc}
                                                        1 & 0 & \cdots & 0 \\
                                                        \vdots &  \vdots&  & \vdots \\
                                                        1 & 0 & \cdots & 0 \\
                                                      \end{array}
                                                    \right):=P_{{\xi}}.
\end{split}
\end{equation*}
Note that this matrix maps any $\xi' \in  \mathsf{cd}^{-1}(\xi)$ to $\xi \in \mathcal{E}_n.$
It follows that
\begin{equation*}
\begin{split}
\lim_{m\rightarrow\infty}\tilde{A}^m\({\xi},{\xi}\)=\tilde{A}\({\xi},{\xi}\):=P_{{\xi}}.
\end{split}
\end{equation*}
Thus, we have
%\begin{equation*}
$P:=\lim_{m\rightarrow\infty}A^m$ %=\(P_{{\xi},{\eta}}\)_{{\xi},{\eta}\in\mathcal{E}_n}$
%\end{equation*}
with sub matrix structure $\tilde{P}({\xi},{\eta})=\mathbf{0}$ for ${\xi}\neq{\eta}$ and $\tilde{P}({\xi},{\xi})=P_{{\xi}}$.
It is easy to show that
\begin{equation*}
\begin{split}
G&:=\lim_{N\rightarrow\infty}PB_{n,N}P
=\lim_{N\rightarrow\infty}\(P_{{\xi}}B_{n,N}\({\xi},{\eta}\)P_{{\eta}}\)_{{\xi},{\eta}\in\mathcal{E}_n}\\
&=\lim_{N\rightarrow\infty}\left(\left(
                                   \begin{array}{cccc}
                                     g^{\(N\)}_{{\xi},{\eta}} & 0 & \cdots & 0 \\
                                     \vdots & \vdots  & &\vdots\\
                                     g^{\(N\)}_{{\xi},{\eta}} & 0 & \cdots & 0 \\
                                   \end{array}
                                 \right)
\right)_{{\xi},{\eta}\in\mathcal{E}_n}
\end{split}
\end{equation*}
where $g^{\(N\)}_{{\xi},{\eta}}$ is the sum of all the entries in the first row of matrix $B_{n,N}\({\xi},{\eta}\)$.
Consequently,
\begin{equation*}%\label{1406113}
\begin{split}
g^{\(N\)}_{{\xi},{\eta}}
&=\sum_{ \eta^{'}\in \mathsf{cd}^{-1}(\eta)}\frac{{\pi}_{n,N}{({\xi},\eta^{'})}}{c_N}-\frac{\delta_{{\xi},{\eta}}}{c_N}. %\\
%&=c_N^{-1}\(\sum_{{\left\{\eta^{'}\in \mathsf{cd}^{-1}(\eta)\right\}}}{\pi}_{n,N}{\({\xi},{\eta^{'}}\)}-\delta_{{\xi},{\eta}}\).\\
\end{split}
\end{equation*}
Assume ${\eta}$ arises from ${\xi}$ by merging $j \ge 1$ groups of classes with sizes $k_1, k_2, \dots, k_j \ge 2$
and leaving $s \ge 0$ singleton classes ($b=k_1+\cdots+k_j+s$). Applying Lemma \ref{lem:piNcd}, we have $$\lim_{N\rightarrow\infty}g^{\(N\)}_{{\xi},{\eta}}=\lambda_{b;k_1,\dots,k_j;s}.$$
Note that a transition from any given state $\eta^{'} \in \mathcal{S}_n$ with $a\leq n$ classes (as in  (\ref{eta})) to its complete dispersion state $\mathsf{cd}(\eta^{'})$ happens whenever none of the ancestral genes of distinct ancestral individuals
in configurations $\eta^{'}$ have a common parental ancestral gene in the previous generation.
(Ancestral genes of the same individual naturally have distinct parental ancestral genes as we have excluded selfing.)
For any pair of such ancestral genes the probability to have a common parental ancestral gene is $c_N$ and
 there are at most $\binom{a}{2}$ such pairs %of ancestral genes
 to consider. Thus, the transition probability satisfies
\begin{equation*}
\pi_{n,N}{(\eta^{'},\mathsf{cd}(\eta^{'}))}
\geq 1-{a\choose 2}c_N.
\end{equation*}
It is clear that $\pi_{n,N}{(\eta^{'} ,\mathsf{cd}(\eta^{'}))}\rightarrow 1$ as $N\rightarrow\infty$ or in other words that
$A=\lim_{N\rightarrow\infty}{\Pi}_{n,N}.$
Hence, complete dispersion happens instantaneously in the limit of the genealogical process for the diploid population model.
By eliminating all those instantaneous states, we can get an $\mathcal{E}_n$-valued marginal process $\({R}_n\(t\)\)_{t\geq 0}$
whose generator is given by $\(\lim_{N\rightarrow\infty}g^{\(N\)}_{{\xi},{\eta}}\)_{{\xi},{\eta}\in\mathcal{E}_n}$. The process $\({R}_{n}\(t\)\)_{t\geq 0}$ is exactly the {$n$}-$\Xi$-coalescent process.
\end{proof}

%\section{Summary of pertinent $\Xi$-coalescent results}
%{\tt tbw...........}

\begin{proof}[Proof of Corollary \ref{cor:res}:]
Our argument is essentially borrowed from the proof of Theorem~3.1 in M\"ohle\ \cite{Mohle99}.
Theorem~\ref{thm:res} yields %in particular
that the finite-dimensional distributions of
\[ \big( \widetilde{\xi}^{n,N}( \lfloor t/c_N \rfloor) \big)_{t\ge0} =
\big( \mathsf{cd}\big(\xi^{n,N}( \lfloor t/c_N \rfloor) \big) \big)_{t\ge0} \]
converge to those of $\xi^n$. Thus, in order to strengthen this to weak convergence on
the path space $D([0,\infty), \mathcal{E}_n)$ we only have to verify tightness
there. Since $\mathcal{E}_n$ is finite (and so in particular compact)
and $\widetilde{\xi}^{n,N}$ can by construction only move by mergers,
it suffices to check that in the limit $N\to\infty$, the jump times of
$\big( \mathsf{cd}\big(\xi^{n,N}( \lfloor \cdot/c_N \rfloor)\big) \big)$ do not accumulate
(see e.g.\ \cite[Thm.~6.2 in Ch.~3]{EK86} or \cite{Billing}).
Noting that for any $\xi \in \mathcal{S}_n$
\begin{align*}
  & \PP\Big(\mathsf{cd}\big(\xi^{n,N}(m+1)\big) \neq \mathsf{cd}\big(\xi\big) \, \Big| \, \xi^{n,N}(m) = \xi \Big) \\
  =& \PP\Big( \, \parbox{19em}{at least one pair of genes (necessarily from \\distinct individuals) merges in the next step} \: \Big| \, \xi^{n,N}(m) = \xi \Big) %\\ &
  \le {n \choose 2} c_N
\end{align*}
we see that the times between jumps of $\big( \widetilde{\xi}^{n,N}(m)\big)_{m \in \NN_0}$ are stochastically
larger than independent geometric random variables with success parameter $c_N n(n-1)/2$ which
after time scaling converge in distribution to independent exponentials with rate $n(n-1)/2$.
\end{proof}

%%%%%***

\subsection{Proof of Proposition~\ref{Exrff:prop1}}
\label{sProofRandomIndFitness}

In this section we prove Lemma~\ref{Exrff:lem:cN} and Proposition~\ref{Exrff:prop1},  the main convergence result for the
diploid population model with random individual fitness of Section~\ref{Section:Exrff}. Apart from Lemma~\ref{Exrff:lem:cN}
we need two auxiliary lemmas for the proof of Proposition~\ref{Exrff:prop1}. Their proofs are postponed to Section \ref{sAuxiliary}.

Note that while the set-up in Proposition~\ref{Exrff:prop1} is quite similar in spirit to that from
\cite{Schweinsberg2003}, the
adaptation %is not completely straightforward
poses some additional technical difficulties. In particular, observe that
$Z_N$, the normalising constant in the representation
\eqref{Exrff:eq:lawVij} of the law of $(V_{i,j})$ as a mixture of
multinomials is not  literally an i.i.d.\ sum as in
\cite{Schweinsberg2003}.

\smallskip
Recall from (\ref{Exrff:defZN}) and (\ref{Exrff:eq:lawVij}) that in the random individual fitness model
$V_{i,j},1\le i < j \le N$ are multinomial with $N$ trials and with success probabilities of individuals $i$ and $j$ proportional to
$W_i W_j$ where the fitness parameters $W_i$ are independent copies of $W$ with mean $\mu_W.$
\begin{lemma}
  \label{Exrff:lem:highmombd}
  1.\ If $\mu_W^{(2)} = \EE\left[ W^2 \right] < \infty$ we have
  \begin{align}
    \label{Exrff:eq:phi13zero}
    \phi_1(3) = \lim_{N\to\infty} \frac{\EE\big[ (V_1)_{3} \big]}{ c_N N^2} = 0.
  \end{align}
  2.\ If \eqref{Exrff:eq:tailassumptW1} holds we have
  \begin{align}
    \label{Exrff:eq:phi22zero}
    \phi_2(2,2) = \lim_{N\to\infty} \frac{\EE\big[ (V_1)_{2} (V_2)_{2}
      \big]}{c_N N^2} = 0.
  \end{align}
\end{lemma}

\begin{lemma}
  \label{Exrff:lem:V1tail}
  If \eqref{Exrff:eq:tailassumptW1} holds we have
  \begin{align}
    \label{Exrff:eq:V1asympttail}
    \frac{N}{c_N} \mathbb{P}(V_1 > N x) \mathop{\longrightarrow}_{N\to\infty}
    8 \int_x^1 \frac{1}{y^2} \, \mathrm{Beta}(2-\alpha,\alpha)(dy) \quad \text{for } x \in (0,1).
  \end{align}
\end{lemma}

\begin{proof}[Proof of Proposition~\ref{Exrff:prop1}] 1.\ The scaling of the pair coalescence
  probability $c_N$ is given in \eqref{Exrff:eq:cNKingman} in Lemma~\ref{Exrff:lem:cN}.

  Using Theorem~\ref{thm:res} we should verify that $\mu_W^{(2)} < \infty$ implies that (cf.\ Condition~\eqref{eq:Vfmcond})
  for all $j \in \NN$ and $k_1,\dots,k_j \ge 2$
  \begin{align}
    \label{Exrff:eq:VkmomlimitKingman}
    \lim_{N\rightarrow\infty}
    \frac1{c_N} \frac{\EE\big[\(V_1\)_{k_1}\cdots\(V_j\)_{k_j}\big]}{N^{k_1+\cdots+k_j-j}2^{k_1+\cdots+k_j}}
    = 2 \ind_{\{j=1,k_1=2\}} .
  \end{align}
  For $j=1, k_1=2$ this follows from the fact that $c_N = \EE[V_1 (V_1-1)]/(8(N-1))$ (see Lemma \ref{lem:cN});
  for $j=1, k_1=3$ it follows from Lemma~\ref{Exrff:lem:highmombd}; it is well known that the latter
  implies that \eqref{Exrff:eq:VkmomlimitKingman} also holds for $j=1$, $k_1>3$ and for $j\ge 2$
  (see \cite{Mohle2001} and \cite{Sagitov2003}).
  \medskip

  \noindent 2.\ In this case, the scaling of $c_N$ is given by \eqref{Exrff:eq:cNBeta}
  in Lemma~\ref{Exrff:lem:cN}. To verify the claimed form of the limiting coalescent
  we should check that the probability measure $\Xi'$ on $\Delta$ appearing in \eqref{eq:PhiNconv} and \eqref{eq:Vfmrel}
  is given by
  \begin{align*}
    %\label{Exrff:eq:limitXi'}
  \Xi'(A) = \int_0^1 \ind_A(y/2,0,0,\dots) \, \frac{y^{1-\alpha} (1-y)^{\alpha-1}}{\Gamma(2-\alpha) \Gamma(\alpha)} \, dy,
  \end{align*}
  i.e.\ $\Xi'$ is the image measure of $\mathrm{Beta}(2-\alpha,\alpha)$ under the mapping $[0,1] \ni x \mapsto (x/2,0,0,\dots)
  \in \Delta$.
  %% \note{ ... note that $(V_1,\dots,V_N)$ form an exchangeable array with total sum $V_1+\cdots+V_N=2N$ ... }

  It is known that \eqref{Exrff:eq:phi22zero} from
  Lemma~\ref{Exrff:lem:highmombd} implies that the vague limit measure
  of $\Phi_N$ from \eqref{eq:PhiNconv} is concentrated on
  $\widetilde\Delta := \{ (x_1,x_2,\dots) \in \Delta: x_2=0\}$, see
  \cite[Cor.~2.1]{Sagitov2003} (one can view $\widetilde\Delta$ as the
  canonical embedding of $[0,1]$ into $\Delta$) so it suffices to observe
  that for any $x \in (0,1/2)$ by Lemma~\ref{Exrff:lem:V1tail}
  \begin{align*}
    \lim_{N\to\infty} & \frac{1}{c_N} \PP\big( V_{(1)} > 2N x \big) =
    \lim_{N\to\infty} \frac{N}{c_N} \PP\big( V_1 > 2N x \big) = 8 \int_{2x}^1 \frac{1}{y^2} \, \mathrm{Beta}(2-\alpha,\alpha)(dy)
    \notag \\
    & = 2 \int_{2x}^1 \frac{1}{(y/2)^2} \, \mathrm{Beta}(2-\alpha,\alpha)(dy)
    = 2 \int_0^1 \ind_{A(x)}(y/2,0,0,\dots) \frac{1}{y^2/4} \, \frac{y^{1-\alpha} (1-y)^{\alpha-1}}{\Gamma(2-\alpha) \Gamma(\alpha)} \, dy
  \end{align*}
  with $A(x) = \{ (y_1,y_2,\dots) \in \Delta : y_1 > x, y_2=0 \}$.
\end{proof}

\subsubsection{Proofs of auxiliary results}
\label{sAuxiliary}
Here, we provide some details of the proofs of
Lemmas~\ref{Exrff:lem:cN}, \ref{Exrff:lem:highmombd} and
\ref{Exrff:lem:V1tail}.

\begin{lemma}
  \label{Exrff:lemWdbWpMmom}
  Assume that $W \ge 0$ satisfies the tail assumption \eqref{Exrff:eq:tailassumptW1}. For $k \in \{2,3,4,\dots\}$,
  we have
  \begin{align}
    \label{Exrff:WdbWpMmom}
    \lim_{M\to\infty} M^{\alpha} \EE\Big[ \frac{W^k}{(W+M)^k} \Big] = c_W \alpha B(k-\alpha,\alpha).
  \end{align}
  \end{lemma}
This is a small variation on Lemma~12 from \cite{Schweinsberg2003},
addressing the case $k=2$ and $W$ integer-valued. For a rough idea
why the asymptotic decay rate of $\EE\big[W^k/(M+W)^k\big]$ is
$M^{-\alpha}$ note that on the event $\{W \ge M\}$, which has
probability $\sim c_W M^{-\alpha}$, the integrand is almost constant.
\begin{proof}
  For any bounded monotone $g \in C^1([0,\infty))$ with $g(0)=0$  we have
  \begin{align*}
    \EE[g(W)] & = \EE\Big[ \int_0^\infty \ind(x \le W) g'(x) \,dx \Big]
    = \int_0^\infty g'(x) \PP(W \ge x) \, dx.
  \end{align*}
  Applying this with $g(x) = x^k/(x+M)^k$ hence $g'(x) = k M x^{k-1}/(x+M)^{k+1}$ we obtain
  \begin{align}
    \label{Exrff:lemWdbWpMmom.e1}
    \EE\Big[ \frac{W^k}{(W+M)^k} \Big] =  \int_0^\infty g'(x) \PP(W \ge x) \, dx = \int_0^\infty M k \frac{x^{k-1}}{(x+M)^{k+1}} \PP(W \ge x) \, dx.
  \end{align}
  For every $L>0$ we have
  \begin{align}
    \label{Exrff:lemWdbWpMmom.e2}
    \limsup_{M\to\infty} M^\alpha \int_0^L  g'(x) \PP(W \ge x) \, dx %M k \frac{x^{k-1}}{(x+M)^{k+1}} \PP(W \ge x) \, dx
    & \le \limsup_{M\to\infty} M^\alpha \int_0^L g'(x) \, dx \notag \\
    & = \lim_{M\to\infty} M^\alpha \frac{L^k}{(L+M)^k} = 0,
  \end{align}
  using that $\alpha < 2$ and that $k \geq 2.$
  Furthermore,
  \begin{align*}
    \int_L^\infty \frac{x^{k-1}}{(x+M)^{k+1}} x^{-\alpha} \, dx
    & = \int_0^{M/(M+L)} \( \frac{M(1-y)}{y} \)^{k-1-\alpha} \Big( \frac{y}{M}\Big)^{k+1} \, M y^{-2} \, dy \\
    & = M^{-\alpha-1} \int_0^{M/(M+L)} (1-y)^{k-1-\alpha} y^{\alpha} \, dy
  \end{align*}
  (we substituted $y=M/(M+x)$, hence $x=M(1-y)/y$, $dx/dy = - M y^{-2}$ for the first equation).
  Thus,
  \begin{align}
    \label{Exrff:lemWdbWpMmom.e3}
    \lim_{M\to\infty} M^\alpha \int_0^\infty M k \frac{x^{k-1}}{(x+M)^{k+1}} x^{-\alpha} \, dx
    & = k \int_0^1 (1-y)^{k-1-\alpha} y^{\alpha} \, dy \notag \\
    & = \frac{k \Gamma(\alpha+1) \Gamma(k-\alpha)}{\Gamma(k+1)}
    = \alpha \frac{\Gamma(\alpha) \Gamma(k-\alpha)}{\Gamma(k)} = \alpha B(k-\alpha,\alpha)
  \end{align}

  For $\varepsilon > 0$ we can choose $L$ so large that
  $(1-\varepsilon) c_W x^{-\alpha} \le \PP(W \ge x) \le
  (1+\varepsilon) c_W x^{-\alpha}$ holds for all $x \ge L$. Combining
  \eqref{Exrff:lemWdbWpMmom.e1}--\eqref{Exrff:lemWdbWpMmom.e3} we see that
  \begin{align*}
    \limsup_{M\to\infty} M^\alpha \EE\Big[ \frac{W^k}{(W+M)^k} \Big] \le (1+\varepsilon) c_W \alpha B(k-\alpha,\alpha)
  \end{align*}
  and similarly for the $\liminf$. \eqref{Exrff:WdbWpMmom} follows by taking $\varepsilon \downarrow 0$.
\end{proof}

  Defining
  \begin{align}
    \label{Exrff:defSN.SN2}
    S_{k,N} := \sum_{j=k}^N W_j, \quad S^{(2)}_{k,N} := \sum_{j=k}^N
    W^2_j
  \end{align}
  (for $2\le  k < N$)
  we can re-express $Z_N$ from \eqref{Exrff:defZN} as
  \begin{align}
    \label{Exrff:eq:ZNre-expr.2}
    Z_N & = W_1 S_{2,N} + \frac12 \big( (S_{2,N})^2 - S^{(2)}_{2,N} \big) \\
    \intertext{and also as}
    \label{Exrff:eq:ZNre-expr}
    Z_N & = \frac{1}{2} \big(W_1+W_2+S_{3,N}\big)^2 - \frac12 W_1^2 -
    \frac12 W_2^2 - \frac12 S^{(2)}_{3,N} \notag \\
    & = W_1W_2 + (W_1+W_2) S_{3,N} + \frac12 \big( (S_{3,N})^2 - S^{(2)}_{3,N} \big) .
  \end{align}

  \begin{lemma}
    \label{Exrff:SNtooextreme}
    Assume that $W$ satisfies \eqref{Exrff:eq:tailassumptW1}
    or that $\mu_W^{(2)} < \infty$.
    For $0 < \delta < 1$ and $k \in \{2,3\}$ let
  \begin{align*}
    A_\delta := \big\{ S_{k,N} < (1-\delta) \mu_W N \big\} \cup \big\{ (S_{k,N})^2 - S^{(2)}_{k,N} < (1-\delta) \mu_W^2 N^2 \big\}
  \end{align*}
  then there exists $r=r(\delta)>0$ such that
  \begin{align}
    \label{Exrff:eq:PAdelta}
    \PP(A_\delta) \le e^{-r N} \quad \text{for all $N$ large enough}.
  \end{align}
  Furthermore, for all $r=r(\delta)>0$ with $\delta \in (0,1)$ we have
  \begin{align}
    \label{Exrff:eq:SN2tooextreme}
    \PP\big( (S_{k,N})^2 - S^{(2)}_{k,N} < \mu_W N S_{k,N}/2 \big) \le  e^{-r N} \quad \text{for all $N$ large enough}.
  \end{align}
  \end{lemma}
  For \eqref{Exrff:eq:SN2tooextreme} (under Assumption~\eqref{Exrff:eq:tailassumptW1}) note that while the typical size of $S_{k,N}$ is $\approx \mu_W N$
  by the law of large numbers, conditioned on $S_{k,N} \gg \mu_W N$ there will typically be just one exceptionally
  large summand (by the tail assumption~\eqref{Exrff:eq:tailassumptW1}, this is much more
  likely than having many moderately large summands, cf.\ \cite{Nagaev82}). Then, the
  order of magnitude of $(S_{k,N})^2 - S^{(2)}_{k,N}$ will in fact be $\approx N S_{k,N}$ up to constants.

  \begin{proof}
  Write $W_i = U_i + O_i$ with $U_i := W_i \ind(W_i \le
  K), O_i := W_i \ind(W_i > K)$ where $K$ is chosen so large that
  $\EE[U_i] > (1-\delta/5) \mu_W$.  Since $U_i \ge 0$ are i.i.d. bounded random variables,
  %standard estimates from the theory of large deviations (cf.\ e.g.\ the proof of Lemma~5 in \cite{Schweinsberg2003})  show that
 we get from Cram\'er's large deviations theorem (see, for example, Theorem 2.2.3 in \cite{DemboZeitouni98}) that
  \begin{align}
    \label{Exrff:eq:PS3Ntoosmall}
   \PP\( \sum_{i=k}^N U_i < (1-\delta/4) \, \mu_W N \)+ \PP\( \sum_{i=k}^N U_i > (1+\delta/4) \, \mu_W N \) \le e^{-r N}
  \end{align}
  for all $N$ large enough with $r=r(\delta)>0$. We now argue that
  $A_{\delta} \subset \{\sum_{i=k}^N U_i < (1-\delta/4) \, \mu_W N\}$ for $N$ large enough: Obviously,
  $\{S_{k,N} < (1-\delta) \, \mu_W N \}\subset \{\sum_{i=k}^N U_i < (1-\delta/4)  \, \mu_W N \}.$
  On the event $\big\{ \sum_{i=k}^N U_i \ge (1-\delta/4) \, \mu_W N \big\}$ we have
  \begin{align*}
    %\label{Exrff:eq:PS3qmS2toosmall}
    \notag
    (S_{k,N})^2 - S^{(2)}_{k,N} & \ge \Big( (1-\delta/4) \, \mu_W N + \sum_{i=k}^N O_i \Big)^2
    - N K^2 - \sum_{i=k}^N O_i^2 \\
    & \ge (1-\delta/4)^2 \mu_W^2 N^2 - N K^2
    = \Big( 1- \frac{\delta}{2} + \frac{\delta^2}{16} - \frac{K^2}{\mu_W^2 N}\Big) \mu_W^2 N^2
    \ge (1-\delta) \mu_W^2 N^2
  \end{align*}
  for $N$ large enough. Thus,  also $\{(S_{k,N})^2 - S^{(2)}_{k,N} < (1-\delta) \mu_W^2 N^2\} \subset
  \{\sum_{i=k}^N U_i < (1-\delta/4) \, \mu_W N\}$ %by  \eqref{Exrff:eq:PS3Ntoosmall}
  and  we have that  \eqref{Exrff:eq:PAdelta} follows from \eqref{Exrff:eq:PS3Ntoosmall}.
  \medskip

%  Since the $U_i$ are i.i.d.\ and bounded with $\EE[U_1] < \mu_W$ we have
 % \begin{align}
 % \label{Exrff:eq:PS3Ntoolarge}
  %  \PP\Big( \sum_{i=k}^N U_i > (1+\delta/4) \mu_W N \Big) \le e^{-r N}
 % \end{align}
  %for $N$ large [put ref.\ e.g.\ to Cram\'er's theorem in Dembo \& Zeitouni or similar] {\tt Combine this with \eqref{Exrff:eq:PS3Ntoosmall}?}.
  On the event $\big\{ 1-\delta/4 \le (\mu_W N)^{-1} \sum_{i=k}^N U_i \le 1+\delta/4 \big\}$
  we can estimate similarly to the above for all $\delta \in (0,1):$
  \begin{align*}
  \nonumber
  (S_{k,N})^2 - S^{(2)}_{k,N} & \ge \( (1-\delta/4) \, \mu_W N + \sum_{i=k}^N O_i \)^2
  - N K^2 - \sum_{i=k}^N O_i^2 \\
  \nonumber
  & = (1-\delta/4)^2 \mu_W^2 N^2 + 2 (1-\delta/4) \mu_W N \sum_{i=k}^N O_i
  + \( \sum_{i=k}^N O_i \)^2 - N K^2 - \sum_{i=k}^N O_i^2 \\
  & \ge \frac{\mu_W N}{2} \( (1+\delta/4) \mu_W N + \sum_{i=k}^N O_i \)
  \ge \frac{\mu_W N}{2} \( \sum_{i=k}^N W_i \) = \frac{\mu_W N}{2} S_{k,N}.
  \end{align*}
  Thus, \eqref{Exrff:eq:SN2tooextreme} follows from \eqref{Exrff:eq:PS3Ntoosmall}. %and \eqref{Exrff:eq:PS3Ntoolarge}.
  \end{proof}

  For ease of reference we recall here a classical fact about multinomal distributions
  % (see \cite{NeNe}).
  (see, e.g., Formula~(35.5) in \cite{JKB97}).
  \begin{lemma}
    \label{lem:Multinomfactorialmoments}
    For $Y=(Y_1,\dots,Y_m) \sim \mathrm{Multinomial}(N,p_1,p_2,\dots,p_m)$ and
    $n_1, n_2,\dots,n_m \in \NN_0$ we have
    \begin{align}
      \label{Exrff:eq:multinommoments}
      \EE\left[ (Y_1)_{n_1} (Y_2)_{n_2} \dots (Y_m)_{n_m} \right]
      = (N)_{n} p_1^{n_1} p_2^{n_2} \cdots p_m^{n_m}
    \end{align}
    with $n=n_1+n_2+\cdots+n_m$.
  \end{lemma}

\begin{proof}[Proof of Lemma~\ref{Exrff:lem:cN}]
  To verify \eqref{Exrff:eq:cNformula} observe that from \eqref{Exrff:eq:lawVij} with $Q_N$ from \eqref{def:QN}
  \[
  %\cL\left( V_1 \,\big|\, (W_i ) \right) = \text{Bin}\left( N, \frac{W_1 \sum_{j=2}^N W_j}{Z_N}\right)
  \cL\left( V_1 \,\big|\, (W_i ) \right) = \text{Bin}\left( N, {Q_N}\right)
  \]
  hence using \eqref{Exrff:eq:multinommoments} from Lemma~\ref{lem:Multinomfactorialmoments}
  \begin{align*}
    \EE\big[V_1 (V_1-1) \mid (W_i ) \big] & =N(N-1){Q_N^2}.
    %N(N-1) \frac{W_1^2 \big( \sum_{j=2}^N W_j\big)^2}{Z_N^2}.
  \end{align*}
  \eqref{Exrff:eq:cNformula} follows
  from this via the formula $c_N = \EE[V_1 (V_1-1)]/(8(N-1)),$ see Lemma \ref{lem:cN}.
  \medskip

  We now assume that \eqref{Exrff:eq:tailassumptW1} holds.
  In order to prove \eqref{Exrff:eq:cNBeta} we first verify that
  \begin{align}
    \label{Exrff:eq:EWquotcN.limsup}
    %\limsup_{N\to \infty} N^{\alpha} \EE\left[ \frac{W_1^2 \big( \sum_{j=2}^N W_j \big)^2}{Z_N^2}\right]
    \limsup_{N\to \infty} N^{\alpha} \EE\left[{Q_N^2}\right]
    \le 8 C_\mathrm{pair}^{(\mathrm{Beta})}.
  \end{align}
  Note that using \eqref{Exrff:eq:ZNre-expr.2} we can re-write
  \begin{align}
    \label{Exrff:eq:rewrite.pV1}
    %\frac{W_1 \sum_{i=2}^N W_j}{Z_N}
   { Q_N}= \frac{W_1 S_{2,N}}{W_1S_{2,N} + \frac12 (S_{2,N})^2 -  \frac12 S^{(2)}_{2,N}}
    = \frac{W_1}{W_1 + \frac{(S_{2,N})^2 - S^{(2)}_{2,N}}{2 S_{2,N}}} .
  \end{align}
  Put
  \begin{align*}
    \tilde{A}_{N,\delta} := \Big\{ \frac{1-\delta}{2} \mu_W N <  \frac{(S_{2,N})^2 - S^{(2)}_{2,N}}{2 S_{2,N}}
    < \frac{1+\delta}{2} \mu_W N \Big\}.
  \end{align*}
  We have
  \begin{align}
    \label{Exrff:eqlimPAtildeNdelta}
    \PP\big(\tilde{A}_{N,\delta}\big) \mathop{\longrightarrow}_{N\to\infty} 1,
  \end{align}
  noting that $S_{2,N}/N \to \mu_W$ and $S^{(2)}_{2,N}/N^2 \to 0$ as $N\to\infty$ in probability (for the latter use that $W_i^2$ have regularly varying tails of index $\alpha/2 \in (1/2,1)$, so in particular $S_{2,N}^{(2)}/N^{2/\alpha}$ is tight; this follows e.g.\
  from \cite[Thm.~2 in Section~XVII.5]{Feller}). Furthermore, consider
  \begin{align*}
  B_N:=\Big\{ \frac{(S_{2,N})^2 - S^{(2)}_{2,N}}{2 S_{2,N}}
  < \mu_W N/5 \Big\}
  \end{align*}
  and note that $ \PP (B_N ) \le e^{- r N} $
  for $N$ large enough with $r>0$ due to \eqref{Exrff:eq:SN2tooextreme}
  from Lemma~\ref{Exrff:SNtooextreme} for $k=2.$ Since for small enough
  $\delta>0$ we have $B_N \subset \tilde{A}_{N,\delta}^c$ it then follows for those  $\delta$ that
  \begin{align}
    \label{Exrff:eq:EWquotcN.UB}
    %\EE\bigg[ \frac{W_1^2 \big( \sum_{j=2}^N W_j \big)^2}{Z_N^2}\bigg]
    \EE\left[ {Q_N}^2\right]
    & \le e^{- rN} + \EE\left[ \frac{W_1^2}{\big( W_1 + \mu_W N/5\big)^2} \ind_{\tilde{A}_{N,\delta}^c}\right]
    + \EE\left[ \frac{W_1^2}{\big( W_1 + (1-\delta) \mu_W N/2 \big)^2} \right] \notag \\
    & =  e^{- rN} + \EE\left[ \frac{W_1^2}{\big( W_1 + \mu_W N/5\big)^2} \right] \PP\big(\tilde{A}_{N,\delta}^c \big)
    + \EE\left[ \frac{W_1^2}{\big( W_1 + (1-\delta) \mu_W N/2 \big)^2} \right]
  \end{align}
  and Lemma~\ref{Exrff:lemWdbWpMmom} together with \eqref{Exrff:eqlimPAtildeNdelta} implies
  \begin{align*}
    \limsup_{N\to \infty} N^{\alpha} \EE\left[Q_N^2\right]
    \le c_W \( \tfrac{2}{(1-\delta) \mu_W}\)^\alpha \alpha B(2-\alpha, \alpha),
  \end{align*}
  \eqref{Exrff:eq:EWquotcN.limsup} follows by taking $\delta\downarrow 0$.
  \medskip

  Analogous, in fact a little easier, arguments can be used to show that
  \begin{align}
    \label{Exrff:eq:EWquotcN.liminf}
    %\liminf_{N\to \infty} N^{\alpha} \EE\left[ \frac{W_1^2 \big( \sum_{j=2}^N W_j \big)^2}{Z_N^2} \right]
    \liminf_{N\to \infty} N^{\alpha} \EE\left[  Q_N^2 \right]
    \ge 8 C_\mathrm{pair}^{(\mathrm{Beta})} .
  \end{align}
  Again using \eqref{Exrff:eq:rewrite.pV1} we get
  \begin{align}
    \label{Exrff:eq:EWquotcN.LB}
    %\EE\left[ \frac{W_1^2 \big( \sum_{j=2}^N W_j \big)^2}{Z_N^2} \right]
    \EE\left[{Q_N}^2 \right]
    & \ge \PP(\tilde{A}_{N,\delta}) \EE\bigg[ \frac{W_1^2}{%
      \big(W_1 + (1+\delta) \mu_W N/2 \big)^2} \bigg],
  \end{align}
  now combine \eqref{Exrff:eqlimPAtildeNdelta} with Lemma~\ref{Exrff:lemWdbWpMmom}
  as above and then let $\delta\downarrow0$ to conclude \eqref{Exrff:eq:EWquotcN.liminf}.
  \eqref{Exrff:eq:EWquotcN.limsup} and \eqref{Exrff:eq:EWquotcN.liminf} combined with \eqref{Exrff:eq:cNformula}
  yield \eqref{Exrff:eq:cNBeta}.
  \medskip

  We now assume $\mu_W^{(2)} < \infty.$  The proof of \eqref{Exrff:eq:cNKingman} is similar, in fact simpler:
  Instead of using Lemma~\ref{Exrff:lemWdbWpMmom}, we simply observe that in this case
  \begin{align*}
    %\label{Exrff:eq:WdbWpMmom2}
    \lim_{M\to\infty} M^2 \EE\Big[ \frac{W_1^2}{(W_1+M)^2} \Big]
    = \lim_{M\to\infty} \EE\Big[ W_1^2 \frac{M^2}{(W_1+M)^2} \Big] = \EE\big[ W_1^2 \big] = \mu_W^{(2)}
  \end{align*}
  by dominated convergence. Thus, \eqref{Exrff:eq:EWquotcN.UB} implies that
  $\limsup_{N\to\infty} N^2 \EE\left[Q_N^2\right] \le 4 (1-\delta)^{-2} \mu_W^{(2)}/\mu_W^2$
  and \eqref{Exrff:eq:EWquotcN.LB} implies
  $\limsup_{N\to\infty} N^2 \EE\left[ Q_N^2 \right] \ge 4 (1-\delta)^2 \mu_W^{(2)}/\mu_W^2$.
  Taking $\delta\downarrow 0$, this combined with \eqref{Exrff:eq:cNformula} yields \eqref{Exrff:eq:cNKingman}.
\end{proof}

\begin{proof}[Proof of Lemma~\ref{Exrff:lem:highmombd}]
  1.\ Assume $\mu_W^{(2)}<\infty$ and recall $S_{2,N}$ and $S^{(2)}_{2,N}$ from \eqref{Exrff:defSN.SN2}.   Since
  %\begin{align}
    $\cL\big( V_1 \,\big|\, (W_i ) \big) = \mathrm{Bin}\big( N, Q_N\big)$
  %\end{align}
  we have by \eqref{Exrff:eq:multinommoments}
  \begin{align*}
    \EE\big[ (V_1)_{3} \, \big| \, (W_i)\big]
    = (N)_{3} \frac{W_1^3 (S_{2,N})^3}{Z_N^3}.
  \end{align*}
  We re-write $Z_N = W_1 S_{2,N} + \frac12 \big( (S_{2,N})^2 - S^{(2)}_{2,N} \big)$ as in \eqref{Exrff:eq:ZNre-expr.2}.
  Lemma~\ref{Exrff:SNtooextreme}
  shows that $(S_{2,N})^2 - S^{(2)}_{2,N} \ge (\mu_W/2) N S_{2,N}$ holds with probability $\ge 1-2e^{-r N}.$ Thus,
  \begin{align*}
    \EE\big[ (V_1)_{3} \big] \le N^3 \EE\bigg[ \frac{W_1^3}{(W_1+\mu_W N/2)^3}\bigg]
    + 2 N^3 e^{-r N},
  \end{align*}
  and \eqref{Exrff:eq:phi13zero} follows  as in the proof of \cite[Proposition~7]{Schweinsberg2003} with the help of \eqref{Exrff:eq:cNKingman}
 in Lemma \ref{Exrff:lem:cN}.
\medskip

\noindent
2.\ Write
\[
V' := \sum_{j=3}^N V_{1,j}, \quad V'' := \sum_{j=3}^N V_{2,j},
\]
hence $V_1=V_{1,2}+V'$, $V_2=V_{1,2}+V''$
and a straightforward
calculation yields
\begin{align*}
  (V_1)_{2 } (V_2)_{2}  & = (V_{1,2})_{4} + 4(V_{1,2})_{3} + 2(V_{1,2})_{2}
  + 2 (V_{1,2})_{3} (V' + V'') + 4 (V_{1,2})_{2} (V' + V'') \\
  & \hspace{1.5em} + (V_{1,2})_{2} (V'')_{2} + (V_{1,2})_{2} (V')_{2}
  + 4 (V_{1,2})_{2} V' V'' + 4 V_{1,2} V' V'' \\
  & \hspace{1.5em} + 2 V_{1,2} V' (V'')_{2} + 2 V_{1,2} V'' (V')_{2} + (V')_{2} (V'')_{2} .
\end{align*}
Instead of spelling out the details of this computation, note that there is a
combinatorial interpretation: Consider $V_{1,2}+V'+V''$ numbered balls,
of which $V_{1,2}$ are white, $V'$ are red and $V''$ are blue.
Then $(V_1)_{2 } (V_2)_{2}$ counts the number of ordered pairs we can form where
the first pair consists of two distinct balls which are either white or red and
the second pair consists of two distinct balls which are either white or blue
(and the same ball(s) might possibly appear in both pairs). The right-hand side
decomposes this number: There are $(V_{1,2})_{4}$ pairs where all balls are white and
all are distinct, $4(V_{1,2})_{3}$ pairs where all balls are white and exactly one ball appears
twice, etc.
\smallskip

Since the law of $\big( V_{1,2}, V', V'', N-V_{1,2}- V'- V''\big)$ given the $(W_i)$ is
\begin{align*}
  \mathrm{Multinomial}\Big( N, \frac{W_1 W_2}{Z_N},
  \frac{W_1 S_{3,N}}{Z_N}, \frac{W_2 S_{3,N}}{Z_N}, \frac{\sum_{3 \le j < k \le N} W_j W_k}{Z_N}\Big)
\end{align*}
we find using Lemma~\ref{lem:Multinomfactorialmoments} in the first equality that
\begin{align}
\label{Exrff:eq:EV12V22}
  & \EE\big[ (V_1)_{2 } (V_2)_{2 } \, \big| \, (W_i)\big] \\
  & = (N)_{4} \big( {\textstyle \frac{W_1 W_2}{Z_N}}\big)^4
  + 4 (N)_{3} \big( {\textstyle \frac{W_1 W_2}{Z_N}}\big)^3
  + 2 (N)_{2} \big( {\textstyle \frac{W_1 W_2}{Z_N}}\big)^2 \notag \\
  & \hspace{1.5em} + 2 (N)_{4} \big( {\textstyle \frac{W_1 W_2}{Z_N}}\big)^3
  {\textstyle \frac{(W_1 + W_2) S_{3,N}}{Z_N}}
  + 4 (N)_{3} \big( {\textstyle \frac{W_1 W_2}{Z_N}}\big)^2 {\textstyle \frac{(W_1 + W_2) S_{3,N}}{Z_N}} \notag \\
  & \hspace{1.5em} + (N)_{4} \big({\textstyle \frac{W_1 W_2}{Z_N}}\big)^2
  \big({\textstyle \frac{W_2 S_{3,N}}{Z_N}}\big)^2
  + (N)_{4} \big({\textstyle \frac{W_1 W_2}{Z_N}}\big)^2
  \big({\textstyle \frac{W_1 S_{3,N}}{Z_N}}\big)^2 \notag \\
  & \hspace{1.5em} + 4 (N)_{4} \big({\textstyle \frac{W_1 W_2}{Z_N}}\big)^2
  {\textstyle \frac{W_1 S_{3,N} W_2 S_{3,N}}{Z_N^2}}
  + 4 (N)_{3} {\textstyle \frac{W_1 W_2}{Z_N}}
  {\textstyle \frac{W_1 S_{3,N} W_2 S_{3,N}}{Z_N^2}} \notag \\
  & \hspace{1.5em}
  + 2 (N)_{4} {\textstyle \frac{W_1 W_2}{Z_N} \frac{W_1 S_{3,N}}{Z_N}}
  \big({\textstyle \frac{W_2 S_{3,N}}{Z_N}}\big)^2
  + 2 (N)_{4} {\textstyle \frac{W_1 W_2}{Z_N} \frac{W_2 S_{3,N}}{Z_N}}
  \big({\textstyle \frac{W_1 S_{3,N}}{Z_N}}\big)^2 \notag \\
  & \hspace{1.5em} + (N)_{4} \big({\textstyle \frac{W_1 S_{3,N}}{Z_N}}\big)^2
  \big({\textstyle \frac{W_2 S_{3,N}}{Z_N}}\big)^2 \notag \\
  & = \frac{(N)_{4}}{Z_N^4} \Big( W_1^4 W_2^4 + 2 W_1^3 W_2^3 (W_1+W_2) S_{3,N}
  + W_1^2 W_2^4 (S_{3,N})^2 + W_1^4 W_2^2 (S_{3,N})^2 \notag \\
  & \hspace{4.5em}
  + 4 W_1^3 W_2^3 (S_{3,N})^2  + 2 W_1^2 W_2^3 (S_{3,N})^3 + 2 W_1^3 W_2^2 (S_{3,N})^3
  + W_1^2 W_2^2 (S_{3,N})^4 \Big) \notag \\
  & \hspace{1.5em} + \frac{4 (N)_{3}}{Z_N^3} \Big( W_1^3 W_2^3 + W_1^2 W_2^2 (W_1+W_2) S_{3,N}
  +  W_1^2 W_2^2 ( S_{3,N} )^2 \Big) + \frac{2 (N)_{2} W_1^2 W_2^2}{Z_N^2}. \notag
\end{align}

Recall
\[
Z_N = W_1W_2 + (W_1+W_2) S_{3,N} + \frac12 \big( (S_{3,N})^2 - S^{(2)}_{3,N} \big)
\]
from \eqref{Exrff:eq:ZNre-expr} and that
we can bound
\begin{align*}
(S_{3,N})^2 - S^{(2)}_{3,N} \ge (1-\delta) \mu_W^2 N^2 \vee \frac{\mu_W}{2} N S_{3,N}
\end{align*}
except on an event with exponentially small probability, cf.\ \eqref{Exrff:eq:PAdelta}
and \eqref{Exrff:eq:SN2tooextreme} from Lemma~\ref{Exrff:SNtooextreme}.
\smallskip

We will not treat all the terms on the right-hand side of \eqref{Exrff:eq:EV12V22} in detail
(but see Remark~\ref{Exrff:rem:V12V22bound} below) since the computations are long but
otherwise relatively straightforward.
Consider for example the term
\begin{align*}
\EE\left[ \frac{W_1^2 W_2^2 (S_{3,N})^4}{Z_N^4} \right] =
\EE\left[ \frac{W_1^2 W_2^2 (S_{3,N})^4}{\( W_1W_2 + (W_1+W_2) S_{3,N} + \( (S_{3,N})^2 - S^{(2)}_{3,N} \)/2 \)^4} \right].
\end{align*}
If $(1-\delta) \mu_W N \le S_{3,N} \le 2 \mu_W N$, say, we can estimate
\begin{align*}
\nonumber
  \frac{W_1^2 W_2^2 (S_{3,N})^4}{Z_N^4}
  & \le 16 \mu_W^4 N^4 \frac{W_1^2 W_2^2}{\big( (W_1+W_2)(1-\delta)  \mu_W N + (1-\delta) \mu_W^2 N^2/2 \big)^4} \\
  & \le 16(1-\delta)^{-4} \frac{W_1^2}{\big( W_1 + \mu_W N/2 \big)^2}
  \frac{W_2^2}{\big( W_2 + \mu_W N/2 \big)^2},
\end{align*}
hence
\begin{align*}
  \EE\left[ \frac{W_1^2 W_2^2 (S_{3,N})^4}{Z_N^4} \ind(S_{3,N} \le 2 \mu_W N)\right] &
  \le e^{-r N} + 16(1-\delta)^{-4} \left( \EE\left[ \frac{W_1^2}{\big( W_1 + \mu_W N/2 \big)^2} \right] \right)^2
  \le C N^{-2 \alpha}
\end{align*}
for $N$ large enough where we used Lemma~\ref{Exrff:lemWdbWpMmom} in the last inequality.
\smallskip

On $\{S_{3,N} \ge 2 \mu_W N \}$ we also have $S_{3,N}^2 - S_{3,N}^{(2)} \ge (\mu_W/2) N S_{3,N}$
with high probability. Then
\begin{align*} \nonumber
  & \EE\left[ \frac{W_1^2 W_2^2 (S_{3,N})^4}{Z_N^4} \ind(S_{3,N} > 2 \mu_W N)\right] \\
  \le& \EE\left[ \frac{W_1^2 W_2^2 (S_{3,N})^4}{\big( (W_1 + W_2) S_{3,N} + (\mu_W/4) N S_{3,N}\big)^4} \ind(S_{3,N} > 2 \mu_W N)\right]
  + e^{-r N} \\
   \nonumber
   \le & \EE\left[ \frac{W_1^2 W_2^2}{\big( (W_1 + W_2) + \mu_W N/4 \big)^4} \right] \PP(S_{3,N} > 2 \mu_W N) + e^{-r N} \\
   \le & \left( \EE\left[ \frac{W_1^2}{\big( W_1 + \mu_W N/2 \big)^2} \right] \right)^2 \PP(S_{3,N} > 2 \mu_W N) + e^{-r N}
  = O(N^{-2\alpha})
\end{align*}
and we obtain
\begin{align*}
  \limsup_{N\to\infty} N^{2\alpha }\EE\left[ \frac{W_1^2 W_2^2 (S_{3,N})^4}{Z_N^4} \right] < \infty.
\end{align*}
\smallskip

Similarly, using Lemma~\ref{Exrff:SNtooextreme}
\begin{align*}
  \EE\left[ \frac{W_1^4 W_2^4}{Z_N^4} \right] & \le
  \EE\left[ \frac{W_1^4 W_2^4}{\big( W_1W_2 + (W_1+W_2) (1-\delta)\mu_W N + (1-\delta)\mu_W^2 N^2/2 \big)^4}
    \right] + e^{-r N} \\
    & \le \frac{2^4}{(1-\delta)^4} \EE\left[ \frac{W_1^4 W_2^4}{\big( W_1W_2 + (W_1+W_2) \mu_W N + \mu_W^2 N^2 \big)^4}
    \right] + e^{-r N} \\
    & = \frac{2^4}{(1-\delta)^4} \EE\left[ \frac{W_1^4}{(W_1+ \mu_W N)^4} \frac{W_2^4}{(W_2 + \mu_W N)^4}
    \right] + e^{-r N}
\end{align*}
and we obtain again
\begin{align*}
  \limsup_{N\to\infty} N^{2\alpha}\EE\left[ \frac{W_1^4 W_2^4}{Z_N^4} \right] < \infty
\end{align*}
from Lemma~\ref{Exrff:lemWdbWpMmom}.
\smallskip

The other terms in \eqref{Exrff:eq:EV12V22} can be treated analogously (see Remark \ref{Exrff:rem:V12V22bound})
to yield
\begin{align}
  \label{Exrff:eq:V12V22bound}
  \limsup_{N\to\infty} N^{2\alpha-4} \EE\big[ (V_1)_{2 } (V_2)_{2 } \big] < \infty.
\end{align}
Since $N^2 c_N \sim C_\mathrm{pair}^{(\mathrm{Beta})} N^{3-\alpha} $ by \eqref{Exrff:eq:cNBeta}  of Lemma~\ref{Exrff:lem:cN}
and $4-2\alpha < 3 -\alpha$ this proves the claim \eqref{Exrff:eq:phi22zero}.
\end{proof}
\medskip

\begin{remark}
  \label{Exrff:rem:V12V22bound}
  \rm
  For a rough idea of the size of the terms on the right-hand side of
  \eqref{Exrff:eq:EV12V22} we can argue as follows: Consider the ``typical
  event'' $S_{3,N} \approx \mu N, (S_{3,N})^2 - S^{(2)}_{3,N} \approx \mu^2 N^2$.
  When $W_1$ and $W_2$ are both bounded, the right-hand side of
  \eqref{Exrff:eq:EV12V22} is then $O(1)$, the
  contribution of the case $W_1 \approx N^{\beta_1}$, $W_2 \approx
  N^{\beta_2}$ is then $\approx N^\gamma$ with
\begin{align*}
  \gamma & = -\alpha(\beta_1+\beta_2) + 4 \\
  & \hspace{1.5em} + \big[ (4\beta_1+4\beta_2) \vee (4\beta_1 + 3\beta_2 + 1) \vee (3\beta_1 + 4\beta_2 + 1) \vee
  (2\beta_1 + 4\beta_2 + 2) \vee (4\beta_1 + 2\beta_2 + 2) \\
  & \hspace{3.5em} \vee (3\beta_1 + 3\beta_2 + 2) \vee (2\beta_1 + 3\beta_2 + 3) \vee
  (3\beta_1 + 2\beta_2 + 2) \vee (2\beta_1 + 2\beta_2 + 4) \big] \\
  & \hspace{1.5em} - 4 \big[ (\beta_1+\beta_2) \vee (\beta_1+1) \vee (\beta_2+1) \vee 2 \big].
\end{align*}
Observe that when $\beta_1, \beta_2 \le 1$ this is
$=-\alpha(\beta_1+\beta_2) + 4 + (2\beta_1 + 2\beta_2 + 4) - 4 \cdot 2
= (2-\alpha) (\beta_1+\beta_2)$ ($<3-\alpha$, note that we divide in \eqref{Exrff:eq:phi22zero} by
$c_N N^2 \approx N^{3-\alpha}$); when $\beta_1 < 1 \le \beta_2$, say,
this is $=-\alpha(\beta_1+\beta_2) + 4 + (2\beta_1 + 4\beta_2 + 2) -
4(\beta_2+1) = 2 + (2-\alpha) \beta_1 - \alpha \beta_2$ ($<3-\alpha$); when $\beta_1, \beta_2 > 1$ this is
$=-\alpha(\beta_1+\beta_2) + 4 + (4\beta_1+4\beta_2) - 4 (\beta_1+\beta_2)
= 4 - \alpha(\beta_1+\beta_2)$ ($<3-\alpha$).
This confirms \eqref{Exrff:eq:V12V22bound} at least on an intuitive level.
\end{remark}

\begin{proof}[Proof of Lemma~\ref{Exrff:lem:V1tail}]
  We start by arguing that in order to show \eqref{Exrff:eq:V1asympttail}
  it suffices to check that for $x \in (0,1)$
  \begin{align}
  \label{frequencyconv}
  \lim_{N\to\infty} \frac{N}{c_N} \PP\left( Q_N > x \right)
  \end{align}
  exists and is given by the right-hand side of \eqref{Exrff:eq:V1asympttail}
  (note that this expression is continuous in $x$). Indeed, since $\cL\( V_1 \,\big|\, (W_i ) \) = \mathrm{Bin}\( N, Q_N\)$
  and
  \[
  \mathrm{Bin}(N, p)(\{0,1,\dots,\lceil (1-\epsilon) p N
  \rceil\} \cup \{\lfloor (1+\epsilon) p N
  \rfloor, \dots, N\}) \le e^{-r N}
  \]
  for any $p, \epsilon \in (0,1)$ with
  $r=r(p,\epsilon)>0$
  by classical large deviation estimates for the
  binomial distribution (e.g.\ \cite{DemboZeitouni98}, Thm.~2.2.3 and Ex.~2.2.23~(b)) we may write for $\epsilon>0$
  \begin{eqnarray*}
   \PP\left( \frac{V_1}{N} > x \right)
   &=& \PP\left( \frac{V_1}{N} > x \big| Q_N> x (1+\epsilon)  \right)  \PP\left(  Q_N > x (1+\epsilon) \right)\\
   &&+ \PP\left( \frac{V_1}{N} > x \big| x (1-\epsilon) < Q_N \leq x (1+\epsilon)  \right)
   \times \, \PP\left(   x (1-\epsilon) < Q_N \leq x (1+\epsilon)  \right)\\
   &&+ \PP\left( \frac{V_1}{N} > x \big| Q_N \leq x (1-\epsilon)  \right) \PP\( Q_N \leq x (1-\epsilon) \right).
\end{eqnarray*}
Due to \eqref{frequencyconv} and the continuity of the limit
we can choose $\epsilon>0$ small and then $N$ large to make
the second probability in the second term multiplied by $N/c_N$ and thus the second term
 multiplied by $N/c_N$ arbitrarily small. By choosing $N$ potentially larger and by
using the above large deviations result the conditional probabilities in the first and third line are arbitrarily close to $1$ respectively arbitrarily close to $0$ even when multiplied by $N/c_N \sim 1/C_\mathrm{pair}^{(Beta)} N^{\alpha}$   by \eqref{Exrff:eq:cNBeta}. Thus, the
claim \eqref{Exrff:eq:V1asympttail} now follows from  \eqref{frequencyconv}.

  As in the proof of Lemma~\ref{Exrff:lem:cN} we re-express (see \eqref{Exrff:eq:rewrite.pV1} and
  recall $S_{2,N}$ and $S^{(2)}_{2,N}$ from \eqref{Exrff:defSN.SN2})
  \[
  Q_N = \frac{W_1}{W_1 + \frac{(S_{2,N})^2 - S^{(2)}_{2,N}}{2 S_{2,N}}}.
  \]
  Recall that $\tilde{A}_{N,\delta} := \Big\{ \frac{1-\delta}{2} \mu_W
  N < \frac{(S_{2,N})^2 - S^{(2)}_{2,N}}{2 S_{2,N}} <
  \frac{1+\delta}{2} \mu_W N \Big\}$ satisfies
  $\PP\big(\tilde{A}_{N,\delta}\big) %\mathop\limits{\longrightarrow}_{N\to\infty}
  \to 1$ as $N\to\infty$ for every $\delta >0$ (see
  \eqref{Exrff:eqlimPAtildeNdelta} in the proof of Lemma~\ref{Exrff:lem:cN}), furthermore
  \begin{align*}
    \PP\bigg(\tilde{A}_{N,\delta}^c \cap \Big\{ \frac{(S_{2,N})^2 - S^{(2)}_{2,N}}{2 S_{2,N}}
    < \mu_W N/5 \Big\} \bigg) \le 3 e^{- r N}
  \end{align*}
  for $N$ large enough with $r=r(\delta)>0$ (combine \eqref{Exrff:eq:PAdelta} and \eqref{Exrff:eq:SN2tooextreme}
  from Lemma~\ref{Exrff:SNtooextreme}).
  Thus
  \begin{align*}
    \limsup_{N\to\infty} \frac{N}{c_N} \PP\( \left\{Q_N > x \right\} \cap \tilde{A}_{N,\delta}^c \)
    %\le \limsup_{N\to\infty} \frac{N}{c_N} \PP\bigg( \frac{W_1}{W_1+\mu_W N/5} > x \bigg) \PP\big(\tilde{A}_{N,\delta}^c\big)
    % \notag \\
    & \le \limsup_{N\to\infty} \frac{N}{c_N} \PP\big( W_1>N\mu_W x /5  \big) \PP\big(\tilde{A}_{N,\delta}^c\big) = 0
  \end{align*}
  where we use $N/c_N \sim N^\alpha / C_\mathrm{pair}^{(\mathrm{Beta})}$ by Lemma~\ref{Exrff:lem:cN},
  combined with \eqref{Exrff:eq:tailassumptW1} and \eqref{Exrff:eqlimPAtildeNdelta}.
  \smallskip

  Now
  \begin{align*}
   % \limsup_{N\to\infty} & \frac{N}{c_N} \PP\bigg( \Big\{ \frac{W_1 \sum_{j=2}^N W_j}{Z_N} > x \Big\} \cap \tilde{A}_{N,\delta} \bigg)
    \limsup_{N\to\infty} & \frac{N}{c_N} \PP\( \left\{ Q_N > x \right\} \cap \tilde{A}_{N,\delta} \)
    \le \limsup_{N\to\infty} \frac{N}{c_N} \PP\( \frac{W_1}{W_1+(1-\delta)\mu_W N/2} > x \) \PP\(\tilde{A}_{N,\delta}\)
    \notag \\
    & = \limsup_{N\to\infty} \frac{N^\alpha}{C_\mathrm{pair}^{(\mathrm{Beta})}} \PP\Big( W_1 > \frac{x}{1-x} (1-\delta) \mu_W N/2 \Big)
    = \frac{(1-\delta)^{-\alpha}}{C_\mathrm{pair}^{(\mathrm{Beta})}} c_W \frac{2^\alpha}{\mu_W^\alpha} \( \frac{1-x}{x} \)^\alpha
  \end{align*}
  by the tail assumption \eqref{Exrff:eq:tailassumptW1}, analogous to the proof of Lemma~14
  in \cite{Schweinsberg2003}, and similarly for the $\liminf$.
  Finally note that from \eqref{Exrff:eq:cNBeta} in Lemma~\ref{Exrff:lem:cN} we have
  \begin{align*}
    \frac{c_W (2/\mu_W)^\alpha}{C_\mathrm{pair}^{(\mathrm{Beta})}} \( \frac{1-x}{x} \)^\alpha
    = \frac{8}{B(2-\alpha,\alpha)} \frac1{\alpha} \( \frac{1-x}{x} \)^\alpha
    = \frac{8}{B(2-\alpha,\alpha)} \int_x^1 \frac1{y^2} y^{1-\alpha} (1-y)^{\alpha-1} \, dy .
  \end{align*}
\end{proof}

%%%%%***

\subsection{Proof of Proposition~\ref{prop:diploidSchweinsberg}}
\label{sProofSchweinsberg}

In this section we prove Lemma~\ref{lemma:SchweinsbergcN} and Proposition~\ref{prop:diploidSchweinsberg},  the main convergence result for the
diploid population model of Section~\ref{section:supercriticalGaltonWatson}, where potential offspring to parental couples are generated as in a supercritical Galton-Watson process and then pruned. %We could also shorten this...
 In order to prove the main
Proposition~\ref{prop:diploidSchweinsberg} we need two lemmas in addition to Lemma~\ref{lemma:SchweinsbergcN}. The proof of these auxiliary lemmas
is postponed to Section \ref{sSchweinsbergauxiliary}.

Our proofs in this section are to some extent parallel to
  those in \cite{Schweinsberg2003}, especially those of
  Lemmas~\ref{lemma:SchweinsbergcN} and \ref{lemma:SNbounds}.  We note
  however that the arguments are somewhat more involved, in particular
  due to the fact that $X_i^{(N)}, 1 \leq i \leq N$ are not
  independent.  Our proof of
  Proposition~\ref{prop:diploidSchweinsberg},~2.\ follows a slightly
  different route than that of its analogue,
  \cite[Thm.~4~(c)]{Schweinsberg2003} in that we verify
  condition~\eqref{eq:PhiNconv} on the law of the ranked offspring
  frequencies directly without recourse to the moment criterion
  \eqref{eq:Vfmcond}. One can alternatively prove
  Proposition~\ref{prop:diploidSchweinsberg} by verifying
  \eqref{eq:Vfmcond} but this route appears more cumbersome
  here because the $X_i$ are not independent in our set-up.
\smallskip

Recall that the offspring $V_{i,j}^{(N)}$ of individual $i$ and $j$ are sampled from the ``potential offspring" $X_{i,j}^{(N)},$ which are i.i.d..
We write $V_i^{(N)} = \sum_{j\neq i}^N V_{i,j}^{(N)}$ as before. Our first observation concerns  large deviations of the total number of potential offspring $S_N= \sum_{ 1\le i<j\le N} X_{i,j}^{(N)}$ of (\ref{S_Ndef}) which by (\ref{ES_N}) has $\EE[S_N] \sim N \mu$ with $\mu>1.$
\begin{lemma}
  \label{lemma:SNbounds}
  For any $\varepsilon>0$ there exist constants $c=c(\varepsilon) >0$ and $C=C(\varepsilon)<\infty$
  such that
  \begin{align}
    \label{lemma:SNbounds.small}
    \PP\big( S_N \leq (1-\varepsilon) N \mu  \big) & \leq e^{-c N} \quad \text{ for all $N$ large enough} \\
    \intertext{and under Assumption~\eqref{ass:Xtail} we have}
    \label{lemma:SNbounds.large}
    \PP\big( S_N \geq (1+\varepsilon) N \mu  \big) & \leq C N^{1-\alpha} \quad \text{ for all $N$ large enough}
  \end{align}
  (under Assumption~\eqref{ass:Xvariance}, \eqref{lemma:SNbounds.large} holds with the right-hand side
  replaced by $C N^{-1}$, i.e., by formally setting $\alpha=2$.)
  \smallskip

  The same bounds hold for $S_N' := S_N - X^{(N)}_{1,2}$ and for $S_N'' := S_N - X^{(N)}_{1}$.
\end{lemma}

We note that on $\{S_N \ge N\}$, conditioned on the $X^{(N)}_{k,\ell}$'s, each $V^{(N)}_i$
is hypergeometric. More precisely, if $\mathrm{hypergeom}(n,m,x)$  is the number of white balls in $n$ draws without replacement from an urn that
contains $m$ balls of which $x$ are white balls and $m-x$ are black balls then $V^{(N)}_i \sim \mathrm{hypergeom}(N,S_N,X^{(N)}_i)$ conditionally
on the $X^{(N)}_{k,\ell}$'s. However, the $V^{(N)}_i$'s for different $i$ are not independent.
\medskip

We now consider $A^{(N)}_{i,j}$, $1 \le i < j \le N$ to be i.i.d.\ $\sim \mathrm{Ber}(p_N)$
(for other $i \neq j \in [N]$ put $A^{(N)}_{j,i} = A^{(N)}_{i,j}$, $A^{(N)}_{i,i}=0$), as well as $Y^{(N)}_{i,j}$, $i, j\in \NN$ to be independent\ copies of
$X$ from \eqref{eq:Xijstructure}, independent of the $A^{(N)}_{i,j}$'s.
A convenient parametrisation of the set-up from \eqref{eq:Xiid} and \eqref{eq:Xijstructure} is to set
\begin{equation}
  \label{eq:Xijrepr}
  X_{i,j}^{(N)} = A^{(N)}_{i,j} Y^{(N)}_{i,j}.
\end{equation}
Note that
\begin{equation}
\label{eq:A^Ndef}
A^{(N)}_i := \sum_{j=1}^N A_{i,j}^{(N)},
\end{equation}
the number of different ``potential partners'' of individual $i$,
is $\mathrm{Bin}(N-1,p_N)$-distributed,
\begin{equation}
\label{A^N}
A^{(N)} := \sum_{1\leq i < j \le N} A_{i,j}^{(N)},
\end{equation}
the total number of ``potential offspring-generating'' pairs,
is $\mathrm{Bin}\big(N(N-1)/2,p_N\big)$-distributed.
\smallskip

\begin{lemma}
  \label{lemma:coinscheme}
  Let $A^{(N)}_{i,j}$, $1 \le i < j \le N$ be i.i.d.\ $\sim \mathrm{Ber}(p_N)$ with
  $p_N \sim c_{X,1}/N$ and $A^{(N)}_{j,i} = A^{(N)}_{i,j}$ for $j>i$ as defined above.
  Then
  \begin{align}
    \label{lemma:coinscheme.eq1}
    \PP\Big( \exists\, k \le N \, : \; {\textstyle \sum\limits_{i \neq k}^N} A^{(N)}_{i,k} \ge \log N \Big)
    = O(N^{-b})
  \end{align}
  for every $b>0$.
  \smallskip

  For $A^{(N)}= \sum_{1 \le i < j \le N} A^{(N)}_{i,j}$ as in (\ref{A^N}) we have
  $\EE[A^{(N)}] = \binom{N}{2} p_N \sim c_{X,1} N/2$ and
  \begin{align}
    \label{lemma:coinscheme.eq2}
    \PP\( \big| A^{(N)} - \EE[A^{(N)}]\big| > \varepsilon \EE[A^{(N)}] \)
    \le 2 \exp\( - \frac{\varepsilon^2}{3} \EE[A^{(N)}]\).
  \end{align}
  for all $N$ and $0< \varepsilon \le 1/2$.
\end{lemma}

In the following we will drop the superscript $(N)$ often for notational simplicity.

\begin{proof}[Proof of Proposition~\ref{prop:diploidSchweinsberg}]
  1.\ In view of Lemma~\ref{lemma:SchweinsbergcN} and
  Condition~\eqref{eq:Vfmrel} it suffices to verify that $\EE[(V_1)_3]
  = o(N)$ for then $\phi_1(3) = \lim_N \EE[(V_1)_3]/(c_N N^2) = 0$
  and this implies \eqref{eq:PhiNconv} and \eqref{eq:Vfmrel} with $\Xi'=\delta_{\mathbf{0}}$
  (see e.g.\ \cite[Thm.~4~(b)]{Mohle00}).
  This can be checked along the lines of the proof of
  \cite[Proposition~7]{Schweinsberg2003}.  We simply need that for some $0<a<1$
  \begin{equation}
  \label{X_1^2unif}
  \limsup_{N \rightarrow \infty} \EE[X_1^2 {\ind}_{\{X_1  \geq N^a\}}] =0.
  \end{equation}
  To see that this is true we use  \eqref{eq:Xijstructure} and \eqref{ass:Xvariance} to estimate
  \begin{eqnarray*}
  \EE[X_1^2 {\ind}_{\{X_1  \geq N^a\}}]
  &=& \(N-1\) \EE\left[ X_{1,2}^2\(\ind_{\{X_{1,2}  \geq N^a/2\} }+\ind_{ \{ \sum_{j=3}^N X_{1,j}  \geq N^a/2 \} }\)\right]\\
  &\leq& \(N-1\) p_N  \EE\left[ X^2 \ind_{\{X  \geq N^a/2\}}\right] + \(N-1\) p_N \(\frac{2}{N^a}\)^2 \EE[X^2] \rightarrow 0
  \end{eqnarray*}
  as $N \rightarrow \infty.$ This shows \eqref{X_1^2unif} and thus the claim.

  \bigskip

\noindent
2.\
  Let $V_{(1)} \ge V_{(2)} \ge \cdots \ge V_{(N)}$ be the ranked $V_i$'s.
  We will verify that \eqref{eq:PhiNconv} holds with
  \[
  \Xi'(dx) = \int_{[0,1]} \delta_{(x/2,x/2,0,0,\dots)} \, \mathrm{Beta}(2-\alpha,\alpha)(dx)
  \]
  (and then Proposition~\ref{prop:diploidSchweinsberg},~2.\ follows from our main result, Theorem~\ref{thm:res},
  via Condition~\eqref{eq:PhiNconv}).

  This in turn will follow if we show that for every $k \in \NN$, $1 > x_1 \ge x_2 \ge x_3 \ge \dots \ge x_k > 0$ we have
  \begin{align}
    & \lim_{N\to\infty}\frac1{2c_N} \PP\( {\textstyle \frac{V_{(1)}}{2N} > x_1, \frac{V_{(2)}}{2N} > x_2, \cdots,
      \frac{V_{(k)}}{2N} > x_k} \) \notag \\
    & \hspace{2em} = \int_{\Delta} \ind(y_1 \ge x_1, y_2 \ge x_2,\dots, y_k \ge x_k) \frac{1}{(y,y)} \Xi'(dy) \notag \\
    & \hspace{2em} = \ind(x_1 \leq 1/2, k \leq 2) \int_{2x_1}^1 \frac{2}{y^2} \, \mathrm{Beta}(2-\alpha,\alpha)(dy) \notag \\
    \label{proofprop:beta.target1}
    & \hspace{2em} = \ind(x_1 \leq 1/2, k \leq 2) \frac{2}{\alpha \Gamma(2-\alpha)\Gamma(\alpha)} \( \frac{1-2x_1}{2 x_1}\)^\alpha.
  \end{align}
%  (as in \cite[Proof of Lemma~14]{Schweinsberg2003}
Here we have used in the last line that with the substitution of $z = (1-y)/y$
\begin{eqnarray}
\nonumber
\Gamma(2-\alpha)\Gamma(\alpha) \int_{2x_1}^1 \frac{1}{y^2} \, \mathrm{Beta}(2-\alpha,\alpha)(dy)
&=&   %\int_{2 x_1}^1 y^{-1-\alpha} (1-y)^{-1+\alpha} \, dy=
 \int_{2 x_1}^1\( \frac{1-y}{y} \)^{\alpha-1} \frac{dy}{y^2} \\
 \label{intcalc}
&=& \int^{ (1-2x_1)/2x_1}_0 z^{\alpha-1}dz= \alpha^{-1}\( \frac{1-2x_1}{2 x_1}\)^\alpha.
\end{eqnarray}
  The intuition behind this result is the following:
  $V_{i,j} > Ny$ with $y>0$ is possible (only) if $X_{i,j}$ is of order $N$ (and then with overwhelming
  probability the sum of all other $X_{k,\ell}$ with $\{k,\ell\} \neq \{i,j\}$ is $\approx N\mu$
  and both $V_i$ and $V_j$ are $\approx V_{i,j}$ up to terms which are $o(N)$); we then have
  \[
  V_{i,j} \approx N \frac{X_{i,j}}{X_{i,j}+ N\mu}
  \]
  (using fluctuation bounds for hypergeometric distributions), thus
  \[
  \text{$V_{i,j} > N y$ ``$\!\iff\!$'' $\frac{X_{i,j}}{X_{i,j}+ N\mu} > y$
    $\iff$ $X_{i,j} > N\mu \frac{y}{1-y}$}.
  \]
  Furthermore
  \[
  \PP\( \frac{V_{(1)}}{2N} > x_1, \frac{V_{(2)}}{2N} > x_2\) \approx
  \PP\( \exists\: \text{ exactly one pair } i < j \le N\text{ with }
  V_{(1)} \approx V_{(2)} \approx V_{i,j} > 2N x_1\) + o(N^{1-\alpha}).
  \]
Recall $A^{(N)}$ from  \eqref{eq:A^Ndef}.  By Lemma~\ref{lemma:coinscheme}, $A^{(N)} \approx c_{X,1} N/2$, then
  \begin{align*}
    \PP\( \frac{V_{(1)}}{2N} > x_1, \frac{V_{(2)}}{2N} > x_2\)
    & \sim \frac{c_{X,1} N}{2} \PP\( X > N\mu \frac{2x_1}{1-2x_1}\) \notag \\
    &
    \sim \frac{c_{X,1} N}{2}\, c_{X,2} N^{-\alpha} \mu^{-\alpha} \( \frac{2x_1}{1-2x_1} \)^{-\alpha}
    = \frac{c_{X,1} c_{X,2}}{2 \mu^\alpha} \( \frac{1-2x_1}{2x_1} \)^{\alpha} N^{1-\alpha}
  \end{align*}
  (note that then because $x_1 \ge x_2$, the event $\big\{ V_{(1)} > 2N x_1, V_{(2)} < 2N x_2\big\}$
  is very unlikely).

  Combining this with Lemma~\ref{lemma:SchweinsbergcN} suggests \eqref{proofprop:beta.target1}
  for $k=2$, namely
  \begin{align*}
    \frac1{2 c_N} \PP\( \frac{V_{(1)}}{2N} > x_1, \frac{V_{(2)}}{2N} > x_2\)
    & \approx \frac{4 \mu^\alpha}{c_{X,1} c_{X,2} \alpha B(2-\alpha,\alpha)} N^{\alpha-1}
    \, \frac{c_{X,1} c_{X,2}}{2 \mu^\alpha} \( \frac{1-2x_1}{2x_1} \)^{\alpha} N^{1-\alpha}
    \notag \\
    & = \frac{2}{\alpha B(2-\alpha,\alpha)} \( \frac{1-2x_1}{2x_1} \)^{\alpha}.
  \end{align*}
  For $\varepsilon > 0$,
  \[
  \text{$\PP(V_{(3)} > \varepsilon N) =o(N^{1-\alpha})$
    and $\PP(V_{(1)} > Ny, N(1-\varepsilon)y \ge V_{(2)}) = o(N^{1-\alpha})$}
  \]
  (both events require essentially that there are at least two distinct pairs
  $\{i,j\} \neq \{k,\ell\}$ with $X_{i,j}, X_{k,\ell} \ge (\varepsilon/2) N$, say;
  observe that $V_{i,j} \le X_{i,j}$ by definition).
  % \smallskip

  % This gives \eqref{proofprop:beta.target1} (modulo making precise the ``$\approx$''
  % and ``$\Leftrightarrow$'' in various places, etc.).
  \medskip

  A more detailed argument runs as follows:
  Let
  \[
  B_{i,j}(y,\varepsilon) = B_{i,j}^{(N)}(y,\varepsilon) := \left\{ X_{i,j}^{(N)} > N\mu \frac{2y}{1-2y}, S_N - X_{i,j}^{(N)} \le N\mu (1+\varepsilon)\right\}
  \]
  for $y\in (0,1/2)$, $\varepsilon>0$.
  In the following we again drop the index $N$ and note that
  \begin{align}
    \label{eq:PBNij}
    \PP\big(B_{i,j}(y,\varepsilon)\big)
    \sim \PP\( X_{1,2} > N\mu \frac{2y}{1-2y} \)
    \sim \frac{c_{X,1}}{N} c_{X,2} \frac{(1-2y)^\alpha}{(2y)^\alpha \mu^\alpha} N^{-\alpha}
    % + O(N^{-1-2\alpha})
    \quad \text{as } N \to \infty
  \end{align}
  because
  \begin{align*}
    & \PP\Big( X_{i,j} > N\mu \frac{2y}{1-2y}, S_N - X_{i,j} > N\mu (1+\varepsilon) \Big) \\
    & = \PP\Big( X_{1,2} > N\mu \frac{2y}{1-2y} \Big) \PP\big( S_N - X_{1,2} > N\mu (1+\varepsilon)\big)
    \le C \frac{1}{N} \frac{(1-2y)^\alpha}{(2y)^\alpha \mu^\alpha} N^{-\alpha} \cdot N^{1-\alpha}
    = C' N^{-2\alpha}
  \end{align*}
  by Assumption~\eqref{ass:Xtail} and Lemma~\ref{lemma:SNbounds} using the fact that $\alpha>1.$
  \smallskip

  Since the $V_{k,\ell}$'s are hypergeometric conditional on the $X_{k,\ell}$'s with
  \begin{align*}
    \EE\big[ V_{i,j} \,\big| \, \text{$X_{k,\ell}$'s}\big] = N \frac{X_{i,j}}{S_N}
    = N \frac{X_{i,j}}{X_{i,j} + (S_N-X_{i,j})}
    \quad \text{on $\{S_N \ge N\}$}
  \end{align*}
  and the right-hand side is on $B_{i,j}(y,\varepsilon) \cap \{S_N \ge N\}$ bounded
  below by
  \begin{align*}
    N \frac{X_{i,j}}{X_{i,j} + N\mu (1+\varepsilon)}
    \ge N \frac{2y/(1-2y)}{2y/(1-2y) + (1+\varepsilon)} \geq N \frac{2y}{1 + \varepsilon},
  \end{align*}
  we have by fluctuation bounds for hypergeometric laws (e.g.\ as recalled in \cite[Lemma~19]{Schweinsberg2003}; see
  \cite{Chvatal79, Hoeffding63}) and by conditioning on $X_{k,\ell}$'s that
  \begin{align}
  \label{proofprop:beta.lb0}
    \PP\Big( \big\{ V_{i,j} \leq 2Ny/(1+2\varepsilon) \big\} \cap B_{i,j}(y,\varepsilon) \cap \{S_N \ge N\} \Big) \le e^{-c(\varepsilon) N}
  \end{align}
  with $c(\varepsilon)>0$. Using \eqref{proofprop:beta.lb0} as well as the fact that  $\{V_{i,j} \ge m\} \subset \{V_i \ge m, V_j \ge m\} \subset \{ V_{(1)} \ge m, V_{(2)}\ge m\}$
  for every $m$ and all $i<j \le N$ we obtain that
  \begin{align}
    \label{proofprop:beta.lb1}
    &\limsup_{N\to\infty} \frac{1}{2c_N} \PP\Big( \bigcup_{i<j\le N} B_{i,j}\big(x_1(1+2\varepsilon),\varepsilon\big) \cap \{S_N \ge N\} \Big) \\
    \nonumber
    &= \limsup_{N\to\infty} \frac{1}{2c_N} \PP\Big( \bigcup_{i<j\le N} \big\{ V_{i,j} > 2Nx_1 \big\}  \cap B_{i,j}\big(x_1(1+2\varepsilon),\varepsilon\big) \cap \{S_N \ge N\} \Big)\\
    \nonumber
    &\leq \liminf_{N\to\infty} \frac{1}{2c_N} \PP\( \frac{V_{(1)}}{2N} > x_1, \frac{V_{(2)}}{2N} > x_1 \).
  \end{align}
  Furthermore $\PP(S_N<N) \le e^{-c N}$ for $N$ large enough by Lemma~\ref{lemma:SNbounds} and
  for $\{i,j\} \neq \{k,\ell\}$ we have  that
  \begin{align}
   \nonumber
    \PP\Big( B_{i,j}\big(y,\varepsilon\big) \cap B_{k,\ell}\big(y,\varepsilon\big) \Big)
   &\le     \PP\Big( X_{i,j} > N\mu \tfrac{2y}{1-2y}, X_{k,\ell} > N\mu \tfrac{2y}{1-2y} \Big) \\
    \nonumber
    &\le \( C \frac{1}{N} c_{X,2} \frac{(1-2y)^\alpha}{(2y)^\alpha \mu^\alpha} N^{-\alpha} \)^2.
  \end{align}
  This gives
  \begin{align}
     & \PP\( \bigcup_{i<j\le N} B_{i,j}\big(x_1(1+2\varepsilon),\varepsilon\big) \cap \{S_N \ge N\} \) \notag \\
     & = \sum_{i<j\le N} \PP\Big( B_{i,j}\big(x_1(1+2\varepsilon),\varepsilon\big) \Big) + O\Big( N^2 \cdot N^{-2-2\alpha}+ e^{-c N}\Big)
     \notag \\
     \label{proofprop:beta.lb2}
     & = \binom{N}{2} \PP\Big( B_{1,2}\big(x_1(1+2\varepsilon),\varepsilon\big) \Big) + o(N^{1-\alpha}).
  \end{align}
  Combining \eqref{eq:PBNij}, \eqref{proofprop:beta.lb1}, \eqref{proofprop:beta.lb2} with the calculation in \eqref{intcalc}
  and Lemma~\ref{lemma:SchweinsbergcN} yields
  \begin{align}
    \label{proofprop:beta.lb3}
    \liminf_{N\to\infty} \frac{1}{2c_N} \PP\( \frac{V_{(1)}}{2N} > x_1, \frac{V_{(2)}}{2N} > x_1 \)
    \ge \int_{2x_1 (1+2\varepsilon)}^1 \frac{2}{u^2} \, \mathrm{Beta}(2-\alpha,\alpha)(du) .
  \end{align}
  \medskip

 \noindent
  For the matching upper bound we let
  \begin{align*}
    D^{(N)} := \bigcap_{i < j \le N} \big\{ S_N - X_{i,j}^{(N)} \ge (1-\varepsilon) \mu N \big\}
    \cap \bigcap_{k \le N} \big\{ A^{(N)}_k < \log N \big\}. 
  \end{align*}
 Note that if we choose $\varepsilon$ small such that $(1-\varepsilon) \mu>1$ we in particular have
 $S_N \geq N$ on $D^{(N)}.$ Also,
  \begin{align}
    \label{proofprop:beta.ub1}
    \PP\big((D^{(N)})^c \big) \leq N^2 e^{-c N} + N^{-b}
  \end{align}
  by Lemmas\ \ref{lemma:SNbounds} and \ref{lemma:coinscheme}
  where $c>0$ and the constant $b>0$ can be chosen (much) larger than $\alpha-1$.
  \smallskip
  Choose $\delta \in \big(\tfrac{\alpha-1}{2}, \alpha-1\big)$.
\begin{comment}
  \begin{align}
    D^{(N)} \cap \bigcap_{i < j \le N} \big\{ X_{i,j} \le N^{(1+\delta)/\alpha} \big\}
    & \subset \Big\{ \max_{i\le N} X_i \le N^{(1+\delta)/\alpha} \log N\Big\} \notag \\
    \label{proofprop:beta.ub2}
    & \subset \big\{ V_{(1)}  \le N^{(1+\delta)/\alpha} \log N \big\}
    \subset \big\{ V_{(1)}  \ge 2N y \big\}^c
  \end{align}
  for all $y>0$ when $N$ is large enough (because $(1+\delta)/\alpha < 1$).
  \end{comment}
  \smallskip
 Then
  \begin{align*}
    E^{(N)}  := \Big\{ & \exists \, i < j \le N, \, k < \ell \le N \text{ with } \{i,j\} \neq \{k,\ell\}
    %\\ &
    \; \text{ and } \; X_{i,j} \ge N^{(1+\delta)/\alpha}, X_{k,\ell} \ge N^{(1+\delta)/\alpha} \Big\}
  \end{align*}
  due to \eqref{eq:Xijstructure} and  \eqref{ass:Xtail} that
  \begin{align}
    \label{proofprop:beta.ub3}
    \PP\big( E^{(N)} \big) \le \big(N^2\big)^2 \, \( \frac1N \frac{C}{(N^{(1+\delta)/\alpha})^{\alpha}} \)^2
    = O( N^{-2\delta}) = o\big(N^{1-\alpha}\big)
  \end{align}
  because $\delta>(\alpha-1)/2$.
  We also have
  \begin{align}
    \label{proofprop:beta.ub4}
    D^{(N)} \cap (E^{(N)})^c \cap \Big\{V_{(1)} \ge 2Ny \Big\}
    \subset \bigcup_{i < j \le N} \Big\{ V_{i,j} \ge   (1-\varepsilon) 2Ny \Big\}
  \end{align}
  for $N$ large enough since on $D^{(N)} \cap (E^{(N)})^c $ every $V_i$ consists of the sum of at most $\log N$
  many nonzero $V_{i,j}$ (since this is true for $X_i$ and $X_{i,j}$) of which only one can be larger than $N^{(1+\delta)/\alpha}$
  (note that $ (1+\delta)/\alpha<1$ so that the largest $V_{i,j}$ needs to be of the same order as $V_{(1)}$).

 \begin{comment}
  We also have
  \begin{align}
    \label{proofprop:beta.ub4}
    D^{(N)} \cap (E^{(N)})^c \cap \Big\{ \max_{i\le N} X_i \ge N\mu (1-\varepsilon) \frac{2y}{1-2y} \Big\}
    \subset \bigcup_{i < j \le N} \Big\{ X_{i,j} \ge N \mu  (1-2\varepsilon) \frac{2y}{1-2y}  \Big\}
  \end{align}
  for $N$ large enough since on $D^{(N)} \cap (E^{(N)})^c $ every $X_i$ consists of the sum of at most $\log N$
  many nonzero $X_{i,j}$ of which only one can be larger than $N^{(1+\delta)/\alpha}$ (note that $ (1+\delta)/\alpha<1$
  so that the the largest $X_{i,j}$ needs to be of the same order as $X_i$).
  \end{comment}

  We now set
  \begin{equation*}
  \tilde{B}_{i,j}(y,\varepsilon) = \tilde{B}_{i,j}^{(N)}(  y,\varepsilon) := \big\{ X_{i,j}^{(N)} < N\mu  \frac{2y}{1-2y}, S_N - X_{i,j}^{(N)} \ge N\mu (1-\varepsilon)\big\}.
  \end{equation*}
  Then we have  by fluctuation bounds for hypergeometric distributions, analogous to the argument for \eqref{proofprop:beta.lb0}, that
  \begin{equation}
   \label{proofprop:beta.ub5}
  \PP\( \Big\{V_{i,j} \geq (1-\varepsilon)2Ny \Big\} \cap  \tilde{B}_{i,j}((1-2\varepsilon) y,\varepsilon) \cap \big\{ S_N \geq N \big\} \) \leq e^{-c(\varepsilon)N}.
  \end{equation}

  \begin{comment}
   \begin{align}
    \label{proofprop:beta.ub5}
    \PP\Big( D^{(N)} \cap \bigcap_{i < j \le N} \Big\{ X_{i,j} \le (1-2\varepsilon) \frac{2y}{1-2y} \mu N \Big\}
    \cap \bigcup  \Big\{ V_{i,j} \ge (1-\varepsilon) \frac{2y}{1-2y} \mu N \Big\}) \le N e^{-c N}
  \end{align}
  by fluctuation bounds for hypergeometric distributions, analogous to the argument for \eqref{proofprop:beta.lb1}.

  \begin{align}
    \label{proofprop:beta.ub5}
    \PP\Big( D^{(N)} \cap \Big\{ \max_{i\le N} X_i < N\mu (1-\varepsilon) \frac{2y}{1-2y} \Big\}
    \cap \big\{ V_{(1)} \ge 2N(1-2\varepsilon) y \big\} \Big) \le N e^{-c N}
  \end{align}
  by fluctuation bounds for hypergeometric distributions, analogous to the argument for \eqref{proofprop:beta.lb1}.
\end{comment}

  Combining \eqref{proofprop:beta.ub1}, \eqref{proofprop:beta.ub3} and then
  %\eqref{proofprop:beta.ub1},
  \eqref{proofprop:beta.ub4}
  we see that
  \begin{align*}
  \notag
    \PP\big( V_{(1)} \ge 2N y\big) & \le
    \PP\big((D^{(N)})^c \big) + \PP\big( E^{(N)} \big) +
    \PP\Big( \{ V_{(1)} \ge 2N y\} \cap D^{(N)} \cap (E^{(N)})^c\Big) \\
     \notag
    & \le
     \PP\Big( \Big\{ V_{(1)} \ge 2N y \Big\} \cap D^{(N)} \cap (E^{(N)})^c\cap \bigcap_{i < j \le N} \Big\{ X_{i,j} <  \frac{2y(1-2\varepsilon)}{1-2y(1-2\varepsilon)} \mu N \Big\}\Big) \\
    &\;\;+\PP\Big( \bigcup_{i < j \le N} \Big\{ X_{i,j} \ge  \frac{2y(1-2\varepsilon)}{1-2y(1-2\varepsilon)} \mu N \Big\}\Big)
    + o\big(N^{1-\alpha}\big) \notag \\
    \notag
    & \leq \PP\(    \bigcup_{i < j \le N} \Big\{V_{i,j} \geq (1-\varepsilon)2Ny \Big\} \cap  \bigcap_{i < j \le N} \tilde{B}_{i,j}((1-2\varepsilon) y,\varepsilon) \cap \big\{ S_N \geq N \big\}\) \\
   &\;\;+ \binom{N}{2} \PP\Big( X_{1,2} \ge  \frac{2y(1-2\varepsilon)}{1-2y(1-2\varepsilon)} \mu N \Big) + o\big(N^{1-\alpha}\big),
  \end{align*}
  hence, by \eqref{proofprop:beta.ub5},
  \begin{align} \notag
    \limsup_{N\to\infty} \frac1{2c_N} \PP\Big( \frac{V_{(1)}}{2N} > x_1, \frac{V_{(2)}}{2N} > x_2 \Big)
    & \le \limsup_{N\to\infty} \frac1{2c_N} \PP\Big( \frac{V_{(1)}}{2N} > x_1 \Big) \\
    & \le \limsup_{N\to\infty} \frac1{2c_N} \binom{N}{2}
    \PP\Big( X_{1,2} \geq  \frac{2x_1(1-2\varepsilon)}{1-2x_1(1-2\varepsilon)} \mu N \Big)
    \notag \\
     \label{proofprop:beta.ub6}
    & \le \int_{2x_1 (1-2\varepsilon)}^1 \frac{2}{u^2} \, \mathrm{Beta}(2-\alpha,\alpha)(du),
  \end{align}
where we have (as in the lower bound \eqref{proofprop:beta.lb3}) used \eqref{intcalc} and \eqref{eq:PBNij} as well as Lemma \ref{lemma:SchweinsbergcN}.
Now let $\varepsilon \downarrow 0$ in \eqref{proofprop:beta.lb3} and
  \eqref{proofprop:beta.ub6} to obtain \eqref{proofprop:beta.target1} for $k \leq 2$.

  \medskip
  \noindent
  Finally, for $\varepsilon > 0$ and $N$ sufficiently large, we have
  \begin{align*}
    \PP\big( V_{(3)} \ge \varepsilon N\big) \le \PP\big( \exists\, \text{distinct } i,j,k \le N \, : \:
    X_i, X_j, X_k \ge \varepsilon N\big)
    \le \PP\big( (D^{(N)})^c \big) + \PP\big(E^{(N)}\big) = o\big(N^{1-\alpha}\big)
  \end{align*}
  as above.  Combining again with Lemma~\ref{lemma:SchweinsbergcN},
  this completes the proof of \eqref{proofprop:beta.target1} for $k
  \geq 3$.
\end{proof}

\subsubsection{Proofs of auxiliary results}
\label{sSchweinsbergauxiliary}

In this section we prove Lemma~\ref{lemma:SchweinsbergcN}, \ref{lemma:SNbounds} and \ref{lemma:coinscheme}.
To start with we need one more lemma.

\begin{lemma}
\label{lemma:randomsumtails}
Let $Y_1, Y_2,\dots$ be independent\ copies of $Y>0$ with
\begin{equation*}
  \PP(Y>y) \sim c_Y y^{-\alpha},
\end{equation*}
where $c_Y \in (0,\infty)$, $\alpha \in (1,2)$. Let $B_N \sim \mathrm{Bin}(N,p_N)$ where
$p_N \sim c_p/N$ with $c_p \in (0,\infty)$ and set
\[
X_N := \sum_{i=1}^{B_N} Y_i.
\]
Then we have that
%\begin{equation*}
%  \PP(X_N > x) \sim c_p c_Y x^{-\alpha}.
%\end{equation*}
\begin{equation*}
  %\label{claim:tailXN}
  \forall \, \varepsilon > 0 \;\: \exists \, N_0, x_0 < \infty \; : \;\;
  \sup_{N\ge N_0} \sup_{x \ge x_0} \Big| \frac{\PP(X_N > x)}{c_p c_Y x^{-\alpha}} - 1 \Big| \le \varepsilon.
\end{equation*}
\end{lemma}
\begin{proof}
  Write $\mu := \EE[Y_1]$ ($<\infty$ because $\alpha>1$),
  $S_m := Y_1+\cdots+Y_m$. Inspection of the proof of \cite[Thm.~1]{Nagaev82}
  gives
\begin{equation}
  \label{eq:Nagaev.quantitative}
  \forall \, c>\mu, \varepsilon > 0 \; \exists\, x_*=x_*(c, \varepsilon) \text{ such that }
  \; \sup_{m \in \NN} \sup_{x \, \ge \, cm \vee x_*}
  \Big| \frac{\PP(S_m > x)}{m c_Y (x- m \mu)^{-\alpha}} - 1 \Big| \le \varepsilon
\end{equation}
(note that \cite{Nagaev82} considers centred summands, we apply the results from
\cite{Nagaev82} to $\PP(S_m > x) = \PP( S_m - m\mu > x -m\mu )$.)

By assumption, $\overline{c}_p := \sup_{N\in\NN} N p_N < \infty$. For
$y > \overline{c}_p$ we have the elementary large deviations bound
\begin{align*}
  \PP(B_N \geq y) & \leq e^{-\lambda y} \EE\big[ e^{\lambda B_N}\big]
  \leq e^{-\lambda y} \Big( \frac{\overline{c}_p}N \big(e^\lambda -1\big) + 1 \Big)^N
  \leq \exp\Big( - \lambda y + \overline{c}_p \big(e^\lambda -1\big) \Big)
\end{align*}
which holds for all $N \in \NN, \lambda>0$. Choosing $\lambda = \log(y/\overline{c}_p) > 0$
yields
\begin{align}
  \label{eq:BinLDub}
  \PP(B_N \geq y) & \leq \exp\(- y \widetilde{I}(y)\) \quad \text{for all } N \in \NN, y > \overline{c}_p
\end{align}
where $\widetilde{I}(y) = \log(y/\overline{c}_p) - 1$.
\smallskip

Fix $\varepsilon > 0$, choose $x_*$ so that the bound in \eqref{eq:Nagaev.quantitative} holds for some $c> \mu$
and potentially enlarge $x_*$ further such that for all $x\geq x_*$ we have $x \geq c m >\mu m$ for all $m \leq \log x.$
Let $x_{**} > e^{\overline{c}_p}$ be so large that
\begin{align*}
\sup_{x \ge x_{**}} \frac{x^\alpha}{c_p c_Y} \exp\big( - \widetilde{I}(\log x) \cdot \log x \big)
< \varepsilon \\
\intertext{and}
\sup_{x \ge x_{**}} \sup_{1 \le m \le \lfloor \log x \rfloor } \Big( \frac{x}{x-m\mu }\Big)^\alpha < 1+\varepsilon .
\end{align*}

For $x \ge x_* \vee x_{**}$ we have
\begin{align*}
  \frac{\PP(X_N > x)}{c_p c_Y x^{-\alpha}} - 1 &
  = \sum_{m=0}^N \binom{N}{m} p_N^m (1-p_N)^{N-m}
  \bigg( \frac{\PP(S_m > x)}{c_p c_Y x^{-\alpha}} - 1 \bigg) \\
  & \le \sum_{m=0}^{\lfloor \log x \rfloor \wedge N} \binom{N}{m} p_N^m (1-p_N)^{N-m}
  \bigg( \frac{\PP(S_m > x)}{c_p c_Y (x-m\mu)^{-\alpha}} \frac{(x-m\mu)^{-\alpha}}{x^{-\alpha}} - 1 \bigg) \\
  & \hspace{2em}
  + \frac{x^\alpha}{c_p c_Y} \sum_{m=\lceil \log x \rceil \wedge N}^N \binom{N}{m} p_N^m (1-p_N)^{N-m} \\
  & \le \sum_{m=0}^{\lfloor \log x \rfloor \wedge N} \binom{N}{m} p_N^m (1-p_N)^{N-m}
  \Big( \frac{m}{c_p} (1+\varepsilon)^2 - 1 \Big) + \frac{x^\alpha}{c_p c_Y} \PP\big( B_N \ge \log x\big)  \\
  & \le \Big( \frac{(1+\varepsilon)^2}{c_p} \EE[B_N] - 1 \Big) +
  \frac{x^\alpha}{c_p c_Y} \exp\big( - \widetilde{I}(\log x) \cdot \log x \big).
\end{align*}
The right-hand side is smaller than $8 \varepsilon$ when $N \ge N_0$ where
$N_0$ is chosen so that
\begin{align*}
  \EE[B_N] < c_p (1+\varepsilon) \quad
  \text{for all } N \ge N_0.
\end{align*}

The lower bound can be proved analogously: it suffices to keep in the
computation above the sum over $m$ up to $\lfloor \log x \rfloor \wedge
N$ where $x \ge \max\{x_*, x_{**}, x_{\dagger} \}$ with $x_{\dagger}$ chosen so
large that $\sum_{m=0}^{\lfloor \log x_{\dagger} \rfloor \wedge N} m
\binom{N}{m} p_N^m (1-p_N)^{N-m} \ge (1-\varepsilon) \EE[B_N]$ uniformly in $N \ge N_0$, with a
suitably enlarged $N_0$.
\end{proof}
\bigskip

\begin{proof}[Proof of Lemma \ref{lemma:coinscheme}] Put $\overline{c}_{p} := \sup_{n\in\NN} N p_N < \infty$.
  $A_k^{(N)} := \sum_{i \neq k}^N A_{i,k} \sim \mathrm{Bin}(N-1,p_N)$, we have (see \eqref{eq:BinLDub})
  \begin{align*}
    \PP(A_k^{(N)} \ge \log N) \le \exp\Big( - \log N \big( \log\log(N) - \log \overline{c}_{p} - 1\big) \Big)
    = O\Big( \exp\big(-(b+1) \log N\big)\Big)
  \end{align*}
for any $b>0$. The probability on the left-hand side of
  \eqref{lemma:coinscheme.eq1} is bounded by $\sum_{k=1}^N
  \PP(A_k^{(N)} \ge \log N) = O\big( N \cdot N^{-b-1}\big)$.
  \medskip

  The probability bound in \eqref{lemma:coinscheme.eq2} is a classical large deviation bound
  for sums of Bernoulli random variables (the claimed constant is not sharp).

  {\smallskip

    (For completeness, here are some details:
    \[
    \EE[e^{\beta A^{(N)}}] = \big( p_N(e^\beta-1)+1\big)^{N(N-1)/2}
    \le \exp\Big(  p_N \frac{N(N-1)}{2} \big(e^\beta-1\big) \Big)
    = \exp\Big(\EE[A^{(N)}] \big(e^\beta-1\big) \Big)
    \]
    for any $\beta\in\RR$, hence (we parametrise $\lambda\ge 0$ in the formulas below)
    \begin{align*}
      \PP\big( A^{(N)} > (1+\varepsilon) \EE[A^{(N)}] \big) &
      \le e^{-\lambda (1+\varepsilon) \EE[A^{(N)}]} \EE[e^{\lambda A^{(N)}}]
      \le \exp\Big( \EE[A^{(N)}] \big(e^\lambda-1-\lambda \, - \, \varepsilon \lambda \big) \Big)
    \end{align*}
    and
    \begin{align*}
      \PP\big( A^{(N)} < (1-\varepsilon) \EE[A^{(N)}] \big) &
      \le e^{\lambda (1-\varepsilon) \EE[A^{(N)}]} \EE[e^{-\lambda A^{(N)}}]
      \le \exp\Big( \EE[A^{(N)}] \big(e^{-\lambda}-1+\lambda \, - \, \varepsilon \lambda \big) \Big).
    \end{align*}
    Now put $\lambda=\varepsilon$ and use that $0 \le e^\beta-1-\beta \le (2/3)\beta^2$ for $|\beta|\le 1/2$.)
  }
\end{proof}

%We can use Lemma~\ref{lemma:coinscheme} for the following:
\begin{proof}[Proof of Lemma \ref{lemma:SNbounds}]
\eqref{lemma:SNbounds.small} can be proved by classical large deviation bound arguments, e.g.\
along the lines of the proof of \cite[Lemma~5]{Schweinsberg2003}. Alternatively,
we see from \eqref{lemma:coinscheme.eq2} in Lemma~\ref{lemma:coinscheme} by conditioning
on the $A^{(N)}_{i,j}$'s in \eqref{eq:Xijrepr} that $S_N$ is indeed a sum of $\approx c_{X,1}N/2$ i.i.d.\
summands (up to an exponentially small error term) and then we can literally apply
\cite[Lemma~5]{Schweinsberg2003}.
\smallskip

For \eqref{lemma:SNbounds.large} under Assumption~\eqref{ass:Xtail} we
represent $S_N$ via \eqref{eq:Xijrepr} and
Lemma~\ref{lemma:coinscheme} as a sum of at most $(1+\varepsilon/2)
c_{X,1} N/2$ i.i.d.\ copies of $X$ (up to an exponentially small error
term) and then apply \cite[Thm.~1]{Nagaev82}.

Under Assumption~\eqref{ass:Xvariance} we have $c:= \sup_N
\EE[(X_1^{(N)})^2]<\infty$ (cf.\ e.g.\ \eqref{eq:X1fact2ndmom} below), so
$\PP(X_1^{(N)} \ge x ) \le c/x^2$ by Markov's
inequality. \cite[Thm.~2]{Nagaev82} then gives a bound of the form $C
N^{-1}$ for the probability in \eqref{lemma:SNbounds.large}.

\medskip

Finally, note that $S_N'' \mathop{=}^d S_{N-1}$ and $S_N'' \le S_N' \le S_N$.
\end{proof}

%%%%%***
\medskip
With the help of Lemma~\ref{lemma:SNbounds} (and Lemma~\ref{lemma:randomsumtails}) we can now prove Lemma \ref{lemma:SchweinsbergcN}.
For this let us recall that the factorial moments for a hypergeometric random variable $H \sim \mathrm{hypergeom}(n,m,x)$
are given by (see e.g. Formula~(39.6) in \cite{JKB97}),
\begin{equation}
\label{facmomenthyp}
  \EE\big[ (H)_k \big] = (n)_k \frac{(x)_k}{(m)_k}.
\end{equation}

\begin{proof}[Proof of Lemma \ref{lemma:SchweinsbergcN}]%[Proof sketch for Lemma~\ref{lemma:SchweinsbergcN}]
Recall that $X_1=X_1^{(N)}.$
We have by \eqref{facmomenthyp} and \eqref{lemma:SNbounds.small} of Lemma \ref{lemma:SNbounds} that
\begin{equation}
\label{c_Nsim}
c_N = \frac{\EE[(V_1)_2]}{8(N-1)}
=  \frac1{8(N-1)} N(N-1) \EE\Big[ \frac{(X_1)_2}{(S_N)_2} \ind(S_N \ge N)\Big] + O(e^{-c N}).
% \sim \frac{N}{8} \frac{\EE[(X_1)_2]}{(\mu N)^2},
\end{equation}
In the case of 1. we now observe that  for every $\varepsilon > 0$
\begin{align}
\label{EX_1_1}
  \EE\left[ \frac{(X_1)_2}{(S_N)_2} \ind(S_N \ge N)\right]
  & \leq \frac{\EE[(X_1)_2]}{((1-\varepsilon) \mu N)_2} + O(e^{-c N}), \\
\label{EX_1_2}
  \EE\left[ \frac{(X_1)_2}{(S_N)_2} \ind(S_N \ge N)\right]
  & \geq \frac{\EE[(X_1)_2 \ind(S_N \leq (1+\varepsilon)\mu N)]}{((1+\varepsilon) \mu N)_2} + O(N^{1-\alpha})
\end{align}
by \eqref{lemma:SNbounds.small} and \eqref{lemma:SNbounds.large} of Lemma~\ref{lemma:SNbounds}.

\begin{align}
  \label{eq:X1fact2ndmom}
  \EE[(X_1)_2] = \EE[X_1^2] - \EE[X_1] & =
  (N-1) \EE[X_{1,2}^2] + (N-1)(N-2) \EE[X_{1,2}]^2 - (N-1) \EE[X_{1,2}] \notag \\
  & \sim c_{X,1} \big( \EE[X^2] - \EE[X]\big) + \big(c_{X,1} \EE[X]\big)^2.
\end{align}
Furthermore,
\begin{align*}
  \EE\big[(X_1)_2 \ind(S_N > (1+\varepsilon)\mu N)\big]
  & \le \EE\big[(X_1)_2 \ind(X_1 > \varepsilon\mu N/2)\big] \\
  & \hspace{1.5em} + \EE\big[(X_1)_2 \ind(S_N-X_1 > (1+\varepsilon/2)\mu N)\big] \mathop{\longrightarrow}_{N\to\infty} 0
\end{align*}
due to the fact that by \eqref{eq:X1fact2ndmom} and \eqref{lemma:SNbounds.large} of Lemma~\ref{lemma:SNbounds}
\begin{equation*}
\EE\big[(X_1)_2 \ind(S_N-X_1 > (1+\varepsilon/2)\mu N)\big]
\leq \EE\big[(X_1)_2\big] \cdot \PP\big[ S_N -X_1 > (1+\varepsilon/2)\mu N)\big] \leq C N^{1-\alpha}
\end{equation*}
with $\alpha=2$ and so as $N \rightarrow \infty,$
\begin{equation*}
\EE[(X_1)_2 \ind(S_N \leq (1+\varepsilon)\mu N)] \sim  \EE[(X_1)_2]. %\sim c_{X,1} \big( \EE[X^2] - \EE[X]\big) + \big(c_{X,1} \EE[X]\big)^2
\end{equation*}
\smallskip
Combining this and \eqref{EX_1_1}, \eqref{EX_1_2}, and \eqref{eq:X1fact2ndmom}  with
$\mu=c_{X,1} \mu_X/2$ gives \eqref{eq:SchweinsbergcN.var}.
\medskip

\noindent
For showing 2.\  we first use \eqref{c_Nsim}, write $S_N = X_1 + (S_N-X_1)$ and use that by Lemma~\ref{lemma:SNbounds} for
$\varepsilon>0,$
\begin{equation*}
\PP\big[  (1-\varepsilon)\mu N \leq S_N-X_1 \leq  (1+\varepsilon)\mu N \big] \rightarrow 1.
\end{equation*}
We may then formally follow the argument in the proof of \cite[Lemma~13]{Schweinsberg2003} in order to obtain
\begin{align*}
c_N %= \frac1{8(N-1)} \EE\big[ (V_1)_2\big] & \sim \frac1{8(N-1)} N(N-1) \EE\Big[ \frac{(X_1)_2}{(S_N)_2} \ind(S_N \geq N)\Big] \\
\sim \frac{N}{8} \EE\Big[ \frac{X_1^2}{(X_1+ N \mu)^2}\Big].
\end{align*}
Now note that $\PP(X_1 > x)=\PP(X^{(N)}_1 > x) \sim c_{X,1} c_{X,2} x^{-\alpha}$ as
$x\to\infty$ uniformly in $N$ by Lemma~\ref{lemma:randomsumtails} and so the statement of
Lemma~\ref{Exrff:lemWdbWpMmom}  also holds %with
in the current setting implying that
\begin{align*}
c_N \sim \frac{N}{8} (N\mu)^{-\alpha} c_{X,1} c_{X,2} \alpha B(2-\alpha,\alpha).
\end{align*}
\end{proof}

\begin{appendix}
\section{The weak convergence criterion for the total offspring numbers}
\label{app:weakconvcrit}

Here, for ease of reference, we briefly recall some of the main results from M\"ohle and Sagitov\ \cite{Mohle2001}, Sagitov \cite{Sagitov2003}, Schweinsberg \cite{Schweinsberg2000}
in the notation of our Theorem~\ref{thm:res} and its assumptions. Note that our offspring numbers $\(V_1,\ldots,V_N\)$ by  (\ref{def:Vi})
are analogous to the family size in each generation for haploid Cannings models. The only difference comes from the fact that the $V_i$  here sum to $2N$ whereas theirs sum to $N$.

\begin{lemma}\label{1404148}
Assume (\ref{eq:Vfmcond}). For any $j\in\mathbb{N}$, there exists
a symmetric measure $F_j$, uniquely determined on the simplex
$\Delta_j:=\left\{\(x_1,\ldots, x_j\): x_1\geq x_2\geq\cdots\geq x_j \geq 0, \sum_{i=1}^{j}x_i\leq 1\right\}$
%$\Delta_j:=\left\{\(x_1,\ldots,x_j\)\in[0,1]^j \, | \, x_1+\cdots+x_j\leq 1\right\}$
via its moments
\begin{equation*}%\label{1404113}
\int_{\Delta_j}x_1^{k_1-2}\cdots x_j^{k_j-2}F_j\(dx_1,\ldots,dx_j\)=\phi_j\(k_1,\ldots,k_j\)\,\,\text{for\,\,}k_1, \ldots, k_j\geq 2.
\end{equation*}
\end{lemma}
\begin{proof}
We refer to the analog proof for Lemma 3.1 in \cite{Mohle2001}.
\end{proof}
\begin{lemma}
The recursive formula for $\psi_{j,s}\(k_1,\ldots,k_j\)$ by (\ref{eq:Vfmcond2}) over $s$ holds as follows:
\begin{equation}\label{recursive}
\begin{split}
\psi_{j,s+1}\(k_1,\ldots,k_j\)=&\psi_{j,s}\(k_1,\ldots,k_j\)-\sum_{i=1}^j\psi_{j,s}\(k_1,\ldots,k_{i-1},k_i+1,k_{i+1},\ldots,k_j\)\\
&-s\psi_{j+1,s-1}\(k_1,\ldots,k_j,2\).
\end{split}
\end{equation}
\end{lemma}
\begin{proof}
We refer to the analog proof for Lemma 3.3 in \cite{Mohle2001}.
%We follow Lemma 3.3 in M$\ddot{\text{o}}$hle and Sagitov \cite{Mohle2001} to show the recursive formula.
%\begin{equation*}
%\begin{split}
%&\(N-j-s\)\mathbb{E}\(\(V_1\)_{k_1}\cdots\(V_j\)_{k_j}V_{j+1}\cdots V_{j+s+1}\)\\
%=&\mathbb{E}\(\(V_1\)_{k_1}\cdots\(V_j\)_{k_j}V_{j+1}\cdots V_{j+s}\(V_{j+s+1}+V_{j+s+2}+\cdots+V_N\)\)\\
%=&\mathbb{E}\(\(V_1\)_{k_1}\cdots\(V_j\)_{k_j}V_{j+1}\cdots V_{j+s}\(2N-V_1-V_2-\cdots-V_{j+s}\)\)\\
%=&\mathbb{E}\(\(V_1\)_{k_1}\cdots\(V_j\)_{k_j}V_{j+1}\cdots V_{j+s}\(2N-\(k_1+\cdots+k_j\)-s-\sum_{i=1}^j\(V_i-k_i\)-\sum_{i=j+1}^{j+s}\(V_i-1\)\)\)\\
%=&\(2N-\(k_1+\cdots+k_j\)-s\)\mathbb{E}\(\(V_1\)_{k_1}\cdots\(V_j\)_{k_j}V_{j+1}\cdots V_{j+s}\)\\
%&-\sum_{i=1}^j\mathbb{E}\(\(V_1\)_{k_1}\cdots\(V_{i-1}\)_{k_{i-1}}\(V_i\)_{k_i+1}\(V_{i+1}\)_{k_{i+1}}\cdots\(V_j\)_{k_j}V_{j+1}\cdots V_{j+s}\)\\
%&-\sum_{i=j+1}^{j+s}\mathbb{E}\(\(V_1\)_{k_1}\cdots\(V_j\)_{k_j}V_{j+1}\cdots\(V_i\)_2\cdots V_{j+s}\).
%\end{split}
%\end{equation*}
%Divide both sides by $N^{k_1+\cdots+k_j+1-j}2^{k_1+\cdots+k_j+s+1}c_N$ and let $N\rightarrow\infty$, we get the recursive formula.
\end{proof}

\begin{proposition} Assume (\ref{eq:Vfmcond}).
\begin{enumerate}[1.]
\item For any $j\geq {\ell}$, $k_1\geq m_1$,$\ldots$, $k_{\ell}\geq m_{\ell}$, we have  $\phi_j\(k_1,\ldots,k_j\)\leq\phi_{\ell}\(m_1,\ldots,m_{\ell}\).$
\item For all $j\in\mathbb{N}$ and $k_1, \ldots, k_j\geq 2$, $\phi_j\(k_1,\ldots,k_j\)$ are uniformly bounded.
\end{enumerate}
\end{proposition}

%\iffalse
\begin{proof}
Note that these $\phi_j$'s are exactly the particular case of $\psi_{j,s}$'s when $s=0$. By recursive formula (\ref{recursive}), the monotonicity is
true for $\phi_j$'s. This result has also been proposed in the Remark on Page 39 of M$\ddot{\text{o}}$hle \cite{Mohle06}.
%Note that
%\begin{equation}\label{1404101}
%\begin{split}
%&\sum_{\begin{smallmatrix}i_1,\ldots,i_j=1\\\text{distinct}\end{smallmatrix}}^N\(V_{i_1}\)_{k_1}\ldots\(V_{i_j}\)_{k_j}\\
%\leq&\sum_{\begin{smallmatrix}i_1,\ldots,i_l=1\\\text{ %distinct}\end{smallmatrix}}^N\(V_{i_1}\)_{m_1}V_{i_1}^{k_1-m_1}\ldots\(V_{i_l}\)_{m_l}V_{i_l}^{k_l-m_l}\sum_{\begin{smallmatrix}i_{l+1},\ldots,i_j=1\\\text{distinct %from each other}\\\text{and distinct from }i_1,\ldots,i_l\end{smallmatrix}}^NV_{i_{l+1}}^{k_{l+1}}\ldots V_{i_j}^{k_j}\\
%\leq&\sum_{\begin{smallmatrix}i_1,\ldots,i_l=1\\\text{ %distinct}\end{smallmatrix}}^N\(V_{i_1}\)_{m_1}\ldots\(V_{i_l}\)_{m_l}\(2N\)^{k_1+\cdots+k_j-m_1-\cdots-m_l}\\
%=&\frac{\(2N\)^{k_1+\cdots+k_j}}{\(2N\)^{m_1+\cdots+m_l}}\sum_{\begin{smallmatrix}i_1,\ldots,i_l=1\\\text{distinct}\end{smallmatrix}}^N\(V_{i_1}\)_{m_1}\ldots\(V_{i_l}\)_{m_l}.
%\end{split}
%\end{equation}
The uniform bounded property follows from monotonicity as
%With $l=1$ and $m_1=2$, we have
\begin{equation*}
\begin{split}
\frac{\mathbb{E}\(\(V_1\)_{k_1}\cdots\(V_j\)_{k_j}\)}{c_NN^{k_1+\cdots+k_j-j}2^{k_1+\cdots+k_j}}\leq\frac{\mathbb{E}\(\(V_1\)_2\)}{4c_NN}\leq 2,
\end{split}
\end{equation*}
where we have used $c_N=\mathbb{E}\(\(V_1\)_2\)/\left[8\(N-1\)\right]$.
%It also follows from (\ref{1404101}) that
%\begin{equation*}
%\begin{split}
%\frac{\mathbb{E}\(\(V_1\)_{k_1}\cdots\(V_j\)_{k_j}\)}{c_NN^{k_1+\cdots+k_j-j}2^{k_1+\cdots+k_j}}
%\leq\frac{\mathbb{E}\(\(V_1\)_{m_1}\cdots\(V_l\)_{m_l}\)}{c_NN^{m_1+\cdots+m_l-l}2^{m_1+\cdots+m_l}},
%\end{split}
%\end{equation*}
%which entails the monotonicity
%$$\phi_j\(k_1,\ldots,k_j\)\leq\phi_l\(m_1,\ldots,m_l\)$$
%whenever $j\geq l$, $k_1\geq m_1$, $\ldots$, $k_l\geq m_l$.
\end{proof}
%\fi

%\begin{remark}\label{1404152}
%Conversely, if we have the symmetric measure $F_j$, $j\in\left\{1,2,\ldots,\right\}$ satisfying (\ref{1404113}) determined by (\ref{1404151}), and the weak %convergence of (\ref{1404112}), we can show that Condition I holds.
%\end{remark}

%\begin{definition}
%For $j\in\mathbb{N}$, $k_1$,$\ldots$,$k_j\geq 2$ and $s\in\mathbb{N}_0$, define
%$$\psi_{j,s}\(k_1,\ldots,k_j\):=\lim_{N\rightarrow\infty}\frac{\mathbb{E}\(\(V_1\)_{k_1}\cdots \(V_j\)_{k_j}V_{j+1}\cdots %V_{j+s}\)}{N^{k_1+\cdots+k_j-j}2^{k_1+\cdots+k_j+s}c_N}$$
%whenever the limit exists.
%\end{definition}

In order to represent those $\{\psi_{j,s}\,:j\in\mathbb{N},s\in\mathbb{N}_0\}$ by integrations with respect to symmetric measures
$\{F_j: j\in\mathbb{N}\}$, we define analogous polynomials
% to \cite{Mohle2001} as following:
$$T_{j,s}^{\(r\)}\(x_1,\ldots,x_r\),~~~~1\leq j\leq r, r\in\mathbb{N}\text{~and~}s\in\mathbb{N}_0$$
as
\begin{equation*}%\label{1404115}
T_{r,s}^{\(r\)}\(x_1,\ldots,x_r\)=\(1-x_1-\cdots-x_r\)^s
\end{equation*}
 and
\begin{equation*}%\label{1404114}
\begin{split}
&T_{r-j,s}^{\(r\)}\(x_1,\ldots,x_r\)
=\(-1\)^{j+1}\sum_{i_j=2j-1}^{i_{j+1}-2}\cdots\sum_{i_1=1}^{i_2-2}\prod_{k=0}^ji_k\(1-\sum_{i=1}^{r-k}x_i\)^{i_{k+1}-i_k-2},\,\,\,\,j=1,\ldots,r,
\end{split}
\end{equation*}
where $i_0=-1$ and $i_{j+1}=s+1$. Note that this implies $T^{\(r\)}_{r-j,s}\(x_1,\ldots,x_r\)=0$ for $s<2j$.
%\begin{lemma}
%The formulas (\ref{1404115}) and (\ref{1404114}) satisfy
%\begin{equation}
%T_{r,s+1}^{\(r\)}\(x_1,\ldots,x_r\)=\(1-\sum_{i=1}^rx_i\)T_{r,s}^{\(r\)}\(x_1,\ldots,x_r\)
%\end{equation}
%and for $j=1,2,\ldots,r-1$
%\begin{equation}
%T_{j,s+1}^{\(r\)}\(x_1,\ldots,x_r\)=\(1-\sum_{i=1}^jx_i\)T_{j,s}^{\(r\)}\(x_1,\ldots,x_r\)-sT_{j+1,s-1}^{\(r\)}\(x_1,\ldots,x_r\).
%\end{equation}
%\end{lemma}
%\begin{proof}
%We refer to Lemma 3.4 in \cite{Mohle2001} for the proof.
%\end{proof}

\begin{lemma}
Assume (\ref{eq:Vfmcond}). Then we have
\begin{equation}\label{eq:psi}
\psi_{j,s}\(k_1,\ldots,k_j\)=\sum_{r\geq j}\int_{\Delta_r}x_1^{k_1-2}\cdots x_j^{k_j-2}T_{j,s}^{\(r\)}\(x_1,\ldots,x_r\)F_r\(dx_1\cdots dx_r\)
\end{equation}
for all $j\in\mathbb{N}$, $k_1,\ldots,k_j\geq 2$ and all $s\in\mathbb{N}_0$.
\end{lemma}
\begin{proof}
The proof is analogous to Lemma 3.5 and Lemma 3.14 in \cite{Mohle2001}.
\end{proof}

\begin{theorem}\label{th:sibling}
Assume (\ref{eq:Vfmcond}). $\left\{\psi_{j,s}\(k_1,k_2,\ldots,k_j\):s\in\mathbb{N}_0, j\in\mathbb{N}, k_1,\ldots,k_j\geq 2\right\}$ is the collection of real numbers given by (\ref{eq:psi}). It is clearly that these $\psi_{j,s}\(k_1,k_2,\ldots,k_j\)$ are nonnegative. The sequence of measures $\(F_j\)_{j\in\mathbb{N}}$ in Lemma \ref{1404148} satisfies
\begin{enumerate}[1.]
\item each $F_j$ is concentrated on $\Delta_j=\left\{\(x_1,\ldots,x_j\):x_i\geq 0\text{~for all }i\text{ and }\sum_{i=1}^{j}x_i\leq 1\right\}$;
\item each $F_j$ is symmetric with respect to the $j$ coordinates of $\Delta_j$;
\item $F_1\(\Delta_1\)\geq F_2\(\Delta_2\)\geq\cdots\cdots$
\end{enumerate}
Then there exists an $\mathcal{E}_{\infty}$-valued coalescent process ${R}_{\infty}:=\({R}_{\infty}\(t\)\)_{t\geq 0}$ satisfying
\begin{enumerate}[1.]
\item ${R}_{\infty}\(0\)$ is the partition of $\mathbb{N}$ into singletons;
\item for each $n$, $\({R}_n\(t\)\)_{t\geq 0}$ is the restriction of $\({R}_{\infty}\(t\)\)_{t\geq 0}$ on $\{1,2,\ldots,n\}$. When ${R}_n$ has $b:=k_1+\cdots+k_j+s$ blocks, $k_1,\ldots,k_j\geq 2$, each $\(b;k_1,\ldots,k_j;s\)$-collision is occurring at rate $\psi_{j,s}\(k_1,\ldots,k_j\)$ given by (\ref{eq:psi}).
For such a simultaneous multiple coalescent process ${R}_{\infty}$, there is a finite measure $\Xi^{'}$ on the infinite simplex $\Delta$ of the form $\Xi^{'}=\Xi_0^{'}+a\delta_0$, where $\Xi_0^{'}$ has no atom at zero and $\delta_0$ is a unit mass at zero, such that
{\rm \begin{equation*}
\begin{split}
&\psi_{j,s}\(k_1,k_2,\ldots,k_j\)
=\int_{\Delta}\sum_{{\ell}=0}^{s}{{s}\choose {\ell}}\sum_{\begin{smallmatrix}i_1,\ldots,i_{j+{\ell}}\\\text{distinct}\end{smallmatrix}}x_{i_1}^{k_1}\cdots x_{i_j}^{k_j}
x_{i_{j+1}}\cdots x_{i_{j+{\ell}}}\(1-{|x|}\)^{s-{\ell}}{\frac{2\Xi^{'}_{0}\(dx\)}{\(x,x\)}+2a\ind_{\{j=1,k_1=2\}}}.
\end{split}
\end{equation*}}
\end{enumerate}
The connection between the probability measure $\Xi^{'}$ on the infinite dimensional simplex $\Delta$ and the sequence of measures $\(F_j\)_{j\in\mathbb{N}}$ is given by
{\rm \begin{equation*}%\label{eq:F_r}
F_j\(S\)=\int_{\Delta}\sum_{\begin{smallmatrix}i_1,\ldots,i_j\\\text{distinct}\end{smallmatrix}}y_{i_1}^2\cdots y_{i_j}^2{{\ind } }_{\left\{\(y_{i_1,\ldots,y_{i_r}}\)\in S\right\}}\frac{2\Xi_{0}^{'}\(dy\)}{\(y,y\)}+2a{\ind}_{\left\{r=1,\(0,\ldots,0\)\in S\right\}}
\end{equation*}}
with $S\subseteq \Delta_r$.
\end{theorem}
\begin{proof}
  We refer to Theorem 2, Proposition 8 and Proposition 11 in Schweinsberg \cite{Schweinsberg2000} as well as Theorem 2.1 in M\"ohle and Sagitov \cite{{Mohle2001}} for %the
  details. %of proof.
\end{proof}

{\noindent \it Condition I:
 The limits of
\begin{equation*}%\label{eq:Vcenconv}
\begin{split}
\lim_{N\rightarrow\infty}\frac{\mathbb{E}\(\(V_1-2\)^{k_1}\cdots\(V_j-2\)^{k_j}\)}{c_NN^{k_1+\cdots+k_j-j}2^{k_1+\cdots+k_j}}
\end{split}
\end{equation*}
exist for all $j\in\mathbb{N}$ and $k_1, \ldots, k_j\geq 2$, which is also equal to $\phi_j\(k_1,\ldots,k_j\)$}.
\smallskip

{\noindent \it Condition II:
The limits of
\begin{equation*}
\lim_{N\rightarrow\infty}\frac{\mathbb{E}\(\(V_1\)_2\cdots\(V_j\)_2\)}{c_NN^{j}2^{2j}}=F_j\(\Delta_j\)
\end{equation*}
exist for all $j\in\mathbb{N}$ and
\begin{equation*}
\lim_{N\rightarrow\infty}\frac{N^{j}}{c_N}\mathbb{P}\(V_1>2Nx_1,\ldots,V_j>2Nx_j\)=\int_{x_1}^1\cdots\int_{x_j}^1\frac{F_j\(dy_1\cdots dy_j\)}{y_1^2\cdots y_j^2}
\end{equation*}
holding for all points $\(x_1,\ldots,x_j\)$ of continuity for the measure $F_j$.}
\smallskip

For any $\epsilon>0$, let  $\Delta_{r,\epsilon}:=\Delta_r\cap\left\{x_1,\ldots,x_r>\epsilon\right\}$. Denote by $\Xi_{r,N}$ the symmetric measure on $\Delta_r$ giving the joint distribution of the vector $\({V_1}/{2N},\ldots,{V_r}/{2N}\)$ of unranked offspring frequencies. Let $\Xi_{r}$ be a symmetric measure on the $r$-dimensional simplex $\Delta_r$.
\smallskip

{\noindent \it Condition III: The weak convergence condition
\begin{equation*}
%\label{eq:c5}
\frac{1}{2c_N}N^r\Xi_{r,N}\rightarrow\Xi_r\text{~~as~}N\rightarrow\infty
\end{equation*}
over $\Delta_{r,\epsilon}$, where the connection between $\Xi_r$ and $\Xi^{'}$ is given by
{\rm\begin{equation*}%\label{1404132}
\Xi_r\(S\):=\int_{\Delta}\sum_{\begin{smallmatrix}i_1,\ldots,i_r\\\text{distinct}\end{smallmatrix}}{\ind}_{\{\(y_{i_1},\ldots,y_{i_r}\)\in S\}}\frac{\Xi^{'}\(dy\)}{\(y,y\)}
\end{equation*}}
for all $S\subseteq\Delta_r$.}

It is clear that the probability measure $\Xi^{'}$ is uniquely determined by the sequence of measures $\left(\Xi_r\right)_{r\in\mathbb{N}}$.

\begin{lemma}{\label{lem:equiv}}
Conditions (\ref{eq:PhiNconv}), (\ref{eq:Vfmcond}), I, II and III are equivalent.
\end{lemma}
\begin{proof}
(\ref{eq:Vfmcond}) $\Leftrightarrow$ Condition I and (\ref{eq:Vfmcond}) $\Leftrightarrow$ Condition II can be proved similarly as Page 1557 of \cite{Mohle2001}. Condition II $\Leftrightarrow$ Condition III and (\ref{eq:PhiNconv}) $\Leftrightarrow$ Condition III can be proved following Pages 849-852 of \cite{Sagitov2003}. Here we omit the details.
\end{proof}

%%%%%%%%%%%%%%%%%\iffalse \fi

\end{appendix}

%%%%%%%%%%%%%%%%%%%%%%%%%%%%%%%%%%%%%%%%%%%%%%%%%%%%%%%%%%%%%%%%%%%%%%%%%%%%%%%%%%%%%%%%
%\bigskip\bigskip

\end{document}